\documentclass[intlim,righttag,10pt]{amsart}
\usepackage{enumitem}
\usepackage{amscd}
\usepackage{amssymb}
\usepackage[all]{xy}
\usepackage{fge}
\usepackage{graphics}
\newcommand{\moins}{\mathbin{\fgebackslash}}

\usepackage[textsize=tiny]{todonotes}

\RequirePackage[T1]{fontenc}
\RequirePackage{amsfonts,latexsym,amssymb}

\RequirePackage{mathrsfs}
\let\mathcal\mathscr
\let\cal\mathcal

\let\bb\mathbb

\usepackage{url}
\usepackage{bbm}
\usepackage{hyperref}

\newtheorem{theorem}[equation]{Theorem}
 \newtheorem{lemma}[equation]{Lemma}
 \newtheorem{proposition}[equation]{Proposition}
 \newtheorem{corollary}[equation]{Corollary}  
\newtheorem{conjecture}[equation]{Conjecture}

\theoremstyle{definition}
\newtheorem{definition}[equation]{Definition}

\newtheorem{remark}[equation]{Remark}
\newtheorem{example}[equation]{Example}
\theoremstyle{remark}
\newtheorem*{acknowledgments}{Acknowledgments}

\def\FIL{{\mathbb F}\mathbbm{il}}

\def\invlim{\mathop{\vtop{\ialign{##\crcr$\hfill{\lim}\hfil$\crcr
\noalign{\kern1pt\nointerlineskip}\leftarrowfill\crcr\noalign
{\kern -3pt}}}}\limits}
\def\dirlim{\mathop{\vtop{\ialign{##\crcr$\hfill{\lim}\hfil$\crcr
\noalign{\kern1pt\nointerlineskip}\rightarrowfill\crcr\noalign
{\kern -3pt}}}}\limits} 
\def\lomapr#1{\smash{\mathop{\relbar\joinrel\longrightarrow}\limits^{#1}}}
 \def\verylomapr#1{\smash{\mathop{\relbar\joinrel\relbar\joinrel\relbar\joinrel\longrightarrow}\limits^{#1}}}

\def\phi{\varphi}
\def\epsilon{\varepsilon}

\newcommand{\ovk}{\overline{K} }
\newcommand{\wlim}{\operatorname{lim} }
\newcommand{\dr}{\operatorname{dR} } 
  \newcommand{\hk}{\operatorname{HK} }   
   \newcommand{\nr}{\operatorname{nr} }   
 
 \newcommand{\colim}{\operatorname{colim} }
 \newcommand{\holim}{\operatorname{holim} }
 
 \newcommand{\proeet}{\operatorname{pro\acute{e}t} } 
 \newcommand{\eet}{\operatorname{\acute{e}t} }
 \newcommand{\Spec}{\operatorname{Spec} }
\newcommand{\Spf}{\operatorname{Spf} }
 \newcommand{\Hom}{{\rm{Hom}} }
 \newcommand{\Ext}{\operatorname{Ext} }
\newcommand{\Rep}{\operatorname{Rep} }
\newcommand{\Gal}{\operatorname{Gal} }
\newcommand{\can}{ \operatorname{can} }
\newcommand{\synt}{ \operatorname{syn} }
 
\newcommand{\st}{\operatorname{st} }
 \newcommand{\kker}{\operatorname{Ker} } 
  \newcommand{\FF}{\operatorname{FF} } 
   
  \newcommand{\ccoker}{\operatorname{Coker} }  
 
 \newcommand{\im}{\operatorname{Im} }
  
 \newcommand{\sff}{{\mathcal{F}}}

 \newcommand{\sg}{{\mathcal{G}}}

 \newcommand{\sbb}{{\mathcal{B}}}
 \newcommand{\scc}{{\mathcal{C}}}

 \newcommand{\so}{{\mathcal O}}
 
 \newcommand{\se}{{\mathcal{E}}}

 \newcommand{\sx}{{\mathcal{X}}}
 
\newcommand{\sd}{{\mathcal{D}}}

\newcommand{\wotimes}{\widehat{\otimes}}
 \newcommand{\wt}{\widetilde}
 \newcommand{\wh}{\widehat}
   \numberwithin{equation}{section}

\def\R{{\mathrm R}}
\def\O{{\cal O}}
  \def\B{{\bf B}}
\def\Q{{\bf Q}} \def\Z{{\bf Z}}
\def\C{{\bf C}}
\def\N{{\bf N}}
\def\O{{\cal O}}
\def\G{{\cal G}}

\def\ainf{{\bf A}_{{\rm inf}}}
 
\def\bcris{{\bf B}_{{\rm cris}}} 
\def\bst{{\bf B}_{{\rm st}}}
\def\bdr{{\bf B}_{{\rm dR}}}
\def\Bcris{{\mathbb B}_{{\rm cris}}} 
\def\Bdr{{\mathbb B}_{{\rm dR}}}
\def\Bst{{\mathbb B}_{{\rm st}}}

\def\rg{{\rm R}\Gamma}
\def\qBC{q{\cal{BC}}}
\def\epsilon{\varepsilon}
\def\dual{{\boldsymbol *}}
\def\bmu{{\boldsymbol\mu}}

\numberwithin{equation}{section}
\begin{document}
\title[The $C_{\st}$-conjecture]
 {On the cohomology of $p$-adic analytic spaces,\\ II: The $C_{\st}$-conjecture}
 \author{Pierre Colmez} 
\address{CNRS, IMJ-PRG, Sorbonne Universit\'e, 4 place Jussieu, 75005 Paris, France}
\email{pierre.colmez@imj-prg.fr} 
\author{Wies{\l}awa Nizio{\l}}
\address{CNRS, IMJ-PRG, Sorbonne Universit\'e, 4 place Jussieu, 75005 Paris, France}
\email{wieslawa.niziol@imj-prg.fr}
 \date{\today}
\thanks{The authors' research was supported in part by the grant NR-19-CE40-0015-02 COLOSS  and the NSF grant No. DMS-1440140.}
 \begin{abstract}
Long ago, Fontaine  formulated conjectures
(now theorems) 
relating \'etale and de Rham cohomologies
of algebraic varieties over $p$-adic fields. 
In an earlier work we have shown that 
 pro-\'etale and de Rham cohomologies of analytic varieties in the
two extreme cases (proper
and Stein) are also related.  In the proper case, the comparison theorems are similar
to those for algebraic varieties, but for Stein varieties they are quite different.

In this paper, we state analogs of Fontaine's conjectures for general smooth dagger varieties, 
that interpolate between these
two extreme cases, and
we prove them
for many ``small'' varieties (cases we deal with include products of overconvergent affinoids and
proper varieties, 
analytifications of algebraic varieties, or "almost proper" dagger varieties).   
The proofs use a ``geometrization'' of all involved cohomologies in terms of 
quasi-Banach-Colmez spaces (qBC's for short, quasi- because we relax 
the finiteness conditions). The heart of the proof
relies on delicate properties of BC's and qBC's.  These properties 
should be of independent interest and we have devoted 
a large part of the paper to them.

%

 \end{abstract}

\maketitle
 \tableofcontents

\section{Introduction}
Let $\so_K$ be a complete discrete valuation ring with fraction field
$K$  of characteristic 0 and with perfect
residue field $k$, countable, of characteristic $p$.  
Let $\ovk$ be an algebraic closure of $K$, 
and let $\so_{\ovk}$ denote the integral closure of $\so_K$ in $\ovk$. 
Let $\O_C$ be the $p$-adic completion of $\so_{\ovk}$ and let $C$ be its
field of fractions
(i.e., $C$ is the completion of $\ovk$; note that $\O_C/p\simeq\so_{\ovk}/p$
is countable by our assumption on $k$).
Let $W(k)$ be the ring of Witt vectors of $k$ with 
 fraction field $F$ (i.e, $W(k)=\so_F$) and 
let $\varphi$ be the absolute
Frobenius on $W(\overline {k})$. Set $\sg_K=\Gal(\overline {K}/K)$.

\subsection{Comparison theorems for algebraic varieties}
In order to put our results into perspective, let us first recall what is known for algebraic varieties.
The story started with Tate conjecturing~\cite{Tate} the existence of a Hodge-like decomposition for
the \'etale cohomology of smooth and proper algebraic (or even rigid analytic) varieties over $K$ and
proving the existence of such a decomposition for abelian varieties with good reduction.
One upshot of Tate's results is that the $p$-adic periods of algebraic varieties do not live in~$C$.
Fontained constructed~\cite{Fo82,Fo83,Fo94b} rings $\bcris$, $\bst$, $\bdr$ that should contain
these periods and refined Tate's conjecture by conjecturing (first~\cite{Fo82,Fo83} in the case of $X$ with good
reduction and then~\cite{Fo94a}
 in the case of  $X$ with semistable reduction) the
existence of functorial period isomorphisms relating \'etale and de Rham cohomologies of smooth algebraic varieties.
\begin{conjecture} {\rm (Fontaine)}\label{CONJ1}
Let $X$ be a proper and smooth algebraic variety over $K$ admitting
a semistable model over $\so_K$. Let $i\geq 0$. 

{\rm (i)}
{\rm (Conjecture $C_{\dr}$)}
We have a functorial $\G_K$-equivariant isomorphism preserving filtrations
$$H^i_{\eet}(X_C,\Q_p)\otimes_{\Q_p}\bdr\simeq H^i_{\dr}(X)\otimes_K\bdr$$

{\rm (ii)}
{\rm (Conjecture $C_{\st}$)}
We have a functorial $\G_K$-equivariant isomorphism commuting  with $\varphi$ and~$N$
$$H^i_{\eet}(X_C,\Q_p)\otimes_{\Q_p}\bst\simeq H^i_{\hk}(X)\otimes_F\bst,$$
compatible with the de Rham period morphism, and the natural injections
$\bst{\subset} \bdr$ and $H^i_{\rm HK}{\subset} H^i_{\dr}$.
\end{conjecture}
As we explain below (Remark~\ref{INTRO2}) these conjectures are now 
 theorems (even without the assumptions on existence of nice models, properness, or even smoothness).
\begin{remark}\label{INTRO1}
{\rm (i)} All the above cohomology groups are finite dimensional over the appropriate field
($\Q_p$ for \'etale cohomology, $K$ for de Rham, and $F$ for Hyodo-Kato).

{\rm (ii)} 
The existence of the period isomorphism implies in particular that
$$\dim_{\Q_p} H^i_{\eet}(X_C,\Q_p)=\dim_K H^i_{\dr}(X).$$

{\rm (iii)} The Hyodo-Kato cohomology group in (ii) is an $F$-vector space with a semilinear 
Frobenius $\varphi$
and a monodromy operator $N$ satisfying $N\varphi=p\varphi N$. Moreover, there is a Hyodo-Kato isomorphism
$$\iota_{\rm HK}:H^i_{\rm HK}(X)\otimes_FK\overset{\sim}{\to} H^i_{\dr}(X)$$

{\rm (iv)} In $C_{\dr}$, the filtration on the left-hand side is the one coming from $\bdr$; the one on the
right-hand side is the tensor product of the Hodge filtration on $H^i_{\dr}(X)$ and the filtration
on $\bdr$. In $C_{\st}$, the $\varphi$ and $N$ on the left-hand side are those coming from $\bst$;
on the right-hand side they are the tensor products of those coming from $H^i_{\rm HK}(X)$
and $\bst$.

{\rm (v)} Galois properties of the rings $\bst$ and $\bdr$ make it possible to recover
de Rham cohomology from \'etale cohomology by taking fixed points by $\G_K$: we have 
functorial ``\'etale-to-de Rham'' isomorphisms 
\begin{align*}
&H^i_{\dr}(X)\simeq (H^i_{\eet}(X_C,\Q_p)\otimes_{\Q_p}\bdr)^{\G_K},\quad{\text{as filtered $K$-vector spaces,}}\\
&H^i_{\rm HK}(X)\simeq (H^i_{\eet}(X_C,\Q_p)\otimes_{\Q_p}\bst)^{\G_K},
\quad{\text{as $F$-vector spaces with a $\varphi$ and a $N$}}.
\end{align*}
Moreover, the Hydo-Kato isomorphism is induced by the inclusion $\bst\subset\bdr$.
Note that, instead of tensor products, we could have considered $\G_K$-equivariant homomorphisms:
we have 
functorial isomorphisms
\begin{align*}
&H^i_{\dr}(X)^\dual\simeq {\rm Hom}_{\G_K}(H^i_{\eet}(X_C,\Q_p),\bdr),\quad{\text{as filtered $K$-vector spaces,}}\\
&H^i_{\rm HK}(X)^\dual\simeq {\rm Hom}_{\G_K}(H^i_{\eet}(X_C,\Q_p),\bst),
\quad{\text{as $F$-vector spaces with a $\varphi$ and a $N$}}.
\end{align*}

{\rm (vi)} It is possible to extract $\Q_p$ from $\bst$, using the extra structures. This induces
a description of \'etale cohomology by  de Rham cohomology (with the extra structures coming from
Hydo-Kato cohomology), i.e., it gives a ``de Rham-to-\'etale'' bicartesian diagram
$$\xymatrix@R=5mm{
H^i_{\eet}(X_C,\Q_p)\ar[r]\ar[d] & (H^i_{\rm HK}(X)\otimes_F\bst)^{N=0,\varphi=1}\ar[d]\\
F^0(H^i_{\dr}(X)\otimes_K\bdr)\ar[r]& H^i_{\dr}(X)\otimes_K\bdr}$$
One can refine this diagram by taking a large enough twist, which
makes it possible to remove denominators\footnote{Recall that $\bdr=\bdr^+[\frac{1}{t}]$,
$\bst=\bst^+[\frac{1}{t}]$.} in $t$:
if $r\geq i$, we have a bicartesian diagram
$$\xymatrix@R=5mm{
H^i_{\eet}(X_C,\Q_p(r))\ar[r]\ar[d] & (H^i_{\rm HK}(X)\otimes_F\bst^+)^{N=0,\varphi=p^r}\ar[d]\\
F^r(H^i_{\dr}(X)\otimes_K\bdr^+)\ar[r]& H^i_{\dr}(X)\otimes_K\bdr^+}$$

{\rm (vii)} The pair $(H^i_{\rm HK}(X),H^i_{\dr}(X))$ is a filtered {\it $(\varphi,N)$-module}
in the sense of Fontaine; the fact that the above diagram is bicartesian and $H^i_{\eet}(X_C,\Q_p)$
is finite dimensional implies, in particular,
that this filtered $(\varphi,N)$-module is {\it weakly admissible}, a condition that can be described
purely in terms of the interplay between $\varphi$, $N$ and the filtration.
\end{remark}
\begin{remark}\label{INTRO2}
{\rm (i)} As we have mentioned above, Fontaine's conjecture is now a theorem. There have been essentially
four lines of attack: the almost \'etale approach~\cite{Fa94}, the syntomic approach~\cite{Ts},
the motivic approach~\cite{Ni08}, and the derived geometry approach~\cite{BE2}.
The resulting period isomorphisms are compared in~\cite{Ni18}.
The most comprehensive results are those of Beilinson~\cite{BE2}: there is
no assumption of properness, existence of good models or smoothness.

{\rm (ii)}
If we don't assume $X$ to have a semistable model over $\so_K$ then $H^i_{\rm HK}(X)$ has to be replaced with $H^i_{\rm HK}(X_{\ovk})$, which is not
an $F$-vector space anymore but a $F^{\rm nr}$-vector space, where $F^{\nr}$ is the maximal unramified extension of $F$. Moreover,   $H^i_{\rm HK}(X_{\ovk})$ is equipped with a semi-linear action
of $\G_K$ commuting with $\varphi$ and $N$ and the action of $\G_K$ is smooth (i.e., any element
$x$ is fixed by $\G_L$ for some finite extension $L$ of $K$ that depends on~$x$). In this case,
the isomorphisms involving $H^i_{\rm HK}(X)$ in (v) of Remark~\ref{INTRO1} involve smooth
vectors for the action of $\G_K$ and not only fixed vectors: i.e, we have
a functorial isomorphism
$$H^i_{\rm HK}(X_{\ovk})^\dual\simeq {\rm Hom}_{\G_K}^{\rm sm}(H^i_{\eet}(X_C,\Q_p),\bst)
\quad{\text{of $(\varphi, N,\G_K)$-modules over $F^{\rm nr}$}}.$$

\end{remark}

\subsection{Analytic varieties}
Interest in  analytic varieties is more recent despite the fact that Tate~\cite{Tate}
already formulated his conjecture for rigid analytic spaces.  Scholze~\cite{Sch} established
a version of Tate's original conjecture for
smooth and proper analytic spaces over $K$ or $C$
and proved the $C_{\rm dR}$-conjecture for smooth and proper analytic spaces over~$K$. We
proved the $C_{\rm st}$-conjecture for
proper analytic spaces with a semistable model  \cite{CN1} (see \cite{SG} for a simplified construction) and ``de Rham-to-pro-\'etale'' versions of the $C_{\rm st}$-conjecture
for Stein varieties~\cite{CDN3} 
over $K$ (also with a semistable model).

\smallskip
In this paper, we will consider smooth dagger varieties
over $K$ or $C$ (without any assumption on the existence of nice models): 

$\bullet$ Any proper or partially proper rigid analytic variety has a natural dagger structure, and this
includes, in particular,  analytifications of algebraic varieties, Stein varieties (e.g.~\'etale
coverings of Drinfeld's spaces in any dimension), or, more generally, holomorphically convex varieties
(proper fibrations over a Stein base).

$\bullet$ Dagger varieties that are not necessarily partially proper include overconvergent affinoids
or, more generally, quasi-compact rigid analytic varieties with an overconvergent structure.

\smallskip
As we have seen in the case of algebraic varieties, Fontaine's conjectures $C_{\dr}$ and $C_{\st}$ can
be split in two: a ``de Rham-to-\'etale'' direction, and a ``\'etale-to-de Rham'' direction.
We are going to state analogous conjectures for analytic varieties but with \'etale cohomology replaced
with pro-\'etale cohomology (i.e.,~we are dealing with ``rational'' $p$-adic Hodge theory and not
``integral'' $p$-adic Hodge theory; in particular, the pro-\'etale cohomology of the analytication
of an algebraic variety is much bigger than the \'etale cohomology of the algebraic 
variety -- the latter is finite dimensional).

\subsubsection{The ``de Rham-to-\'etale'' $C_{\dr}$ and $C_{\st}$ conjectures for dagger varieties}
\begin{conjecture}\label{CONJ2}
{\rm (de Rham-to-pro-\'etale $C_{\dr}$ + $C_{\st}$)}
Let $X$ be a smooth dagger variety over $C$. Let $i\leq r$.
We have a functorial bicartesian diagram:
$$\xymatrix@C=.4cm@R=.5cm{
H^i_{\proeet}(X,\Q_p(r))\ar[r]\ar[d]
&(H^i_{\rm HK}(X)\wotimes_{F^{\nr}}\B^+_{\st})^{N=0,\varphi=p^r}\ar[d]\\
H^i(F^r\R\Gamma_{\rm dR}(X/\B^+_{\dr}))\ar[r]
&H^i_{\rm dR}(X/\B^+_{\dr})}$$
\end{conjecture}
\begin{remark}\label{INTRO3}
(i) The $\bdr^+$-cohomology group $H^i_{\dr}(X/\bdr^+)$ is a torsion-free $\bdr^+$-module from
which one recovers the usual de Rham cohomology by moding out by $t$. If $X$ is defined over $K$ then
$H^i_{\dr}(X_C/\bdr^+)\simeq H^i_{\dr}(X)\wotimes_K\bdr^+$.

(ii) If $X$ is defined over $K$, all spaces are endowed with a natural topology
and an action of $\G_K$ and all maps
are supposed to be $\G_K$-equivariant and continuous.

(iii) As we have shown in~\cite{CN4},
if $X$ is defined over $C$, then all spaces have a natural topology
and are $C$-points of VS's (pro-\'etale sheaves of $\Q_p$-vector spaces on the category
${\rm Perf}_C$ of perfectoid spaces over $C$) and all maps are supposed to be evaluations
of morphisms of VS's and continuous.

(iv) The Hyodo-Kato cohomology group
$H^i_{\rm HK}(X)$ (see \cite[Sec. 2]{CN4}) is a $F^{\nr}$-module with a Frobenius $\varphi$,
a monodromy operator $N$, a (pro-)smooth action of $\G_K$, and a Hyodo-Kato isomorphism
$\iota_{\rm HK}:H^i_{\rm HK}(X)\wotimes_{F^{\nr}}\B^+_{\dr}\overset{\sim}{\to}H^i_{\rm dR}(X/\B^+_{\dr})$.
The definitions of $H^i_{\rm HK}(X)$ and $\iota_{\rm HK}$
are adapted from the ones of Beilinson in the algebraic setting  and use the alterations
of Hartl and Temkin to produce good local models.

(v) 
In the case of proper analytic varieties, all cohomology groups
in the diagram are finite dimensional (as was the case in the algebraic  setting)
and the kernels of the horizontal arrows are~$0$.
 This is not the case for a general analytic variety and all spaces 
have to be seen in the derived category of locally convex topological vector spaces over $\Q_p$; in particular, 
 the tensor products are (derived) completed tensor products. 
Even if $H^i_{\rm dR}(X)$, for $X$ over $K$,  is finite
dimensional, $H^i({\rm Fil}^r(\bdr^+\wotimes_K{{\rm R}\Gamma}_{\rm dR}(X)))$  
surjects onto $\Omega^i(X_C)^{d=0}$
and hence can be huge (and then so is $H^i_{\proeet}(X_C,\Q_p(r))$).
\end{remark}

\subsubsection{The ``\'etale-to-de Rham'' $C_{\dr}$ and $C_{\st}$ conjectures for dagger varieties}
In the other direction, we have the following conjectures:
\begin{conjecture}\label{CONJ3}
{\rm (pro-\'etale-to-de Rham)}
Let $X$ be a smooth dagger variety defined over $K$. We have functorial isomorphims:
\begin{align*}
&(C_{\dr})&&H^i_{\dr}(X)^\dual\simeq {\rm Hom}_{\G_K}(H^i_{\proeet}(X_C,\Q_p),\bdr),\quad{\text{as filtered $K$-vector spaces}},\\
&(C_{\st})&&H^i_{\rm HK}(X)^\dual\simeq {\rm Hom}_{\G_K}^{\rm sm}(H^i_{\proeet}(X_C,\Q_p),\bst),
\quad{\text{as $(\varphi, N,\G_K)$-modules over $F^{\rm nr}$}}.
\end{align*}
\end{conjecture}
\begin{remark}\label{INTRO4}
As we mentioned above, even if $H^i_{\dr}(X)$ is finite dimensional, $H^i_{\proeet}(X_C,\Q_p)$ is, in general,
huge. Hence it is a little bit of a miracle that one could recover $H^i_{\dr}(X)$ from it.
\end{remark}

The previous conjecture uses Galois action to recover de Rham and Hyodo-Kato cohomologies from pro-\'etale
cohomology. If $X$ is defined over $C$, there is no Galois action anymore but one can use the VS structure
alluded to above to recover part of the structure.  This leads to the following conjecture:
\begin{conjecture}\label{CONJ4}
Let $X$ be a smooth dagger variety defined over $C$. 
We have functorial isomorphisms:
\begin{align*}
&({\mathbb C}_{\dr})&&{\rm Hom}_{\bdr^+}(H^i_{\dr}(X/\bdr^+),\bdr)\simeq {\rm Hom}_{\rm VS}({\mathbb H}^i_{\proeet}(X,\Q_p),\Bdr),\quad{\text{as $\bdr$-modules}},\\
&({\mathbb C}_{\st})&&{\rm Hom}_{F^{\nr}}(H^i_{\rm HK}(X),\bst)\simeq {\rm Hom}_{\rm VS}({\mathbb H}^i_{\proeet}(X,\Q_p),\Bst),
\quad{\text{as $\bst$-modules}}.
\end{align*}
\end{conjecture}
\begin{remark}\label{INTRO5}
(i) It is not possible to recover the filtration on $\bdr^+$-cohomology
 just by looking at the cohomology level because the $t^k\Bdr^+$'s are all isomorphic
as VS's whereas the $t^k\bdr^+$'s are all distinct as $\G_K$-modules.

(ii) In the same way, $\varphi$ and $N$ disappear since 
$M=M^{N=0,\varphi=1}\otimes_{\bcris^{\varphi=1}}\bst$
if $M=M_0\otimes_{F^{\rm nr}}\bst$, where $M_0$ is a finite rank $(\varphi,N)$-module over $F^{\rm nr}$.

(iii) 
One way to understand points (i) and (ii) of the remark is the following.
Conjecture~\ref{CONJ2} represents $H^i_{\proeet}$ as the $H^0$ 
of a quasi-coherent sheaf on the Fargues-Fontaine curve that is obtained as the "modification" 
at $\infty$ of the sheaf associated to a $\varphi$-module. The $H^0$ determines the sheaf 
(because there is no $H^1$), but not
the "modification" that gave rise to it. Maybe a more sophisticated
formulation on the level of derived categories would allow to do that  (see Remark~\ref{recover})~?

\end{remark}
\subsubsection{Results supporting the conjectures}
To state our main result towards these conjectures, say that a smooth dagger variety
is {\it small} if its de Rham cohomology is finite dimensional, and that a small
(smooth dagger) variety {\it has de Rham slopes~$\geq 0$} if the slopes of the vector bundle
on the Fargues-Fontaine curve associated to the filtered $(\varphi,N)$-modules
$(H^i_{\rm HK}(X), F^0H^i(X/\bdr))$ are~$\geq 0$ for all $i$ (this includes in particular
the case when these filtered $(\varphi,N)$-modules are weakly admissible: the slopes of the
associated vector bundles are all $0$ in this case).

\begin{remark}\label{slopes}
(i) Small varieties include proper varieties, overconvergent affinoids, quasi-compact dagger varieties,
analytifications of
algebraic varieties, certain tubular neighbourhoods of subvarieties of  proper varieties or complements of such 
tubular neighbourhoods.

(ii) We conjecture that all small varieties have de Rham slopes~$\geq 0$.
Small varieties with de Rham slopes~$\geq 0$ include
proper varieties, overconvergent affinoids,  analytifications of
algebraic varieties, various complements of
tubular neighbourhoods of subvarieties of proper varieties.  Moreover, a product of two
small varieties with de Rham slopes~$\geq 0$ is again
a small variety with de Rham slopes~$\geq 0$.

(iii) It is conceivable that the above mentionned results
could be extended to show that all complements of
tubular neighbourhoods of subvarieties of proper varieties have 
de Rham slopes~$\geq 0$.  If one is very optimistic, one can also imagine
that any quasi-compact dagger variety can be written as 
the complement of a
tubular neighbourhood of a subvariety of a proper variety.
\end{remark}

\begin{theorem}\label{slope1}
If $X$ is small and has de Rham slopes~$\geq 0$, then:

$\bullet$ Conjecture~\ref{CONJ2} is true.

$\bullet$ If $X$ is defined over $K$, then Conjecture~\ref{CONJ3} is true,
and if $X$ is defined over $C$,
Conjecture~\ref{CONJ4} is true.
\end{theorem}
The VS's in the diagram in Conjecture~\ref{CONJ2}
 have some finiteness properties: they are extensions of
finite Dimensional Vector Spaces (also known as Banach-Colmez spaces, 
and referred to as BC's in the rest
of the text) by torsion $\Bdr^+$-Modules.  In particular, such objects ${\mathbb W}$
(referred to as qBC's, the ``q'' standing
for ``quasi'') have a height ${\rm ht}({\mathbb W})\in\Z$ (if ${\mathbb W}$ is the qBC attached
to a finite dimensional $\Q_p$-vector space $W$, then ${\rm ht}({\mathbb W})=\dim_{\Q_p}W$).
We have the following result echoing (ii) of Remark~\ref{INTRO1};
it is a little bit surprising
that the pro-\'etale cohomology encodes this finiteness result considering how big it is
if $X$ is not proper.
\begin{theorem}\label{slope2}
If $X$ is small with de Rham slopes~$\geq 0$, then
$${\rm ht}({\mathbb H}^i_{\proeet}(X_C,\Q_p))=\dim_C(H^i_{\dr}(X_C)).$$
\end{theorem}
\begin{remark}\label{PRINCE7}
(i) The proof of Theorem~\ref{slope1} makes heavy use of the theory of BC's and, in particular,
the geometric point of view advocated in Le Bras's thesis~\cite{lebras}.  About half of the paper
is devoted to proving results about BC's that are needed in the proof of our main result.
Some of these results may be of  independent interest.

(ii) For proper varieties, we can use the more naive theory of BC's
 as in~\cite{CN1},
where we treated the case of varieties with semistable models over the integers.

(iii) For an overconvergent affinoid or a small Stein variety, one gets a direct proof from
the basic comparison theorem proved in~\cite{CN4}.  
\end{remark}

\begin{remark}\label{PRINCE8}
If $X$ is smooth and Stein over $C$ or $K$, 
it can be written as a strict increasing union of overconvergent affinoids. 
We get that
the horizontal rows in the diagram in Conjecture~\ref{CONJ2} are surjective and their kernels are
$(\Omega^{i-1}(X)/{\rm Ker}\,d)(r-i)$.  In particular, we have an exact sequence
$$0\to \Omega^{r-1}(X)/{\rm Ker}\,d\to H^r_{\proeet}(X,\Q_p(r))\to
(H^r_{\rm HK}(X)\wotimes_{F^{\nr}}\bst^+)^{N=0,\varphi=p^r}\to 0$$
generalizing the exact sequence of~\cite[Th.\,1.8]{CDN3} which was proven under the assumption that $X$ has
a semistable model. It is not difficult to deduce Conjectures~\ref{CONJ2}, \ref{CONJ3} and~\ref{CONJ4}
in this case.
\end{remark}

\subsection{Proofs}
The main ingredients in the proofs are the results from~\cite{CN4} and the parallel theories
of BC's~\cite{CB,CF,lebras} and Fontaine's almost $C$-representations~\cite{Fo-Cp,Fonew}.

 Let $X$ be a quasi-compact smooth dagger variety over $C$. From~\cite{CN4} we use the existence of the basic comparison isomorphism
with syntomic cohomology:
\begin{equation}
\label{iso1}
H^i_{\synt}(X,\Q_p(r))\stackrel{\sim}{\to} H^i_{\proeet}(X,\Q_p(r)),\quad i\leq r,
\end{equation}
and the exact sequence
\begin{equation}
 \label{rigid1}
\cdots\to {\rm DR}_r^{i-1}(X)\to  H^i_{\synt}(X,\Q_p(r))\to {\rm HK}_r^i(X)
\lomapr{\iota_{\hk}} {\rm DR}_r^i(X)\to\cdots,
\end{equation}
where we have set:
$${\rm HK}_r^i(X):=(H^i_{\hk}(X)\wh{\otimes}_{F^{\nr}}\B^+_{\st})^{N=0,\phi=p^r},\quad
{\rm DR}_r^i(X):=H^i(\rg_{\dr}(X/\B^+_{\dr})/F^r).$$
Very important for what follows is that this exact sequence and the  isomorphism (\ref{iso1}) can be promoted to the category of VS's
(see \cite[Th.\,1.3, Th.\,1.7]{CN4}). 

\subsubsection{On the proof of Conjecture~\ref{CONJ2}}
Proving Conjecture~\ref{CONJ2} amounts to splitting the long exact sequence (\ref{rigid1})
into short exact sequences,
which can be done directly for dagger affinoids (Theorem~\ref{affinoids}), 
going back to the definition of syntomic cohomology.
To treat the case of a small variety with de Rham slopes~$\geq 0$,
we use the theory of BC's. We already used BC's in~\cite{CN1} to treat
the proper case (with a semistable model), but there we were helped by Scholze's theorem~\cite[Th.\,1.1]{Sch} that
the $H^i_{\proeet}(X,\Q_p(r))$'s are, in that case, finite dimensional $\Q_p$-vector spaces; this made
it possible to use the basic theory of BC's as developed in~\cite{CB,CF}.  In the present paper,
we need to use the more powerful point of view advocated in Le Bras' thesis~\cite{lebras}.
The key result that comes out of the theory of BC's is the following proposition
(Proposition~\ref{pos10.3}): 

\begin{proposition}\label{PRINCE5}
The following properties are equivalent:

\quad {\rm (a)} The diagram in Conjecture~\ref{CONJ2} is bicartesian.

\quad {\rm (b)} 
The slopes of the vector bundle associated
to $(H^{i}_{\rm HK}(X), F^0H^{i}_{\rm dR}(X/\B_{\dr}))$ are~$\geq 0$, for $i=r-1$ and $i=r$.

\quad {\rm (c)} The kernel and cokernel of ${\mathbb H}^r_{\proeet}(X,\Q_p(r))\to (H^r_{\rm HK}(X)\otimes_{F^{\nr}}\Bst^+)^{N=0,\varphi=p^r}$ have height~$0$.

\quad {\rm (d)} ${\rm ht}({\mathbb H}^r_{\proeet}(X,\Q_p(r)))=\dim_{C} H^r_{\rm dR}(X)$.
\end{proposition}

Among the ingredients that go into the proof of this proposition
are:

$\bullet$ an interpretation (Proposition~\ref{baco3}) of ${\rm ht}({\mathbb W})$ for a qBC ${\mathbb W}$ as
the rank of ${\rm Hom}_{\rm VS}({\mathbb W},\Bdr)$ (i.e. a {\it categorification} of the height ${\rm ht}$),

$\bullet$ results on $(\varphi,N)$-modules (Lemma~\ref{quasi121}) and BC's (Proposition~\ref{lacompagnie})
 that tell us
what happens if the sequence associated to the diagram is not exact on the right. 

It is in the application of Lemma~\ref{quasi121}
 that we use in an essential way that the degrees of involved cohomology groups are $\leq r$: 
this implies that the slopes of Frobenius on Hyodo-Kato cohomology  are $\leq r$.

\subsubsection{On the proof of Conjectures~\ref{CONJ3} and~\ref{CONJ4}}
Conjectures~\ref{CONJ3} and~\ref{CONJ4} follow from Conjecture~\ref{CONJ2} and results 
about almost $C$ representations or BC's such as:
$${\rm Hom}_{\G_K}(\bdr^+/t^k,\bdr)=0 \quad{\rm and}\quad
{\rm Hom}_{\rm VS}(\Bdr^+/t^k,\Bdr)=0$$
(Proposition~\ref{twists} for the first statement and Corollary~\ref{HN12} for the second;
this type of results allow us  to get rid of the de Rham terms in the sequence of Hom's 
that is deduced from the exact sequence associated to the 
bicartesian diagram of Conjecture~\ref{CONJ2}.) and, if $M$ is a $(\varphi,N)$-module of slopes~$\leq r$,
\begin{align*}
{\rm Hom}_{\G_K}((M\otimes_{F^{\nr}}\bst^+)^{N=0,\varphi=p^r},\bdr)&=M^\dual_K\\
{\rm Hom}_{\rm VS}((M\otimes_{F^{\nr}}\Bst^+)^{N=0,\varphi=p^r},\Bdr)&= M^\dual\otimes_{F^{\nr}}\bdr
\end{align*}
(Theorem~\ref{new-tate2} and corollaries
for the first statement and Propositions~\ref{new-baco3.5} and~\ref{baco3.7} for the second).

 \begin{acknowledgments}
W.N. would like to thank MSRI, Berkeley, and the Isaac Newton Institute, Cambridge, for hospitality during Spring 2019 and Spring 2020 semesters, respectively,  when parts of this paper were written. We would like to thank  Guido Bosco, Jean-Beno\^it Bost, Gabriel Dospinescu,   Laurent Fargues,  Marco Maculan, 
 Jer\^ome Poineau, Peter Scholze for helpful conversations concerning the content of this paper.
Special thanks are due to the referee whose comments helped us discover a serious flaw in the first version of 
this paper.
  \end{acknowledgments}

   \subsubsection{Notation and conventions.}\label{Notation}
 Let $\so_K$ be a complete discrete valuation ring with fraction field
$K$  of characteristic 0 and perfect
residue field $k$ of characteristic $p$. Let $\ovk$ be an algebraic closure of $K$ and let $\so_{\ovk}$ denote the integral closure of $\so_K$ in $\ovk$. Let $C=\wh{\ovk}$ be the $p$-adic completion of $\ovk$.  Let
$W(k)$ be the ring of Witt vectors of $k$ with 
 fraction field $F$ (i.e., $W(k)=\so_F$); let $e=e_K$ be the ramification index of $K$ over $F$.   Set $\sg_K=\Gal(\overline {K}/K)$ and 
let $\phi$ be the absolute
Frobenius on $W(\overline {k})$. 
We will denote by $\bcris, \B_{\st},\B_{\dr}$ 
the crystalline, semistable, and  de Rham period rings of Fontaine. 

 We will denote by $\so_K$,
$\so_K^{\times}$, and $\so_K^0$, depending on the context,  the scheme $\Spec ({\so_K})$ or the formal scheme $\Spf (\so_K)$ with the trivial, the canonical (i.e., associated to the closed point), and the induced by $\N\to \so_K, 1\mapsto 0$,
log-structure, respectively.  Unless otherwise stated all   formal schemes are $p$-adic, locally of finite type, and equidimensional. For a ($p$-adic formal) scheme $X$ over $\so_K$, let $X_0$ denote
the special fiber of $X$; let $X_n$ denote its reduction modulo $p^n$. 

All rigid analytic spaces considered will be over $K$ or $C$. We assume that they are separated, taut, and countable at infinity. 

   Our cohomology groups will be equipped with a canonical topology. To talk about it in a systematic way,  we will work  rationally in the category of locally convex $K$-vector spaces. For details  the reader may consult \cite[Sec. 2.1, 2.3]{CDN3}. To summarize quickly:  $C_K$ is the category of convex $K$-vector spaces; 
  it  is a quasi-abelian category.
   We will denote the left-bounded derived $\infty$-category of $C_K$ by $\sd(C_K)$.
  A morphism of complexes that is a quasi-isomorphism in $\sd(C_K)$, i.e., its cone is strictly exact,  will be called a {\em strict quasi-isomorphism}.
 The associated cohomology objects are denoted by $\wt{H}^n(E)\in {\rm LH}(C_K)$, where
${\rm LH}(C_K)$ is the left heart of $\sd(C_K)$; 
they are called {\em classical} if the natural map $\wt{H}^n(E)\to {H}^n(E)$ is an isomorphism.

   Unless otherwise stated, we work in the derived (stable) $\infty$-category $\sd(A)$ of left-bounded complexes of a quasi-abelian category $A$ (the latter will be clear from the context).
  Many of our constructions will involve (pre)sheaves of objects from $\sd(A)$. 
  We will use a shorthand for certain homotopy limits:
 if $f:C\to C'$ is a map  in the derived $\infty$-category of a quasi-abelian  category, we set
$$[\xymatrix{C\ar[r]^f&C'}]:=\holim(C\to C^{\prime}\leftarrow 0).$$ 
For an operator $F$ acting on $C$, we will use the brackets $[C ]^{F=\lambda}$ to denote the derived eigenspaces and the brackets $(C)^{F=\lambda}$ or simply $C^{F=\lambda}$  to denote the non-derived ones.

  Finally, we will use freely the notation and results from \cite{CN3}.

\section{Review of almost $C$-representations }\label{SS1}
We will briefly review Fontaine's theory of almost $C$-representations 
\cite{Fo-Cp} (see also \cite{Fonew}) and some of its consequences. 
The theory has a satisfactory shape only when $[K:\Q_p]<\infty$, but some parts
work for $K$ arbitrary. Fortunately, the almost $C$-representations that we are going
to deal with have special features and we are only going to use the results in
\S\,\ref{USED} for which we provide alternative proofs (working for general $K$).

\subsection{Notation}
A {\em banach} is a Banach space over $\Q_p$ (up to an equivalence of norms) and a {\em banach representation} of $\sg_K$ is a banach with a continuous  and linear action of $\sg_K$.
 Denote by $\sbb(\sg_K)$ the category of banach representations of $\sg_K$.
 It has a natural exact category structure: a {\em short exact sequence} in $\sbb(\sg_K)$ is a sequence
$$
0\to B_1\stackrel{f}{\to} B_2\stackrel{g}{\to} B_3\to 0,
$$
where $g$ is a strict epimorphism and $f$ is a
 kernel of $g$ (thanks to the open mapping theorem, this is equivalent to $f$, $g$ being continuous and
the above sequence to be exact algebraically).

 A {\em $\Q_p$-representation of $\sg_K$} is a finite dimensional $\Q_p$-vector space with a continuous and linear action of $\sg_K$.  Similarly, a {\em $C$-representation of $\sg_K$} is a finite dimensional $C$-vector space with a continuous and semilinear action of $\sg_K$. We will denote by $\Rep_{\Q_p}(\sg_K)$, resp. $\Rep_C(\sg_K)$,
  the category of $\Q_p$-representations, resp. $C$-representations. More generally, one can define the category $\Rep_{\B^+_{\dr}}(\sg_K)$ of finite length $\B^+_{\dr}$-representations.
  
   Fontaine proved the following surprising theorem:
 \begin{theorem}\label{surprise}{\rm (Fontaine, \cite[Th.\,$A$, Th.\,$A^{\prime}$]{Fo-Cp})}
 If $[K:\Q_p]<\infty$, the forgetful functors
 $$
 \Rep_{C}(\sg_K)\to  \sbb(\sg_K),\quad \Rep_{\B^+_{\dr}}(\sg_K)\to \sbb(\sg_K)
 $$
 are fully faithful.
  \end{theorem} 
  In other words, if $W_1, W_2$ are two $C$-representations of $\sg_K$, all $\Q_p$-linear continuous $\sg_K$-equivariant maps $W_1\to W_2$ are necessarily $C$-linear. Similarly,
   if $W_1, W_2$ are two $\B^+_{\dr}$-representations of $\sg_K$, all $\Q_p$-linear continuous $\sg_K$-equivariant maps $W_1 \to W_2$ are necessarily $\B^+_{\dr}$-linear.
\begin{remark}\label{surprise5}
The proof uses Sen's theory~\cite{sen} and gives a stronger result: one has the same statements 
for $E$-linear maps between $E$-representations 
for which the Hodge-Tate-Sen weights are algebraic over $E$
(a condition that is automatic for $E=\Q_p$, if $[K:\Q_p]<\infty$, whence the theorem).
In particular, we have the following fundamental result (where ${\rm Hom}_{\G_K}$ is
an abuse of notation for ${\rm Hom}_{\sbb(\G_K)}$), valid for arbitrary $K$:
$${\rm Hom}_{\G_K}(C,C)\simeq K$$
\end{remark} 
 \subsection{Almost $C$-representations}
\subsubsection{The general theory}
 Two banach representations $W_1$ and $W_2$ are {\em almost isomorphic} if there exist two finite dimensional $\Q_p$-vector spaces $V_i\subset W_i$, $i=1,2,$ stable under $\sg_K$ such that $W_1/V_1\simeq W_2/V_2$.
 An {\em almost $C$-representation} is a banach representation which is almost isomorphic to $C^d$ for some $d\in\N$.
 Denote by $\scc(\sg_K)$ the category of almost $C$-representations.
\begin{remark}
The above definition makes sense for arbitrary $K$, but the theory has a satisfactory shape only
for $[K:\Q_p]<\infty$;
maybe the point of view developped in \cite{Fonew} would lead to a satisfactory theory for
arbitrary $K$?
\end{remark}

From now on, assume that $[K:\Q_p]<\infty$. Then $\scc(\sg_K)$
is an abelian subcategory of the exact category $\sbb(\sg_K)$. 
If $W/V_2\simeq C^d/V_1$,
one sets $\dim (W)=d$ and  ${\rm ht}(W)=\dim_{\Q_p} (V_2)-\dim_{\Q_p}(V_1)$.
 This is independent of choices and yields additive functions \cite[Th.\,B]{Fo-Cp} -- a nontrivial fact whose proof uses the theory of BC's\footnote{More specifically, the proof uses the result which says that, 
if $S$ is an effective BC of dimension $1$ and height $h$ and if $f:S\to C$ is a morphism of BC's whose image is not finite-dimensional, then $f$ is surjective and its kernel has dimension $0$ and height $h$.} \cite{CB}, \cite{CF}.
 We have $\dim(C)=1, {\rm ht}(C)=0$ and $\dim(V)=0, {\rm ht}(V)=\dim_{\Q_p}(V)$ if $V$ is a $\Q_p$-representation. The category $\scc(\sg_K)$ contains all $\B^+_{\dr}$-representations and if $W$ is a $\B^+_{\dr}$-representation of length $d$ then it is almost isomorphic to $C^d$ \cite[Th.\,C]{Fo-Cp}; 
we have  $\dim(W)=d$ and ${\rm ht}(W)=0$. In particular, the  $\B^+_{\dr}$-representations  $C$ and $C(1)$ are almost isomorphic.
\begin{remark}
The category $\scc(\sg_K)$ modulo almost isomorphisms is semi-simple with a single isomorphism class of simple objects, the class of $C$. 
\end{remark}

 Fontaine reduced the computation of $\Ext$-groups in the category $\scc(\sg_K)$ to 
their computation in the categories $\Rep_{\B^+_{\dr}}(\sg_K)$ and
$\Rep_{\Q_p}(\sg_K)$ via the following fact (which relies on \cite[Prop.\,5.5]{Fo-Cp}
and \cite[Prop.\,5.6]{Fo-Cp}): 
 \begin{proposition}{\rm (Fontaine, \cite[Prop.\,6.4, Prop.\,6.5]{Fo-Cp})}\label{font0}
 
 {\rm (i)} Let $X,Y\in {\rm Rep}_{\Q_p}(\sg_K)$. Then we have a canonical isomorphism
 $$
 \Ext^i_{\sg_K}(X,Y)\stackrel{\sim}{\to}\Ext^i_{\scc(\sg_K)}(X,Y), \quad i\geq 0.
 $$
\hskip4mm {\rm (ii)}
 Let $X,Y\in {\rm Rep}_{\B^+_{\dr}}(\sg_K)$. Then we have a canonical isomorphism
 $$
 \Ext^i_{\B^+_{\dr}(\sg_K)}(X,Y)\stackrel{\sim}{\to}\Ext^i_{\scc(\sg_K)}(X,Y),\quad i\geq 0.
 $$
 \end{proposition}
He also proved the following result:
\begin{theorem} {\rm (Fontaine, \cite[Th.\,6.1, Prop.\,6.8, Prop.\,6.9]{Fo-Cp})}\label{font1}
Let $X,Y\in \scc({\G_K})$.

{\rm (i)} The $\Q_p$-vector spaces $\Ext^i_{\scc(\sg_K)}(X,Y)$ have finite rank and are trivial for $i\geq 3$.

{\rm (ii)} $\sum_{i=0}^2(-1)^i\dim_{\Q_p}\Ext_{\scc(\sg_K)}^i(X,Y)=-[K:\Q_p]\, {\rm ht}(X)\,{\rm ht}(Y)$.

{\rm (iii)} There exists a natural trace map $\Ext^2_{\scc(\sg_K)}(X,X(1))\to \Q_p$ and, for $0\leq i\leq2$, the map
$$
\Ext^i_{\scc(\sg_K)}(X,Y)\times  \Ext_{\scc(\sg_K)}^{2-i}(Y,X(1))\to \Ext^2_{\scc(\sg_K)}(X,X(1))\to \Q_p
$$
defines a perfect duality.
\end{theorem}
\subsubsection{$\bdr^+$-representations}
$\bdr^+$-representations are objects of $\scc(\sg_K)$ and we have a recipe for computing
${\rm Ext}$ groups between these objects. We still assume $[K:\Q_p]<\infty$,
but the results below are valid in greater generality (see~Remark\,\ref{surprise5}).

 Let $\chi$ be the cyclotomic character.  Let $K_{\infty}\subset K(\mu_{p^\infty})$ 
denote the cyclotomic $\Z_p$-extension of $K$.
Let $\gamma$ be a topological generator of $\Gal(K_{\infty}/K)$.
We choose a sequence $\{\zeta_{p^n}\}_{n\geq 1} $ of primitive $p^n$'th roots of unity $\zeta_{p^n}, n\geq 1$, such that $\zeta_{p^{n+1}}^p=\zeta_{p^n}$. Let  $t\in \B^+_{\dr}$ be the uniformizer associated to $\{\zeta_{p^n}\}_{n\geq 1} $. 
We will also use its twisted form $t^{\prime}:=t/\pi_t$ defined in \cite[Sec. 2.1]{Fo-Cp};
it is a uniformizer of $\B^+_{\dr}$ as well, 
fixed by ${\rm Gal}(\ovk/K_\infty)$ whereas $t$ is only fixed by ${\rm Gal}(\ovk/K(\mu_{p^\infty}))$. 
\begin{proposition}\label{font2}
Let   $W$ be a $\bdr^+$-representation. The groups 
${\rm Ext}^i_{\scc(\sg_K)}(C,W)$ are computed by the complex 
$$\xymatrix@C=.5cm{W_{(0)}\ar[rrr]^-{x\mapsto(t^{\prime}x,(\gamma-1)x)}&&&
W_{(1)}\oplus W_{(0)}\ar[rrrrr]^-{(x,y)\mapsto(t^{\prime}y-(\chi^{-1}(\gamma)\gamma-1)x)}&&&&& W_{(1)},}$$
where  $W_{(0)}$ is the space of generalized invariants\footnote{The elements killed by 
$(g_1-1)\cdots(g_r-1)$ for all  $g_1,\dots,g_r\in\G_K$, for $r$ big enough.} 
and $W_{(1)}=t^{\prime}((t^{\prime})^{-1}W)_{(0)}$.
\end{proposition}
\begin{proof}
By Proposition~\ref{font0},  the natural map 
\begin{equation}
\label{neg1}
{\rm Ext}^i_{\B^+_{\dr}(\sg_K)}(C,W)\to {\rm Ext}^i_{\scc(\sg_K)}(C,W)
\end{equation}
 from the $\Ext$-groups in the category of $\B^+_{\dr}$-representations  is an isomorphism. 
Our proposition is now  \cite[Th.\,2.14]{Fo-Cp}. 
\end{proof}

 We will state two simple consequences of the above proposition.
\begin{corollary}\label{font3}
If  $N\geq 0$ and  $1\leq i<j$, then 
$${\rm Hom}_{\scc(\sg_K)}(\bdr^+/t^i,t^{-N}\bdr^+/t^j)=0\quad{\rm and }\quad {\rm Ext}^1_{\scc(\sg_K)}(\bdr^+/t^i,t^{-N}\bdr^+/t^j)=0.$$
\end{corollary}
\begin{proof}
By devissage  (and twisting by  $\chi^{-j}$ for  $0\leq j\leq i-1$), we can reduce to the case $i=1$.
If  $W=t^{-N}\bdr^+/t^j$, we have  $W_{(0)}=K$ and  $W_{(1)}=Kt'$,
and our result follows from Proposition~\ref{font2} (multiplication by  $t'$
induces an isomorphism of $W_{(0)}$ with  $W_{(1)}$ and the other maps are identically zero, 
and it follows that the complex in the proposition is acyclic).
\end{proof}
\begin{lemma} 
\label{fog1}
Let  $j\in\Z, k\geq 1$. Then 
\begin{align*}
&{\rm Ext}^0_{\scc(\sg_K)}(\bdr^+/t^k,C(j))\simeq \begin{cases} 0&{\text{if $j\neq 0$}}\\ K&{\text{if $j=0,$}}\end{cases}
\\
&{\rm Ext}^1_{\scc(\sg_K)}(\bdr^+/t^k,C(j))\simeq \begin{cases} 0&{\text{if $j\neq 0,k$}}\\ K&{\text{if $j=0,k,$}}\end{cases}
\\
&{\rm Ext}^2_{\scc(\sg_K)}(\bdr^+/t^k,C(j))\simeq \begin{cases} 0&{\text{if $j\neq k$}}\\ K&{\text{if $j=k.$}}\end{cases}
\end{align*}
\end{lemma}
\begin{proof}
We will use  Proposition~\ref{font2} with  $W=(\bdr^+/t^k)(1-j)$ as well as the duality between 
${\rm Ext}^i_{\scc(\sg_K)}(C,W)$  and ${\rm Ext}^{2-i}_{\scc(\sg_K)}(W,C(1))$ from Theorem~\ref{font1}.  
We have  $W_{(0)}= Kt^{j-1}(1-j)$ if $1\leq j\leq k$, 
$W_{(0)}=0$ if  $j\leq 0$ or $j\geq k+1$, and  $W_{(1)}\simeq Kt't^{j-1}(1-j)$ if $0\leq j\leq j-1$, 
$W_{(1)}=0$ if  $j\leq -1$ or  $j\geq k$.  In the complex from 
Proposition~\ref{font2} 
computing  the $\Ext$-groups ${\rm Ext}^i_{\scc(\sg_K)}(C,W)$, the only nonzero maps
are the multiplications by  $t^{\prime}:W_{(0)}\to W_{(1)}$ which are isomorphisms unless exactly  one of the two groups is trivial
 (i.e.,  $j=0$ or $j=k$).
Our result follows. 
\end{proof}
\begin{remark}\label{fog12}
Every non-trivial 
extension of  $\bdr^+/t^k$ by $C(k)$ is isomorphic to $\bdr^+/t^{k+1}$.
\end{remark}

\begin{example}\label{firstex}{\em Extensions of Tate twists.}
We have  ${\rm Ext}^1_{\scc(\sg_K)}(C,C)\simeq K$ (the  $K$-vector space is generated by the class of 
$C\otimes(\Q_p\oplus\Q_p\log t)$).
Since ${\rm Ext}^0_{\scc(\sg_K)}(C,C)\simeq K$ by Remark~\ref{surprise5}, 
this implies that ${\rm Ext}^2_{\scc(\sg_K)}(C,C)=0$ by Theorem~\ref{font1}.
We also have, by duality, ${\rm Ext}^1_{\scc(\sg_K)}(C,C(1))\simeq K$ (generated by the class of  $\bdr^+/t^2$)
and  ${\rm Ext}^2_{\scc(\sg_K)}(C,C(1))\simeq K$, and all the other  ${\rm Ext}^i_{\scc(\sg_K)}(C,C(j))$ are trivial.
\end{example}

\subsection{Morphisms of $\bdr^+$-representations}\label{USED}
In this section, $K$ is arbitrary.  We are going to derive consequences of the following two
results which are valid for such $K$.
These are the results that we will use in the rest  of the paper.
\begin{proposition}
{\rm (i)} ${\rm Hom}_{\G_K}(C,C)\simeq K$.

{\rm(ii)} $\overline K$ is dense in $\bdr^+$.
\end{proposition}
\begin{proof}
Fontaine's proof~\cite{Fo00} of (i) works
for arbitrary $K$ (the alternative proof in~\cite{FQ}, which uses class field theory,  works
only for $[K:\Q_p]<\infty$). For (ii), see~\cite{dense}.
\end{proof}

\begin{proposition} \label{twists}
We have 

{\rm (i)} ${\rm Hom}_{\G_K}(C,C)\simeq K$, ${\rm Hom}_{\G_K}(C,C(j))=0$ if  $j\neq 0$.

{\rm (ii)} ${\rm Hom}_{\G_K}(\bdr^+,t^{-N}\bdr^+/t^k\bdr^+)\simeq K$ 
for $N\geq 0$, $k\geq 1$; compatibly  in $k$. 

{\rm (iii)} ${\rm Hom}_{\G_K}(\bdr^+,\bdr^+)\simeq K$, ${\rm Hom}_{\G_K}(\bdr,\bdr)\simeq K$.

{\rm (iv)} ${\rm Hom}_{\G_K}(t^j\bdr^+/t^k\bdr^+,\bdr)=0$, 
${\rm Hom}_{\G_K}(\bdr^+/t^k\bdr^+,\bdr^+/t^\ell\bdr^+)=0$ if  $\ell>k$.
\end{proposition}
\begin{proof}
The first claim of (i) is Fontaine's theorem.
The second claim of (i) follows from the fact that, thanks to Tate's theorem that $H^0(\sg_L,C(j))=0$
if $j\neq 0$ and $L$ is a finite extension of $K$, $C(j)$ does not have elements on which 
 the action of $\G_K$ factors through a finite quotient;
hence $\lambda\in\Hom_{\sg_K}(C,C(j))$ is identically 
zero on $\overline K$, and thus on $C$ by continuity and density of $\overline K$ in $C$.

Let us prove (ii).
Any map $\lambda:\bdr^+\to t^{-N}\bdr^+/t^k\bdr^+$ sends $\overline K$ to
$\overline K$ since elements of $\overline K$ are the smooth vectors for the action
of $\G_K$ in $\bdr^+$ and $t^{-N}\bdr^+/t^k\bdr^+$ (by Tate's theorem recalled above). 
By continuity, and density of $\overline K$
in $\bdr^+$ one sees that $\lambda(\bdr^+)\subset \bdr^+/t^k\bdr^+$.

Now, if $k=1$, $\lambda(t\overline K)$ is sent to $0$ in $\bdr^+/t=C$ since 
$C(-1)$ does not have elements on which the action of $\G_K$ factors through a finite quotient.
But $t\overline K$ is dense in $t\bdr^+$, hence $\lambda$ factors through $\bdr^+/t=C$.
It follows from (i) that $\lambda$ is the multiplication by $\kappa\in K$ (composed with
$\bdr^+\to C$).
Now, if $k\geq 1$, we can compose with $\bdr^+/t^k\to C$ to deduce that
there exist $\kappa\in K$ such that $\lambda-\kappa$ has values in $t\bdr^+/t^k\bdr^+$.
But the image of $\overline K$ by $\lambda-\kappa$ is identically $0$ for the same reasons as above;
hence $\lambda=\kappa$ and we are done.

The first statement of (iii) follows from (ii) since
$${\rm Hom}_{\G_K}(\bdr^+,\bdr^+)=\wlim_k {\rm Hom}_{\G_K}(\bdr^+,\bdr^+/t^k\bdr^+)$$
For the second statement,
let  $\lambda\in\Hom_{\sg_K}(\bdr,\bdr)$. 
By the same arguments as above, $\lambda(\overline K)\subset\overline K$; hence, by continuity, 
$\lambda(\bdr^+)\subset\bdr^+$ and the restriction of $\lambda$ to $\bdr^+$ is the multiplication
by an element $\kappa$ of $K$. But $\lambda_N$ defined by $\lambda_N(x)=t^{N}\lambda(t^{-N}x)$
also belongs to $\Hom_{\sg_K}(\bdr,\bdr)$ and is the multiplication by $\kappa$ on $t^N\bdr^+$.
It follows that $\lambda_N$ is the multiplication by $\kappa$ on $\bdr^+$, and $\lambda$ is the multiplication
by $\kappa$ on $t^{-N}\bdr^+$.   
Hence the second claim of  (iii).

Finally, for (iv) we may assume $j=0$ by twisting (if $\lambda\in {\rm Hom}_{\G_K}(t^j\bdr^+/t^k\bdr^+,\bdr)$,
then $x\mapsto t^{-j}\lambda(t^jx)\in {\rm Hom}_{\G_K}(\bdr^+/t^{k-j}\bdr^+,\bdr)$).  
Then the same arguments as for (ii) show that
$\lambda(\bdr^+/t^k)\subset\bdr^+$ for the first claim,
and that $\theta\circ\lambda$ is $\kappa\theta$ for some
$\kappa\in K$ for both claims. One deduces that $\lambda=\kappa$ on $\overline K$.
But we can find
a sequence  $(x_n)_{n\in\N}$ of elements of  $\overline K$ converging to  $t^k$
in  $\bdr^+$ hence to $0$ in  $\bdr^+/t^k$, and continuity of
 $\lambda$ implies that $0=\lambda(0)=\kappa t^k$. Hence $\kappa=0$ (since
$\ell>k$ for the second claim), and $\lambda=0$,
which finishes the proof.
\end{proof}
\section{The category $\cal{BC}$}\label{BCQBC}
In this chapter we recall the definition of BC's\footnote{Often called Banach-Colmez spaces.}, i.e., the objects of 
 the category $\cal{BC}$, and study their properties.
 In particular, we study the canonical filtration and its relation to the Harder-Narasimhan filtration.
 Moreover, we will partially categorify height of BC's and introduce a notion 
of acyclic $(\phi,N)$-modules that will play an important role later on in the paper. 
\subsection{The category $\cal{BC}$}\label{BACO} 
We will  discuss now basic properties 
of the category ${\cal BC}$ of BC's and the canonical filtration of its objects. 
\subsubsection{Definitions and basic properties} 
  Recall~\cite{CB} that a  BC\footnote{Called in \cite{CB}  "finite dimensional Banach Space".} 
${\mathbb W}$ is, 
morally, a finite dimensional $C$-vector space 
 up to  a finite dimensional $\Q_p$-vector space. 
It has a Dimension  ${\rm Dim}\,{\mathbb W}=(a,b)$, where $a=\dim {\mathbb W}\in\N$, 
the {\it dimension of ${\mathbb W}$}, is the dimension of the $C$-vector space 
and $b={\rm ht}\,{\mathbb W}\in\Z$, the {\it height of ${\mathbb W}$}, 
is the dimension of the $\Q_p$-vector space. 

\smallskip
   More precisely, a {\em Vector Space} (VS for short) 
${\mathbb W}$ is a functor $\Lambda\mapsto {\mathbb W}(\Lambda)$ from 
sympathetic algebras\footnote{In~\cite{CB}, these are defined as being
spectral and connected $C$-Banach algebras $\Lambda$ such that $x\mapsto x^p$ is surjective on 
$\{x,\ \Vert x-1\Vert_\Lambda<1\}$.  We will denote by $\so_{\Lambda}$ the unit ball in $\Lambda$. We are going to request additionally that the natural map
$\Lambda\to{\cal C}({\rm Spm}(\Lambda)\to C)$ is injective (here ${\rm Spm}(\Lambda)$
is the set of continuous morphisms $\Lambda\to C$, with pointwise convergence topology),
a property that was taken for granted in the arguments of~\cite{CB}, but which
fails for example if $\Lambda=\O_{C'}$ and $C'$ is a strictly bigger complete algebraic field
than $C$, like its spherical closure (another possibility would be to choose a big
field $\widehat C$ contaning $C$, algebraically closed and spherically complete, 
with group of valuation $\R$, and to define ${\rm Spm}(\Lambda)$
as the set of continuous morphisms $\Lambda\to \widehat C$). We are also going to restrict ourselves
to separable $\Lambda$'s (i.e.,~having a dense $C$-subspace of countable dimension), in order to be
able to use the Hahn-Banach theorem despite the fact that we are not assuming $C$ to be
spherically complete. Since we have assumed $\O_C/p$ 
to be countable, the sympathetic closure of a spectral and connected $C$-Banach algebra
that is separable is itself separable.}
to $\Q_p$-vector spaces.
By definition, a sequence $0\to {\bb W}_1\to {\bb W}\to {\bb W}_2\to 0$ is exact
if and only if $0\to {\bb W}_1(\Lambda)\to {\bb W}(\Lambda)\to {\bb W}_2(\Lambda)\to 0$
is exact for all $\Lambda$.

\smallskip
 Trivial examples of VS's are:

$\bullet$ 
finite dimensional $\Q_p$-vector spaces $V$:
 $\Lambda\mapsto V$, 

$\bullet$ ${\mathbb V}^d$, for $d\in\N$, with ${\mathbb V}^d(\Lambda)=\Lambda^d$, for
all $\Lambda$.

\noindent
More interesting examples are provided by Fontaine's rings:

$\bullet$ $\Bcris^+$, $\Bst^+$, $\Bdr^+$, $\Bcris$, $\Bst$, $\Bdr$ are naturally VS's (and even Rings).

$\bullet$ If $m\geq 1$, then ${\mathbb B}_m:=\Bdr^+/t^m\Bdr^+$ is a VS (and also a Ring).

$\bullet$ Let $h\geq 1$ and $d\in\Z$. Then ${\mathbb U}_{h,d}=(\Bcris^+)^{\varphi^h=p^d}$ if $d\geq 0$,
and ${\mathbb U}_{h,d}={\mathbb B}_{-d} /\Q_{p^h}$ if $d<0$, are VS's.

\smallskip
A Vector Space ${\mathbb W}$ is said to be finite Dimensional (a BC for short) 
 if it ``is equal to ${\mathbb V}^d$, for some $d\in\N$, up to
finite dimensional $\Q_p$-vector spaces'':
there exist finite dimensional $\Q_p$-vector spaces
$V_1,V_2$ and exact sequences
$$0\to V_1\to {\mathbb Y}\to {\mathbb V}^d\to 0,\quad
0\to V_2\to {\mathbb Y}\to {\mathbb W}\to 0,$$
so that ${\mathbb W}$ is obtained from ${\mathbb V}^d$ by ``adding $V_1$ and
moding out by $V_2$''.
Then 
$$\dim{\mathbb W}=d \quad{\rm and}\quad  {\rm ht}\,{\mathbb W}=\dim_{\Q_p}V_1-
\dim_{\Q_p}V_2$$ 
\begin{remark}
{\rm (i)} We are, in general, only interested in $W={\mathbb W}(C)$ but, without
the extra structure, it would be impossible to speak of its Dimension (for example, $C$ and $C\oplus\Q_p$ are isomorphic as topological $\Q_p$-vector spaces). 

{\rm (ii)} The functor ${\mathbb W}\mapsto {\mathbb W}(C)$
of $C$-points is faithful for BC's (this follows from Proposition~\ref{BS1}). 
Also, if ${\mathbb W}$ is a BC, then ${\mathbb W}(\Lambda)$ is a $\Q_p$-Banach, for all $\Lambda$, and
if ${\mathbb W}_1\to {\mathbb W}_2$ is a morphism of BC's, then ${\mathbb W}_1(\Lambda)\to {\mathbb W}_2(\Lambda)$
is continuous and strict,
for all $\Lambda$, and an exact sequence $0\to {\mathbb W}_1\to {\mathbb W}_2\to {\mathbb W}_3\to 0$ of BC's
induces a strictly exact sequence for all sympathetic $\Lambda$.
\end{remark}

 We quote \cite[Prop.\,5.16]{CN1}: 
\begin{proposition}\label{BS1}
{\rm (i)} The Dimension of a BC is independent
of the choices made in its definition.

{\rm (ii)} If $f:{\mathbb W}_1\to{\mathbb W}_2$ is a morphism of
BC's, then ${\rm Ker}\,f$, ${\rm Coker}\,f$ and $\im \,f$
are BC's, and
we have 
$${\rm Dim}\, {\mathbb W}_1={\rm Dim}\, {\rm Ker}\,f+{\rm Dim}\,\im \,f
\quad{\rm and}\quad
{\rm Dim}\, {\mathbb W}_2={\rm Dim}\, {\rm Coker}\,f+{\rm Dim}\,\im \,f.$$

{\rm (iii)} If $\dim{\mathbb W}=0$, then ${\rm ht}\,{\mathbb W}\geq 0$.

{\rm (iv)} If ${\mathbb W}$ has an increasing filtration such that the successive quotients
are ${\mathbb V}^1$, and if ${\mathbb W}'$ is a sub-BC of ${\mathbb W}$, then
${\rm ht}\,{\mathbb W}'\geq 0$.
\end{proposition} 

 We will denote by $\mathcal{BC}$ the category of BC's. It is an abelian category.  

\begin{example}
The Spaces ${\mathbb B}_m$ 
and ${\mathbb U}_{h,d}$ defined above are BC's. Their Dimensions are
$${\rm Dim}\,{\mathbb B}_m=(m,0),\quad
{\rm Dim}\,{\mathbb U}_{h,d}=\begin{cases}
 (d,h)&{\text{if $d\geq 0$,}}\\ (-d,-h)&{\text{if $d<0$.}}\end{cases}$$
\end{example}

\subsubsection{Canonical filtration}\label{BACO1}
In his thesis~\cite{plut}, \cite{plut1}, Pl\^ut introduced a filtration on objects of $\cal{BC}$ and
stated a number of results about this filtration.  We will review them briefly here. 
\begin{remark}
Most of these results can
be recovered from the relation of $\cal{BC}$ to vector bundles
on the Fargues-Fontaine curve and the Harder-Narasimhan filtration studied in Le Bras' thesis (see Section~\ref{HN5}).
\end{remark}
\begin{definition}\label{BACO1.1}
{\rm ({\em Curvature})}
Let ${\mathbb W}\in\cal{BC}$.  We say that
${\mathbb W}$ has:

$\bullet$ curvature~$>0$, if ${\rm Hom}({\mathbb W},{\mathbb V}^1)=0$,

$\bullet$ curvature~$\geq 0$, if  ${\rm Hom}({\mathbb W},\Bdr^+)=0$.

$\bullet$ curvature~$=0$, or is {\em affine}, if it is a successive extension of  ${\mathbb V}^1$,

$\bullet$ curvature~$<0$, if it injects into $\Bdr^d$ (or, equivalently, into $(\Bdr^+)^d$),

$\bullet$ curvature~$\leq 0$, if it injects into a  $\Bdr^+$-Module (i.e., a VS with an action of $\Bdr^+$),
\end{definition}
\begin{remark}\label{ppp1}
{\rm (i)} If  ${\mathbb W}$ has curvature~$>0$ (resp.~$\geq 0$)
and  ${\mathbb W}'$ has curvature~$\leq 0$ (resp.~$<0$),
then ${\rm Hom}_{\mathcal{BC}}({\mathbb W},{\mathbb W}')=0$. 

{\rm (ii)}
A sub-VS of an VS with curvature~$\leq 0$ (resp.~$<0$) has curvature ~$\leq 0$ (resp.~$<0$).  

{\rm (iii)}
A quotient  of an VS with curvature~$\geq 0$ (resp.~$>0$) has curvature~$\geq 0$ (resp.~$>0$). 
\end{remark}

\begin{proposition}{\rm ({The canonical filtration})}\label{ppp4}
Every ${\mathbb W}\in\cal{BC}$ has a unique filtration, called  {the canonical filtration},
$${\mathbb W}_{>0}\subset {\mathbb W}_{\geq 0}\subset {\mathbb W}$$ such that:

$\bullet$ ${\mathbb W}_{>0}$ has curvature~$>0$,

$\bullet$ ${\mathbb W}_{\geq 0}/{\mathbb W}_{>0}$ has curvature~$0$,

$\bullet$ ${\mathbb W}/{\mathbb W}_{\geq 0}$ has curvature~$<0$.
\end{proposition} 
\begin{proof}
One defines ${\mathbb W}_{>0}$ as the intersection of the kernels 
of all morphisms  $\alpha:{\mathbb W}\to {\mathbb B}_m$, for  $m\geq 1$,
and  ${\mathbb W}_{\geq 0}$  as the intersection of the kernels of all morphisms
$\alpha:{\mathbb W}\to \Bdr$.  That ${\mathbb W}_{>0}$ has curvature~$>0$ and
${\mathbb W}/{\mathbb W}_{\geq 0}$ curvature~$<0$ are clear.
That ${\mathbb W}_{\geq 0}/{\mathbb W}_{>0}$ has curvature~$0$ results from the description
of the canonical filtration in terms of the Harder-Narasimhan filtration (cf.~section~\ref{HN13}).
\end{proof}

\begin{remark}\label{ppp5}
(i) ${\mathbb W}_{\leq 0:}={\mathbb W}/{\mathbb W}_{>0}$ is the largest quotient of curvature~$\leq 0$
of~${\mathbb W}$.

(ii) ${\mathbb W}_{=0}:={\mathbb W}_{\geq 0}/{\mathbb W}_{>0}$ is the largest
affine sub-VS of  ${\mathbb W}_{\leq 0}$.
\end{remark}

\subsection{The category $\cal{BC}$ and coherent sheaves on the Fargues-Fontaine curve}\label{HN1}  
The canonical filtration of BC's is closely related to 
the Harder-Narasimhan filtration defined, using the   presentation of $\cal{BC}$
via coherent sheaves on the Fargues-Fontaine curve, by Le Bras \cite{lebras}. We will now explain this relation.
\subsubsection{The Fargues-Fontaine curve}\label{HN2}
The (algebraic) Fargues-Fontaine curve $X=X_{\rm FF}$ is the projective scheme attached to the
graded $\Q_p$-algebra $\oplus_{d\geq 0}{\rm U}_d$,
where  ${\rm U}_d=(\bcris^+)^{\varphi=p^d}$.
The closed points of $X$ are in bijection with the $\Q_p$-lines of ${\rm U}_1$;
if $x\in X$, we fix a basis $t_x$ of the corresponding line.
The field $C$ corresponds to a specific point $\infty$ of $X$, and $t=t_\infty$
is Fontaine's $p$-adic $2\pi i$.
We denote by $C_x$ the residue field at $x$; it is an algebraically close field,
complete for $v_p$,
and $C_x^\flat=C^\flat$, but $C_x$ is not necessarily~\cite{KT} isomorphic to $C$.
The residue field at $\infty$ is $C$ itself.

The completed local ring $\widehat\O_{X,x}$ at $x$ is the ring $\Bdr^+(C_x)$, and $t_x$
is a uniformizer (if $x=\infty$, then $\Bdr^+(C_x)=\bdr^+$).

\subsubsection{Harder-Narasimhan categories}
A {\em Harder-Narasimhan category} is an exact  category with two real valued functions ${\rm rk}$ and ${\rm deg}$
(rank and degree) -- which are additive in short exact sequences -- satisfying extra conditions (see \cite[5.5.1]{FF}, \cite{And}). 
This allows to define a slope function $\mu=\frac{{\rm deg}}{{\rm rk}}$ taking values
in ${\bf R}\coprod\{\pm\infty\}$ (endowed with the obvious ordering).

An object ${\cal E}$ is {\em of slope~$\lambda$} if 
$\mu({\cal E})=\lambda$. It is {\em semistable} 
if 
 $\mu({\cal E}')\leq \mu({\cal E})$
for all strict subobjects ${\cal E}'\subset {\cal E}$. It is {\em stable} 
if 
 $\mu({\cal E}') <  \mu({\cal E})$
for all strict subobjects ${\cal E}'\subset {\cal E}$.
Any object ${\cal E}$ has a canonical decreasing filtration (the {\it Harder-Narasimhan filtration})
by strict subobjects ${\cal E}^{\geq\lambda}$ (with ${\cal E}^{\geq\lambda}\subset {\cal E}^{\geq\mu}$ if
$\lambda\geq \mu$), such that, if ${\cal E}^{>\lambda}=\cup_{\mu>\lambda}{\cal E}^{\geq\mu}$,
then ${\cal E}^{\geq\lambda}/{\cal E}^{>\lambda}$ is semistable of slope~$\lambda$.

We say that ${\cal E}$ has slopes~$\geq \lambda$ (resp.~$>\lambda$) if
${\cal E}={\cal E}^{\geq\lambda}$ (resp.~${\cal E}={\cal E}^{>\lambda}$),
and has slopes $\leq \lambda$ (resp.~$<\lambda$)
if ${\cal E}^{>\lambda}=0$ (resp.~${\cal E}^{\geq\lambda}=0$).

This has a number of consequences:

$\bullet$ If ${\cal E}_1$ is of slopes~$\geq\lambda$ and ${\cal E}_2$ is of slopes~$<\lambda$,
then ${\rm Hom}({\cal E}_1,{\cal E}_2)=0$.

$\bullet$ A quotient of an object of slopes~$\geq\lambda$ has slopes~$\geq\lambda$.

$\bullet$ A subobject of an object of slopes~$\leq\lambda$ has slopes~$\leq\lambda$.

$\bullet$ Il $\lambda_1\leq\lambda_2$, if ${\cal E}_1$ is of slopes~$\geq\lambda_1$ and
${\cal E}_2$ of slopes~$\leq\lambda_2$, then an extension of ${\cal E}_2$ by ${\cal E}_1$
has slopes in $[\lambda_1,\lambda_2]$.

\subsubsection{Vector bundles}\label{HN3}
To a vector bundle
${\cal E}$ on $X$, one can attach
its rank ${\rm rk}({\cal E})\in\N$, its degree $\deg({\cal E})\in\Z$,
and its slope $\mu({\cal E})=\frac{\deg({\cal E})}{{\rm rk}({\cal E})}$.
These definitions can be extended to torsion coherent sheaves by additivity in short exact
sequences. In particular,
 torsion sheaves
have rank~$0$, degree~$>0$ and slope~$+\infty$.
Endowed with ${\rm rk}$ and  ${\rm deg}$, the category ${\rm Coh}_X$ of
coherent sheaves on $X$ is a Harder-Narasimhan category.

The following result~\cite[Th.\,8.2.10]{FF} is fundamental:
\begin{theorem} {\rm (Fargues-Fontaine)}\label{HN4}

{\rm (i)} If $\lambda=\frac{d}{h}$ {\rm(in lowest terms)}, 
there exists, up to isomorphism, a unique stable vector bundle
$\O(\lambda)$ of slope $\lambda$; its rank is $h$ and its degree $d$.

{\rm (ii)} Every vector bundle ${\cal E}$ on $X$ is a direct sum 
$${\cal E}=\O(\lambda_1)\oplus\cdots\oplus\O(\lambda_r).$$
In particular the Harder-Narasimhan filtration splits {\rm (non canonically)}.

{\rm (iii)} Every  coherent sheaf ${\cal E}$ on $X$ is a direct sum 
$${\cal E}=\O(\lambda_1)\oplus\cdots\oplus\O(\lambda_r)\oplus\big(\oplus_{x\in X}{\cal F}_x\big),$$
where ${\cal F}_x$ is a torsion coherent sheaf, supported at $x$ and zero for almost all $x$.
\end{theorem}
The $\lambda_1,\dots,\lambda_r$ above are {\it the slopes of ${\cal E}$}
 (to which one has to add $+\infty$
if one of the ${\cal F}_x$ is non zero).

We have, by~\cite[Prop.\,8.2.3]{FF},
\begin{align}\label{HN4.1}
H^0(X,\O(\lambda))=(\bcris^+)^{\varphi^h=p^d},
\quad H^1(X,\O(\lambda))=0,\quad{\text{if $\lambda\geq 0$,}}\\
H^0(X,\O(\lambda))=0,
\quad H^1(X,\O(\lambda))=\bdr^+/(t^{-d}\bdr^+\oplus\Q_{p^h}),\quad{\text{if $\lambda< 0$.}}\notag
\end{align}
Note that $H^i(X,\O(\lambda))$, for $i=0,1$, is the space of
$C$-points of a BC ${\mathbb H}^i(X,\O(\lambda))$,
and we have
\begin{align*}
{\rm Dim}({\mathbb H}^0(X,\O(\lambda)))=(d,h),\quad {\rm Dim}({\mathbb H}^1(X,\O(\lambda)))=0,
\quad{\text{if $\lambda\geq 0$,}}\\
{\rm Dim}({\mathbb H}^0(X,\O(\lambda)))=0,\quad {\rm Dim}({\mathbb H}^1(X,\O(\lambda)))=(-d,-h),
\quad{\text{if $\lambda< 0$.}}
\end{align*}
Note also that
${\mathbb H}^i(X,\O(\lambda))$, for $i=0,1$, clearly depends only on
$C^\flat$ if $\lambda\geq 0$; this is less clear when $\lambda<0$ (since $t$ depends on
$C$) but it is still true.
\begin{remark}\label{EP}
It follows for the above formulas that 
$${\rm ht}({\mathbb H}^0(X,\O(\lambda)))-
{\rm ht}({\mathbb H}^1(X,\O(\lambda)))=h,\quad{\text{for all $\lambda$.}}$$
  By additivity, this implies the following Euler-Poincar\'e formula for any vector bundle ${\cal E}$
on $X$:
$${\rm ht}({\mathbb H}^0(X,{\cal E}))-{\rm ht}({\mathbb H}^1(X,{\cal E}))={\rm rk}\,{\cal E}$$
This extends to coherent sheaves on $X$.
\end{remark}

\subsubsection{The category ${\rm Coh}_X^-$}\label{HN5}
In his thesis, Le Bras shows (cf.~\cite[Prop.\,7.8]{lebras}) that, 
if ${\bb W}$ is a BC and if ${\bb W}^{\rm sh}$ stands
for the sheaf on ${\rm Perf}_C$ associated to ${\bb W}$, then 
${\bb W}^{\rm sh}(\Lambda)={\bb W}(\Lambda)$ for any sympathetic algebra~$\Lambda$.
He also proves that the functor ${\bb W}\mapsto{\bb W}^{\rm sh}$ identifies $\cal{BC}$
as the smallest sub-abelian category of the category of pro-\'etale sheaves
on ${\rm Perf}_C$, stable
by extensions, and containing $\Q_p$ and ${\mathbb V}^1$;  this gives
an efficient alternative definition of $\cal{BC}$.

We note ${\rm Coh}_X^-$ the sub-category of $D^b({\rm Coh}_X)$
of complexes ${\cal E}_\bullet$ such that $H^i({\cal E}_\bullet)=0$ if $i\neq -1,0$,
$H^{-1}({\cal E}_\bullet)$ has slopes~$<0$ and $H^{0}({\cal E}_\bullet)$ has slopes~$\geq 0$.
Any object of ${\rm Coh}_X^-$ can be represented by a 
complex ${\cal E}_{-1}\overset{0}{\to}{\cal E}_0$ of coherent sheaves such that
$H^0(X,{\cal E}_{-1})=0$ (i.e.,~${\cal E}_{-1}$ has slopes~$<0$) and $H^1(X,{\cal E}_0)=0$
(i.e.,~${\cal E}_0$ has slopes~$\geq 0$).
Le Bras defines an exact  functor
$${\rm BC}:{\rm Coh}_X^-\to {\cal {BC}}$$
 (denoted by $R^0\tau_\dual$ in~\cite[6.2]{lebras}): for a complex  of coherent sheaves $\sff$ on $X$, ${\rm{BC}}(\sff)$ is the sheaf associated to the presheaf sending $S$ to the hypercohomology in degree $0$
of $X_S$ with values in ${\cal F}_S$.
By definition of this functor\footnote{More generally,
${\rm BC}({\cal E}_\bullet)$ is an extension of ${\mathbb H}^0(X,H^0({\cal E}^\bullet))$ by
${\mathbb H}^1(X,H^{-1}({\cal E}^\bullet))$.}, 
one gets an exact sequence in ${\cal{BC}}$, in which ${\mathbb H}^i(X,{\cal E}_j)$
denotes the BC whose $C$-points are $H^i(X,{\cal E}_i)$:
$$0\to {\mathbb H}^1(X,{\cal E}_{-1})\to {\rm BC}({\cal E}_{-1}\overset{0}{\to}{\cal E}_0)\to
{\mathbb H}^0(X,{\cal E}_{0})\to 0.$$
If ${\cal E}_{-1}\overset{0}{\to}{\cal E}_0$ 
and ${\cal F}_{-1}\overset{0}{\to}{\cal F}_0$ are objects of
${\rm Coh}_X^-$, then
$${\rm Hom}_{\rm BC}({\rm BC}({\cal E}_{-1}\overset{0}{\to}{\cal E}_0),{\rm BC}({\cal F}_{-1}\overset{0}{\to}{\cal F}_0))=
\begin{pmatrix}
{\rm Hom}({\cal E}_{-1},{\cal F}_{-1}) & {\rm Ext}^1({\cal E}_{0},{\cal F}_{-1})\\
0 & {\rm Hom}({\cal E}_{0},{\cal F}_{0})
\end{pmatrix}
$$
The following theorem is the main result of~\cite{lebras}.
\begin{theorem}\label{HN6}
{\rm (Le Bras, \cite[Th.\,1.2]{lebras})}
The functor ${\rm BC}$ realizes an equivalence of categories 
$${\rm Coh}^-_X\simeq {\cal {BC}}.$$
\end{theorem}

\subsubsection{Harder-Narasimhan filtration on BC's}\label{HN7}
One endows ${\rm Coh}_X^-$ with functions rank~${\rm rk}^-$, degree~${\rm deg}^-$
and slope $\mu^-=\frac{{\rm deg}^-}{{\rm rk}^-}$, by setting:
$${\rm rk}^-({\cal E}_{-1}\overset{0}{\to}{\cal E}_0)=\deg({\cal E}_0)-\deg({\cal E}_{-1}),
\quad
{\rm deg}^-({\cal E}_{-1}\overset{0}{\to}{\cal E}_0)={\rm rk}({\cal E}_{-1})-{\rm rk}({\cal E}_{0}),$$
which turn it into a
Harder-Narasimhan category.

By transport of structure, this endows also ${\cal{BC}}$ with functions
rank~${\rm rk}^-$, degree~${\rm deg}^-$
and slope $\mu^-$ that turn it into a Harder-Narasimhan category; we have
$${\rm rk}^-={\rm dim}\quad{\rm and}\quad \deg^-=-{\rm ht}.$$
If ${\cal F}_x$ is torsion, then $\mu^-({\rm BC}(0\to {\cal F}_x))=0$.
If ${\mathbb W}$ is a BC, we let ${\mathbb W}^{\geq\lambda}, {\mathbb W}^{>\lambda}$
denote the Harder-Narasimhan filtration, and we set ${\mathbb W}^{>-\infty}:=\cup_{\lambda\in\R}
{\mathbb W}^{\geq\lambda}$.

If $\lambda=\frac{d}{h}$ (in lowest terms),
denote by ${\mathbb U}_\lambda$ the BC defined by
${\mathbb U}_\lambda:={\mathbb U}_{h,d}$. Note that, if $\lambda=\frac{d}{h}$ is in lowest terms and $e\geq1$, then 
${\mathbb U}_{eh,ed}={\mathbb U}_{\lambda}^e$.
Then we have
$${\mathbb U}_\lambda=
\begin{cases}
{\mathbb H}^0(X,\O(\lambda)), &{\text{if $\lambda\geq 0$,}}\\
{\mathbb H}^1(X,\O(\lambda)), &{\text{if $\lambda< 0$.}}
\end{cases}
$$
Alternatively,
$${\mathbb U}_\lambda=
\begin{cases}
{\rm BC}(0\to\O(\lambda)), &{\text{if $\lambda\geq 0$,}}\\
{\rm BC}(\O(\lambda)\to0), &{\text{if $\lambda< 0$.}}
\end{cases}
$$
Then
$${\rm rk}^-({\mathbb U}_\lambda)={\rm sign}(\lambda)\,d,
\quad \deg^-({\mathbb U}_\lambda)=-{\rm sign}(\lambda)\,h,\quad
\mu^-({\mathbb U}_\lambda)=\frac{-1}{\lambda}.$$

\begin{remark}\label{HN8}
(i) 
Since
$\Q_p={\mathbb U}_0$,
we have 
$\mu^-(\Q_p)=-\infty$.

(ii) BC's are naturally diamonds (and were amongst the first non trivial examples of diamonds) and, as such,
have connected components. If ${\mathbb W}$ is a BC, then ${\mathbb W}^{>-\infty}$ is the connected component of $0$
and the quotient ${\mathbb W}^{-\infty}$ is the largest \'etale quotient (group of connected 
components, a finite dimensional $\Q_p$-vector space).

(iii) The Harder-Narasimhan filtration splits (non canonically), and every BC can be decomposed as
\begin{equation}\label{HNX9}
{\mathbb W}={\mathbb U}_{-1/\lambda_1}\oplus\cdots\oplus {\mathbb U}_{-1/\lambda_r}
\oplus\big(\oplus_x{\mathbb H}^0(X,{\cal F}_x)\big),
\end{equation}
where the $\lambda_i$ are non zero elements of $\Q\cup\{-\infty\}$,
${\mathbb U}_{-1/\lambda_i}$ is of slope~$\lambda_i$, and ${\cal F}_x$
is a torsion coherent sheaf, supported at $x$, zero for almost all $x$, and
${\mathbb H}^0(X,{\cal F}_x)$ is of slope~$0$.
The $\lambda_i$ are {\it the slopes of ${\mathbb W}$} (to which one has to add $0$ if
one of the ${\cal F}_x$ is non zero).

(iv) In the exact sequence
$$0\to {\mathbb H}^1(X,{\cal E}_{-1})\to {\rm BC}({\cal E}_{-1}\overset{0}{\to} {\cal E}_0)\to
{\mathbb H}^0(X,{\cal E}_{0})\to 0,$$
the term on the left represents the subspace of slopes~$>0$ of
${\rm BC}({\cal E}_{-1}\overset{0}{\to} {\cal E}_0)$.
\end{remark}

\begin{remark}\label{Mittag}
(i)
The existence of the exact sequence $0\to {\mathbb W}^{>-\infty}\to {\mathbb W}\to
{\mathbb W}^{-\infty}\to 0$ makes it possible to show that a decreasing sequence of BC's is stationary:
indeed, if $({\mathbb W}_n)_{n\in\N}$ is such a sequence, then $\dim({\mathbb W}_n)$ is decreasing and
bounded below by $0$, hence is constant for $n\geq N$. It follows that ${\mathbb W}_N/{\mathbb W}_n$
is of dimension~$0$ and hence is a quotient of ${\mathbb W}_N^{-\infty}$. Since
${\rm ht}({\mathbb W}_N/{\mathbb W}_n)$ is increasing
and bounded by ${\rm ht}({\mathbb W}_N^{-\infty})<\infty$, one concludes that 
${\mathbb W}_N/{\mathbb W}_n$ is constant for $n$ big enough, and that so is ${\mathbb W}_n$.

(ii) One can also use a presentation to reduce to the case ${\mathbb W}={\mathbb V}^d$, and then
argue by induction on $d$ using the fact that a sub-BC of ${\mathbb V}^1$ is either ${\mathbb V}^1$
or a finite dimensional $\Q_p$-vector space (\cite[Prop.\,7.13]{CB} ou~\cite[Prop.\,2.1]{CF}).
This proof applies verbatim to almost $C$-representations. 
\end{remark}

\begin{lemma} {\rm (\cite[Prop.\,2.4]{CF})}\label{HN17}
A sub-BC ${\mathbb W}$ of ${\mathbb V}^N$, containing no ${\mathbb V}^1$, satisfies
$\dim({\mathbb W})<{\rm ht}({\mathbb W})$ or, equivalently, has slopes~$<-1$.
\end{lemma}
\begin{proof}
We will use  the equivalence ${\rm Coh}^-_X\simeq{\cal{BC}}$.
By decomposing ${\mathbb W}$ as a direct sum of stable BC's, it suffices to prove the statement for
${\mathbb U}_\lambda$, with
$\lambda=\frac{d}{h}>0$ (here we used the fact that ${\mathbb W}$ does not contain ${\mathbb V}^1$). This  amounts to showing  that, if $f: {\mathbb U}_\lambda\to {\mathbb V}^N$ is an injective map 
 then $h>d$.  Passing to the category ${\rm Coh}^-_X$, we see that we need to show that if a map $f^{\flat}: \so(\lambda)\to \iota_{\infty,*}C^N$ is injective on global sections then $h>d$. 
But this map factors as 
$$\xymatrix@R=.4cm@C=.5cm{
\O(\lambda)\ar[d]\ar[rr]^{f^{\flat}}&& \iota_{\infty,*}C^N\\
\widehat\O_{X,\infty}\otimes\O(\lambda)\ar[r]& (\iota_{\infty,*}C)\otimes \O(\lambda)\ar[r]^-{\sim}& \iota_{\infty,*}C^h\ar[u]}$$

  Tracing this diagram from the left upper corner first vertically then horizontally we obtain a map $\overline{f}: \O(\lambda)\to  \iota_{\infty,*}C^h$, 
  which is injective on global sections. Now, passing back (from the category ${\rm Coh}^-_X$) to the category $\cal{BC}$ we get an injective  map 
  $\wt{f}: {\mathbb U}_\lambda\to {\mathbb V}^h$. 
Since ${\mathbb U}_\lambda$ is of Dimension $(d,h)$ and ${\mathbb V}^h$ of Dimension $(h,0)$,
the existence of an injection implies $h<d$, as wanted.
\end{proof}

\subsubsection{Morphisms}\label{HN14}
We can describe 
${\rm Hom}$ and  ${\rm Ext}^1$ in the category $\cal{BC}$ using the curve.
For example:
\begin{enumerate}
\item  If $D_\lambda$
is the division algebra with center $\Q_p$ and  invariant $\lambda$,
we have 
$${\rm End}_{\cal BC}({\mathbb U}_{\lambda})={\rm End}(\O(\lambda))=D_\lambda.$$

 \item Recall that,  if  $\lambda_1=\frac{d_1}{h_1}$, $\lambda_2=\frac{d_2}{h_2}$, and 
$\lambda_1+\lambda_2=\frac{d}{h}$ in minimal form, 
we have $\O(\lambda_1)\otimes\O(\lambda_2)=\O(\lambda_1+\lambda_2)^{n(\lambda_1,\lambda_2)}$,
where $n(\lambda_1,\lambda_2)=\frac{h_1}{h_2}{h}$. Also: 
$${\rm Hom}(\O(\lambda_1),\O(\lambda_2))={\rm Hom}(\O,\O(\lambda_2-\lambda_1)),
\quad
{\rm Hom}(\O,\O(\lambda))=H^0(X,\O(\lambda)).$$

%

\item If   $\lambda=\frac{d}{h}\geq 0$ then ${\rm Hom}({\mathbb U}_\lambda,{\mathbb V}^1)$ is the  $C$-module of rank~$h$
generated by  $\theta\circ\varphi^i$, for  $0\leq i\leq h-1$.
\end{enumerate}

\subsubsection{Slope~$0$ and curvature~$0$}\label{HN10}
Let $x$ be a closed point of $X$.
The map ${\cal F}\mapsto {H}^0(X,{\cal F})$ induces an equivalence of categories
from the category of torsion coherent sheaves, supported
at $x$, to the category of finite length $\Bdr^+(C_x)$-modules.
A finite length $\Bdr^+(C_x)$-module is a direct sum of ${\mathbb B}_m(C_x)=
\Bdr^+(C_x)/t_x^m$, and the sheaf $i_{x,*}{\mathbb B}_m$ corresponding to ${\mathbb B}_m(C_x)$
(where $i_x$ denotes the inclusion of $x$ in $X$) lives in an exact sequence
$$0\to \O\overset{t_x^m}{\longrightarrow}\O(m)\to i_{x,*}{\mathbb B}_m\to 0$$
Passing to the sequence of $H^0$'s (which is exact as $H^1(X,\O)=0$), this gives an isomorphism
$${\mathbb H}^0(X,i_{x,*}{\mathbb B}_m)={\mathbb U}_m/\Q_pt_x^m.$$
   One deduces from the equivalence ${\rm Coh}_X^-\simeq \cal{BC}$ that
$${\rm End}_{\cal{BC}}({\mathbb U}_m/\Q_pt_x^m)\simeq{\rm End}_{{\rm Coh}_X}(i_{x,*}{\mathbb B}_m(C_x))\simeq
{\rm End}_{\Bdr^+(C_x)}({\mathbb B}_m(C_x))\simeq
{\mathbb B}_m(C_x).$$
In particular, for $m=1$, one obtains
$${\rm End}_{\cal{BC}}({\mathbb U}_1/\Q_pt_x)\simeq C_x.$$
Note also that, if $x\neq\infty$, then
$${\rm Hom}_{\cal{BC}}({\mathbb U}_1/\Q_pt_x,{\mathbb V}^1)\simeq
{\rm Hom}_{{\rm Coh}_X}(i_{x,*}{\mathbb B}_m,i_{\infty,*}{\mathbb B}_m)=0,$$
since the two sheaves are supported at distinct points.

In the special case $x=\infty$, which will be crucial for our results, 
$t_x=t$ and ${\mathbb U}_m/\Q_pt^m={\mathbb B}_m$, and
one can describe
directly the object ${\mathbb M}$ of $\cal{BC}$ attached to a $\bdr^+$-module of finite length $M$:
we have 
$${\mathbb M}=M\otimes_{\bdr^+}\Bdr^+.$$
This can be summarized by the following result.
\begin{proposition}\label{HN11}
The functor 
$M\mapsto {\mathbb M}=M\otimes_{\bdr^+}\Bdr^+$ defines an equivalence of categories between  the category
of $\bdr^+$-modules of finite length and the subcategory of $\cal{BC}$ of objects of
curvature~$0$.
\end{proposition}
\begin{corollary}\label{HN12}
{\rm (i)}
The kernel and cokernel of a morphism of objects of curvature~$0$ are of curvature~$0$.

{\rm (ii)} If ${\mathbb W}$ is a torsion $\Bdr^+$-Module, i.e., 
a $\Bdr^+$-Module annihilated by $t^r$ for some $r\geq 1$,  
then ${\rm Hom}_{\rm VS}({\mathbb W},\Bdr^+)=0$ and ${\rm Hom}_{\rm VS}({\mathbb W},\Bdr)=0$.
\end{corollary}
\begin{proof}
Point (i) is a direct consequence of Proposition~\ref{HN11}.  To prove point (ii),
we can use Proposition~\ref{HN11}, Section~\ref{HN10}, and the decomposition \eqref{HNX9} to write
${\mathbb W}$ as $W\otimes_{\bdr^+}\Bdr^+$ for some  torsion $\bdr^+$-module $W$.
Now we have, using Proposition~\ref{HN11} and the fact that $\Bdr^+=\wlim_k\Bdr^+/t^k$,
$${\rm Hom}_{\rm VS}({\mathbb W},\Bdr^+)=\wlim_k {\rm Hom}_{\rm VS}({\mathbb W},\Bdr^+/t^k)
=\wlim_k {\rm Hom}_{\bdr^+}({W},\bdr^+/t^k)={\rm Hom}_{\bdr^+}({W},\bdr^+)=0$$
This proves the result for $\Bdr^+$. 

Now, by~\cite[\S\,7.3, Lemme~7.10]{CB}, a BC of dimension~$1$ is a quotient
of $\Q_p^r\oplus{\bb L}_\ell$ where ${\bb L}_\ell$ is the Graph of an additive element~$\ell$.
It follows that a BC ${\bb W}$ of dimension $d$ is a quotient of
$\Q_p^r\oplus{\bb L}_{\ell_1}\oplus\cdots\oplus{\bb L}_{\ell_d}$ (one can also use universal covers
of $p$-divisible groups instead of Graphs of additive elements~\cite{lebras}),
and that the image of 
$\alpha:{\mathbb W}\to\Bdr$ is generated by the images\footnote{That is, the
image of ${\bb W}(\Lambda)$ is generated by $\alpha(\Q_p^r)$ and $\tau(\alpha(\ell_i))-\alpha(\ell_i)$
for $1\leq i\leq d$ and $\tau\in \widetilde T_\Lambda$ in the notations of~\cite[\S\,7.3]{CB}.}
 of $\Q_p^r$ and $\ell_1,\dots\ell_d$, 
hence factors through $t^{-N}\Bdr^+$ for some $N$.
This makes it possible to use the case of $\Bdr^+$ to finish the proof.
\end{proof}

\subsubsection{Canonical and Harder-Narasimhan filtrations}\label{HN13}
Let ${\mathbb W}$ be a BC.
The relation between the decomposition (\ref{HNX9}) and the filtration of Proposition~\ref{ppp4} is given by:
\begin{align*}
{\mathbb W}_{>0}&\simeq \big(\oplus_{\lambda_i>0}{\mathbb U}_{-1/\lambda_i}\big)
\oplus\big(\oplus_{x\neq\infty} {\mathbb H}^0(X,{\cal F}_x)\big),\\
{\mathbb W}_{\leq 0}&\simeq
\big(\oplus_{\lambda_i<0}{\mathbb U}_{-1/\lambda_i}\big)\oplus {\mathbb H}^0(X,{\cal F}_\infty)
={\mathbb H}^0\big(X,{\cal F}_\infty\oplus\big(\oplus_{\lambda_i<0}\O(-1/\lambda_i)\big)\big),\\
{\mathbb W}_{<0}&\simeq \oplus_{\lambda_i<0}{\mathbb U}_{-1/\lambda_i},\quad {\mathbb W}_{=0}\simeq {\mathbb H}^0(X,{\cal F}_\infty).
\end{align*}
In particular, ${\bb W}$ is of curvature~$<0$ if and only if its Harder-Narasimhan slopes
are $<0$, and if its slopes are~$>0$ then ${\bb W}$ is of curvature~$>0$.

From the properties of Harder-Narasimhan filtrations and the above decompositions,
we can deduce the following results.

\begin{corollary}\label{ppp2.1}
{\rm (i)}
${\mathbb W}$ is of curvature~$<0$ {\rm(}resp.~$\leq 0${\rm )} if and only if ${\mathbb W}\simeq H^0(X,{\cal E})$,
where ${\cal E}$ is a vector bundle of slopes~$\geq 0$ {\rm(}resp.~the sum of a vector
bundle of slopes~$\geq 0$ and a torsion sheaf supported at~$\infty${\rm )}.

{\rm (ii)} An extension of two BC's of curvature~$<0$ {\rm(}resp.~$\leq 0${\rm )} is
of curvature~$<0$ {\rm(}resp.~$\leq 0${\rm )}.
\end{corollary}

\begin{corollary}\label{ppp2.2}
{\rm The sign of the curvature determines the sign of the height}:

{\rm (a)}
curvature~$0$ implies height~$0$;

{\rm (b)} curvature~$<0$ implies height~$>0$; 

{\rm (c)} curvature~$> 0$ implies\footnote{If $x\neq\infty$, then ${\rm ht}({\bb U}_1/\Q_pt_x)=0$,
but the curvature of ${\bb U}_1/\Q_pt_x$ is $>0$, which shows that we can't replace
``height~$\leq 0$'' by ``height~$<0$''.} height~$\leq 0$.
\end{corollary}

\begin{corollary}\label{ppp2.3}
 {\rm The curvature decreases by going to a subobject and increases by taking a quotient:}

{\rm (i)} A sub-BC of a BC of curvature~$\leq 0$ {\rm(}resp.~$<0${\rm )} has curvature~$\leq 0$ {\rm(}resp.~$<0${\rm )}.

{\rm (ii)}
A sub-BC of height~$0$ of a BC of curvature~$\leq 0$ has curvature~$0$.

{\rm (iii)}
A quotient of a BC of curvature~$\geq 0$ {\rm(}resp.~$>0${\rm )} has curvature~$\geq 0$ {\rm(}resp.~$>0${\rm )}.

{\rm (iv)}
A quotient of height~$0$ of a BC of curvature~$\geq 0$ has curvature~$0$.

\end{corollary}

\begin{remark}\label{ppp2.7}
An important consequence of (ii) of Corollary~\ref{ppp2.3} is that a sub-BC ${\mathbb U}$ of
a torsion $\Bdr^+$-Module ${\mathbb W}$ satisfies ${\rm ht}({\mathbb U})\geq 0$ and
${\mathbb U}$ is itself a torsion $\Bdr^+$-Module if and only if ${\rm ht}({\mathbb U})= 0$.

This can be proven, without the Harder-Narasimhan decomposition, by induction on the length
of ${\mathbb W}$, using the fact that a sub-BC of ${\mathbb V}^1$ is either ${\mathbb V}^1$
or a finite dimensional $\Q_p$-vector space and the fact that an extension
of $\Bdr^+$-Modules is itself a $\Bdr^+$-Module. This proof extends verbatim
to almost $C$-representations thanks to Proposition~\ref{font0}.
\end{remark}

\begin{proposition}\label{lacompagnie}
Let $f:{\bb W}_1\to {\bb W}_2$ be a morphism of BC's with ${\bb W}_2$ a $\Bdr^+$-Module.
Assume that ${\rm Im}(f)$ generates ${\bb W}_2$ as a $\Bdr^+$-Module. Then, if ${\rm Coker}(f)\neq 0$,
it is of curvature~$>0$ and height~$<0$.
\end{proposition}
\begin{proof}
We can assume $f$ to be injective (mod out ${\bb W}_1$ by ${\rm Ker}(f)$ if this is not the case),
and not surjective (otherwise there is nothing to prove).
Then ${\bb W}_1={\bb H}^0(X,{\cal F}_1)$ and ${\bb W}_2={\bb H}^0(X,{\cal F}_2)$ where ${\cal F}_1,{\cal F}_2$
are coherent sheaves on $X_{\rm FF}$ with vanishing $H^1$, 
${\cal F}_2$ is supported at $\infty$,
and $f$ is induced from a morphism $f_X:{\cal F}_1\to{\cal F}_2$ of coherent sheaves.
The hypothesis that ${\rm Im}(f)$ generates ${\bb W}_2$ implies that $f_X$ is
surjective and the injectivity of $f$ implies that ${\bb H}^0(X,{\rm Ker}(f_X))=0$.  Vanishing of
${\bb H}^1(X,{\cal F}_1)$ gives an isomorphism ${\rm Coker}(f)\simeq {\bb H}^1(X,{\rm Ker}(f_X))$.
Now, ${\cal F}_1$ is not torsion (otherwise it would be supported at $\infty$ since
its ${\bb H}^0$ has curvature~$\leq 0$, and ${\bb W}_1$ would be a $\Bdr^+$-module, and $f$ would be
surjective which we assumed to not be the case); hence ${\rm Ker}(f_X)$ is not torsion.
Since its $H^0$ is $0$, its Harder-Narasimhan slopes are~$<0$,
 which implies that ${\rm Coker}(f)$ is of curvature~$>0$ and height~$<0$, as desired.
\end{proof}

\subsubsection{BC's of curvature~$\leq 0$}\label{BACO2}
\begin{lemma}\label{baco2.1}
The following conditions are equivalent: 

{\rm (i)}
${\mathbb W}$ is of curvature~$\leq 0$.

{\rm (ii)} There is an exact sequence
\begin{equation}
\label{kicia1}
0\to V\to {\mathbb W}\to {\mathbb M}\to 0,
\end{equation}
where ${\mathbb M}$ is of curvature~$0$
and  $V$ is finite dimensional over  $\Q_p$.
\end{lemma}
\begin{proof}Implication 
 (i)$\Rightarrow$(ii) 
follows from the fact that, if ${\mathbb W}$ is of curvature~$\leq 0$, then
$${\mathbb W}={\mathbb H}^0(X,{\cal F}_\infty)\oplus\big(\oplus_{d_i/h_i\geq 0}{\mathbb U}_{d_i/h_i}\big)$$
and we have an exact sequence $0\to \Q_{p^{h_i}}\to {\mathbb U}_{d_i/h_i}\to {\mathbb B}_{d_i}\to 0$.
Then $V=\oplus_{d_i/h_i}\Q_{p^{h_i}}$ gives what we want.

The converse implication (ii)$\Rightarrow$(i) follows from point (ii) of Corollary~\ref{ppp2.1}.
\end{proof}

\subsection{Categorification of height} 
We will introduce now a partial categorification of height of BC's. 
If  ${\mathbb W}$ is a BC, set  
$$h({\mathbb W}):={\rm Hom}_{\rm VS}({\mathbb W},\Bdr)$$
This is a $\bdr$-module; hence $h(-)$ is a functor 
from the category $\cal{BC}$ to the category of vector spaces over $\bdr$. 
Thanks to the arguments in the proof of Corollary~\ref{HN12},
we also have
$$h({\mathbb W})={\rm Hom}_{\rm VS}({\mathbb W},\Bdr)
\simeq \colim_{k\geq 0}{\rm Hom}_{\rm VS}({\mathbb W},t^{-k}\Bdr^+)
={\rm Hom}_{\rm VS}({\mathbb W},\Bdr^+)\otimes_{\bdr^+}\bdr$$
We recall (see \cite{ALB}) that we also have 
$${\rm Ext}^{1}_{\rm VS}({\mathbb W},\Bdr)\simeq 
\colim_{k\geq 0}{\rm Ext}^1_{\rm VS}({\mathbb W},t^{-k}\Bdr^+).
$$
This is a $\bdr$-module as well.

\begin{lemma}\label{baco5}
{\rm (i)} If  ${\mathbb W}$ is of curvature~$\leq 0$,
then ${\rm Ext}^{1}_{\rm VS}({\mathbb W},\Bdr)=0$.

{\rm (ii)}
If  $0\to {\mathbb W}_1\to {\mathbb W}\to {\mathbb W}_2\to 0$ is an exact sequence of BC's
of curvature~$\leq 0$,
the sequence $0\to h({\mathbb W}_2)\to h({\mathbb W})\to h({\mathbb W}_1)\to 0$ is exact.
\end{lemma}
\begin{proof}
Claim  (ii) follows immediately from (i).   
 
Now,
the exact sequence (\ref{kicia1}) shows that it suffices to prove (i) 
for an affine and for $\Q_p$.
 That ${\rm Ext}^1_{\rm VS}(\Q_p,\Bdr^+)=0$  follows easily, by devissage, from the fact that the maps $\Bdr^+/t^{m+1}\to \Bdr^+/t^m$ are surjective and 
${\rm Ext}^1_{\rm VS}(\Q_p,{\mathbb V}^1)=0$ (see \cite[Th.\,4.1]{lebras}).  This implies  that $$
{\rm Ext}^{1}_{\rm VS}(\Q_p,\Bdr)\simeq\colim_{k\geq 0}{\rm Ext}^1_{\rm VS}(\Q_p,t^{-k}\Bdr^+)=0,
$$
as wanted.

 To show that ${\rm Ext}^{1}_{\rm VS}({\mathbb W},\Bdr)=0$ for ${\mathbb W}$ affine, it suffices,
again by devissage, to show that ${\rm Ext}^{1}_{\rm VS}({\mathbb V}^1,\Bdr)=0$. 
   To show the latter fact we use 
the exact sequence 
$$
0\to t^{-k}\Bdr^+
\to t^{-k-1}\Bdr^+\to t^{-k-1}{\mathbb V}^1\to 0
$$
and the fact that ${\rm Hom}_{\rm VS}({\mathbb V}^1, \Bdr^+)=0$. This gives us the exact sequence 
$$0\to {\rm Hom}_{\rm VS}({\mathbb V}^1,t^{-k-1}{\mathbb V}^1)\to {\rm Ext}^1_{\rm VS}({\mathbb V}^1,t^{-k}\Bdr^+)
\to {\rm Ext}^1_{\rm VS}({\mathbb V}^1,t^{-k-1}\Bdr^+).$$ 
Since the first two terms are isomorphic to 
$\bdr^+/t$ as $\bdr^+$-modules, 
${\rm Ext}^1_{\rm VS}({\mathbb V}^1,t^{-k}\Bdr^+)
\to {\rm Ext}^1_{\rm VS}({\mathbb V}^1,t^{-k-1}\Bdr^+)$ is  zero.
This finishes the proof.
\end{proof}
\begin{proposition}\label{baco3}
{\rm (i)} If  ${\mathbb W}$ is of curvature~$\leq 0$  then
${\rm rk}(h({\mathbb W}))= {\rm ht}({\mathbb W})$.

{\rm (ii)}
In general, ${\rm rk}(h({\mathbb W}))= {\rm ht}({\mathbb W})
+{\rm rk}({\rm Ext}^{1}_{\rm VS}({\mathbb W},\Bdr))$.
\end{proposition}
\begin{proof} For (i), 
by Lemma~\ref{baco2.1}, we have the exact sequence
$0\to V\to {\mathbb W}\to {\mathbb M}\to 0$, where ${\mathbb M}$ is  affine.
This  yields  the exact sequence
$$0\to h({\mathbb M})\to h({\mathbb W})\to h(V)\to {\rm Ext}^{1}_{\rm VS}({\mathbb M},\Bdr).$$
Since $h({\mathbb M})=0$ and  we have ${\rm Ext}^{1}_{\rm VS}({\mathbb M},\Bdr)=0$, by Lemma~\ref{baco5}, this sequence implies that
 ${\rm rk}(h({\mathbb W}))= \dim_{\Q_p}V$. We are done because 
 ${\rm ht}({\mathbb W})=\dim_{\Q_p}V$. 

(ii) Follows via the exact sequence
of  ${\rm Ext}$ of a presentation:
 if $0\to V\to {\mathbb W}'\to {\mathbb W}\to 0$ represents  ${\mathbb W}$, where ${\mathbb W}'$ is an extension 
of  ${\mathbb V}^d$ by a $\Q_p$-vector space $V^{\prime}$ of finite dimension,  arguing as for (1) we get the following diagram with exact row:
$$
\xymatrix@R=4mm@C=4mm{
0\ar[r] & h({\mathbb W})\ar[r] & h({\mathbb W}^{\prime})\ar[r] \ar[d]^{\wr} & h(V)\ar[r] &  \Ext^{1}_{\rm VS}({\mathbb W},\Bdr)\ar[r] & \Ext^{1}_{\rm VS}({\mathbb W}^{\prime},\Bdr) \ar[r] & 0\\
& & h(V^{\prime})
}
$$
Since  ${\mathbb W}'$ is of curvature~$\leq 0$, by  Lemma~\ref{baco5}, we have   $\Ext^{1}_{\rm VS}({\mathbb W}^{\prime},\Bdr)=0$. It follows that 
$${\rm rk}(h({\mathbb W}))={\rm rk}(h(V^{\prime}))-{\rm rk}(h(V))+ {\rm rk}(\Ext^{1}_{\rm VS}({\mathbb W},\Bdr)),
$$
which gives us what we wanted because ${\rm ht}({\mathbb W})={\rm rk}(h(V^{\prime}))-{\rm rk}(h(V))$.
\end{proof}

\section{The category $\qBC$}\label{HN15} 
We  need to enlarge the category $\cal{BC}$ to allow extensions by arbitrary ${\mathbb B}_m$-Modules, 
for $m\geq 1$. 

\subsection{Categories of Topological Vector Spaces}\label{dqbc1}
We start by putting topology on Vector Spaces. 
\subsubsection{Topological Vector Spaces and $\Bdr^+$-Pairs}
\begin{definition}\label{dqbc2}
A {\it Topological Vector Space} ${\mathbb W}$ 
is a functor $\Lambda\mapsto {\mathbb W}(\Lambda)$
from sympathetic algebras to ${\rm LH}(C_{\Q_p})$ 
with the property that the image of ${\rm Hom}(\Lambda,\Lambda')\times B$
is bounded in ${\bb W}(\Lambda')$ for $B$ bounded in ${\bb W}(\Lambda)$.

$\bullet$
A morphism ${\bb W}\to {\bb W}'$ of Topological Vector Spaces, is a natural transformation
of functors.

$\bullet$ A sequence
$0\to{\bb W}_1\to {\bb W}\to {\bb W}_2\to 0$ of Topological Vector Spaces is exact if and only if
$0\to{\bb W}_1(\Lambda)\to {\bb W}(\Lambda)\to {\bb W}_2(\Lambda)\to 0$ 
is exact in ${\rm LH}(C_{\Q_p})$, for every $\Lambda$.
\end{definition}

\begin{remark}
The condition on the image of ${\rm Hom}(\Lambda,\Lambda')\times B$ is automatic
if ${\bb W}$ is a BC. This follows form the decomposition~(\ref{HNX9}) or the existence of a covering
of ${\bb W}$ by a direct sum of universal covers of $p$-divisible groups (these universal covers
have natural lattices which are functorial in $\Lambda$).  
\end{remark}
\begin{remark}
Following \cite[Cor. 1.2.20]{Schn}, we will represent the category ${\rm LH}(C_{\Q_p})$ as the localization of the full subcategory of the homotopy category of $C_{\Q_p}$ consisting of complexes $E$ of the form 
$$
0\to E_1\lomapr{\delta_E} E_0\to 0,
$$
where $E_0$ is in degree $0$ and $\delta_E$ is a monomorphism, by the multiplicative system formed by morphisms $u:E\to F$ such that the square 
$$
\xymatrix@R=5mm{
F_1\ar[r]^{\delta_F} & F_0\\
E_1\ar[r]^{\delta_E} \ar[u]^{u_1}&E_0\ar[u]^{u_0}
}
$$
is bicartesian. Note that this implies that for $E\in {\rm LH}(C_{\Q_p})$ the (algebraic) quotient $E_0/E_1$ is well-defined. 
\end{remark}

We denote by ${\rm Hom}_{\rm TVS}(-,-)$ and ${\rm Ext}^i_{\rm TVS}(-,-)$ 
the $\Hom$ and Yoneda $\Ext$ groups in the
category of {Topological Vector Space}.

\begin{definition}\label{dqbc3}
(i) A {\it $\Bdr^+$-Module} ${\bb W}$ is a VS such that there exists a topological $\bdr^+$-module $W$ (i.e., a $\bdr^+$-module in the category $C_{\Q_p}$)
killed by a power of $t$ (hence $W$ is a $\B_m$-module for $m$ big enough) such that 
${\bb W}(\Lambda)=\Bdr^+(\Lambda)\wotimes_{\bdr^+}W$ for all $\Lambda$ (if $W$ is killed by $t^m$,
we also have\footnote{${\bb B}_m(\Lambda)\wotimes_{\B_m}W$ is defined as the quotient of
${\bb B}_m(\Lambda)\wotimes_{\Q_p}W$ by the closure of the subspace generated by the
$\lambda a\otimes w-a\otimes\lambda w$, for $a\in {\bb B}_m(\Lambda)$, $\lambda\in\B_m$
and $w\in W$.  If we choose a Banach basis $(e_i)_{i\in I}$ of $\Lambda$ over $C$ (note that
$I$ is countable by our asumptions on $\Lambda$'s), and a lifting $\hat e_i$ of
$e_i$ in ${\bb B}_m(\Lambda)$, then $(\hat e_i)_{i\in I}$ is a Banach basis
of ${\bb B}_m(\Lambda)$ over $\B_m$, and we get isomorphisms ${\bb B}_m(\Lambda)\simeq
\ell^\infty_0(I,\B_m)$ (bounded sequences tending to $0$ at $\infty$), and 
${\bb B}_m(\Lambda)\wotimes_{\B_m}W\simeq
\ell^\infty_0(I,\B_m)\wotimes_{\B_m}W\simeq \ell^\infty_0(I,W)$ of $\B_m$-modules
(to recover the structure of ${\bb B}_m(\Lambda)$-module, one needs to express multiplication
by $\lambda\in {\bb B}_m(\Lambda)$ in the basis $(\hat e_i)_{i\in I}$).}
 ${\bb W}(\Lambda)={\bb B}_m(\Lambda)\wotimes_{\B_m}W$ for all $\Lambda$).
In particular ${\bb W}(C)=W$.
A morphism ${\bb W}_1\to {\bb W}_2$ of $\Bdr^+$-Modules is a morphism of TVS's such that 
${\mathbb W}_1(\Lambda)\to {\mathbb W}_2(\Lambda)$ is $\Bdr^+(\Lambda)$-linear for every $\Lambda$.

(ii) A {\it $\Bdr^+$-Pair} is a Topological Vector Space $({\mathbb W}_1,{\mathbb W}_2)$ obtained from  a pair $(W_1,W_2)$ of topological $\B_m$-modules, for some $m$ 
(i.e., on $C$-points, we have an injection $W_1\to W_2$ of topological $\B_m$-modules),
by base change as above. Note that ${\mathbb W}_1,{\mathbb W}_2$ are $\Bdr^+$-Modules. A morphism of $\Bdr^+$-Pairs is a morphism of TVSs that is $\Bdr^+$-linear. 
\end{definition}
We denote by ${\rm Hom}_{\Bdr^+}(-,-)$ and ${\rm Ext}^i_{\Bdr^+}(-,-)$ the $\Hom$ and $\Ext$ groups in the
category of $\Bdr^+$-Pairs.
We have natural maps
$${\rm Hom}_{\bdr^+}(W,W')\to {\rm Hom}_{\Bdr^+}({\bb W},{\bb W}')\to {\rm Hom}_{\rm TVS}({\bb W},{\bb W}')$$
and the same for $\Ext$ groups.
It follows from Theorem~\ref{iso2} below that, if $W,W'$ are reasonable of the same type, these maps are isomorphisms.

%
%

\subsubsection{qBC's}
\begin{definition}\label{dqbc4}
(i) A {\it qBC} is a Topological Vector Space ${\mathbb W}$ admitting a filtration
$$0\subset {\mathbb W}_{>0}\subset {\mathbb W}_{\geq 0}\subset{\mathbb W}$$
with:

\quad $\bullet$ ${\mathbb W}_{>0}$ -- a BC of curvature~$>0$,

\quad $\bullet$ ${\mathbb W}_{=0}:={\mathbb W}_{\geq 0}/{\mathbb W}_{>0}$ --
a $\Bdr^+$-Pair,

\quad $\bullet$ ${\mathbb W}_{<0}:={\mathbb W}/{\mathbb W}_{\geq 0}$ --
a BC of curvature~$<0$.

(ii) A {\it morphism ${\bb W}\to{\bb W}'$ of qBC's} is a morphism of TVS's compatible with the 
filtrations and such that the induced map ${\bb W}_{=0}\to {\bb W}'_{=0}$ is a morphism
of $\Bdr^+$-Pairs.

(iii) The {\it height ${\rm ht}({\bb W})$ of a qBC} ${\bb W}$ is defined as:
$${\rm ht}({\mathbb W}):={\rm ht}({\mathbb W}_{>0})+{\rm ht}({\mathbb W}_{<0})$$
It only depends on the isomorphism class of ${\bb W}$.
\end{definition}

\begin{remark}\label{dqbc5} 
Assume that ${\mathbb W}_{=0}$ is a reasonable $\Bdr^+$-Pair 
(see Section~\ref{PAT3} for the definition). 

(i) As shown in Proposition~\ref{filt},
 ${\mathbb W}_{>0}$ is the largest sub-BC of ${\mathbb W}$ of curvature~$>0$,
and ${\mathbb W}_{<0}$ is the largest 
quotient-BC of ${\mathbb W}$ of curvature~$<0$.
Hence the filtration on ${\mathbb W}$ is uniquely
determined.

(ii) If ${\bb W}$ and ${\bb W}'$ are qBC's, and
${\mathbb W}_{=0}$ and ${\mathbb W}'_{=0}$ are reasonable $\Bdr^+$-Pairs of the same type,
we show (see~Corollary~\ref{iso5})
 that any morphism ${\bb W}\to{\bb W}'$ of TVS's is a morphism of qBC's
(the most difficult part of the proof is to show that the induced morphism
${\bb W}_{=0}\to {\bb W}'_{=0}$ is a morphism
de $\Bdr^+$-Pairs).
\end{remark}
\begin{definition}\label{dqbc6}
{\rm ({\em Curvature})}
Let ${\mathbb W}\in\qBC$.  We say that
${\mathbb W}$ has:

$\bullet$ curvature~$>0$, if ${\bb W}={\bb W}_{>0}$.

$\bullet$ curvature~$\geq 0$, if ${\bb W}={\bb W}_{\geq 0}$.

$\bullet$ curvature~$=0$, if ${\bb W}={\bb W}_{=0}$ (i.e.~if ${\bb W}_{>0}={\bb W}_{<0}=0$).

$\bullet$ curvature~$<0$, if ${\bb W}={\bb W}_{<0}$ (i.e.~if ${\bb W}_{\geq 0}=0$).

$\bullet$ curvature~$\leq 0$, if ${\bb W}={\bb W}_{\leq 0}$ (i.e.~if ${\bb W}_{> 0}=0$).
\end{definition}

\begin{remark}\label{nonsplit}
(i) It is conceivable that the filtration on a qBC splits and that a qBC can be written
(non canonically) as the direct sum of a BC and a $\Bdr^+$-Pair, but proving such a statement 
requires to understand ${\rm Ext}^2({\bb Y},{\bb W})$ where ${\bb Y}$ is a BC  and
${\bb W}$ is a $\Bdr^+$-Module.  Presumably\footnote{As we will show in a sequel to this paper,
this vanishing always holds but does not imply that any qBC can be written
as the direct sum of a BC and a $\Bdr^+$-Pair, only that a qBC can be written as
an extension of a $\Bdr^+$-Pair by a BC. On the other hand, it allows to prove that
$\qBC$ is an abelian category.}, this group is~$0$ (this is the case if ${\bb W}$
is of finite length), but this remains to be proved (and point (ii) below shows
that strange things can happen passing from $C_{\Q_p}$ to ${\rm LH}(C_{\Q_p})$).

(ii)
If $W\in C_K$ is  Hausdorff and $V$ is finite dimensional, 
then any strict exact sequence $0\to V\to W'\to W\to 0$ splits (\cite[Lemma 2.5]{CGN}).
  This is not always 
the case if $W\in{\rm LH}(C_K)$; indeed, if $W=(W_1\to W_2)$ with $W_1,W_2$ Hausdorff, then we have
an exact sequence
$$0\to {\rm Hom}(W,V)\to {\rm Hom}(W_2,V)\to {\rm Hom}(W_1,V)\to {\rm Ext}^1(W,V)\to 0$$
(The next term ${\rm Ext}^1(W_2,V)$ vanishes since $W_2$ is Hausdorff.)

If $W_1$ is dense in $W_2$, then ${\rm Hom}(W,V)=0$ and
${\rm Ext}^1(W,V)={\rm Hom}(W_1,V)/{\rm Hom}(W_2,V)$.
For example, if $K=\Q_p$, $W_1={\rm LA}(\Z_p,\Q_p)$, $W_2={\cal C}(\Z_p,\Q_p)$, and $V=\Q_p$,
then ${\rm Ext}^1(W,V)={\cal D}(\Z_p,\Q_p)/{\rm Mes}(\Z_p,\Q_p)$ which is a big space
in which $0$ is dense (it is better to view it as an element of ${\rm LH}(C_{\Q_p})$).
\end{remark}

\subsection{A perfectoid point of view on ${\rm Ext}^1_{\rm TVS}({\bb V}^1,{\bb V}^1)$}\label{PAT2}
\ 

If $\Lambda$ is a sympathetic algebra, we denote by $\overline{\Lambda\{X\}}$ the sympathetic
closure of $\Lambda\{X\}$, by 
$H_{\Lambda\{X\}}$ the group ${\rm Aut}_{\rm cont}(\overline{\Lambda\{X\}}/\Lambda\{X\})$
and by $T_\Lambda$ the extension of $\O_\Lambda$ by $H_{\Lambda\{X\}}$ considered 
in~\cite[Prop.\,5.23]{CB}
(it is the group of continuous automorphisms $\tau$ of $\overline{\Lambda\{X\}}$ whose
restriction to
$\Lambda\{X\}$ is of the form $X\mapsto X+x(\tau)$, with $x(\tau)\in \O_\Lambda$; the map
$\tau\mapsto x(\tau)$ is a surjective group morphism whose kernel is
$H_{\Lambda\{X\}}$). We denote by $0$ the neutral element of $T_\Lambda$.

We fix a morphism $\alpha_C:\overline{C\{X\}}\to C$ with $\alpha_C(X)=0$, and we define
$f(\tau)$, if $\tau\in T_C$ and $f\in \overline{C\{X\}}$, by $f(\tau):=\alpha_C(\tau(f))$;
in particular, $X(\tau)=x(\tau)$ and $f(0)=\alpha_C(f)$.

If $\alpha\in C$ is such that $p^n\alpha\in 2p\O_C$, let
 $E_\alpha\in \overline{C\{X\}}$ be the unique
$p^n$-th root of $e^{p^n\alpha\,X}$ with $E_\alpha(0)=1$.
There exists a character $\chi_\alpha:T_C\to C^\dual$ such that
$\tau(E_\alpha)=\chi_\alpha(\tau)E_\alpha$ and an additive character
$\psi_\alpha:H_{C\{X\}}\to\Q_p$ such that $\chi_\alpha(\sigma)^\flat=\epsilon^{\psi_\alpha(\sigma)}$
(note that $\chi_\alpha$ takes values in $\bmu_{p^\infty}$ on $H_{C\{X\}}$).
Since $\chi_{\alpha}$ is a character,  $\psi_\alpha(\tau^{-1}\sigma\tau)=\psi_{\alpha}(\sigma)$
if $\sigma\in H_{C\{X\}}$ and $\tau\in T_C$.

\begin{lemma}\label{patch2}
$\O_C\{X\}\big[\frac{E_{1/p^n}-1}{p^{1/p^{n+1}}},\ n\in\N\big]$ is integral over $\O_{C\{X\}}$
and perfectoid\footnote{I.e., there exists $r>0$ such that $x\mapsto x^p$ is surjective modulo~$p^r$.}.
\end{lemma}
\begin{proof}
Set $Y_n=\frac{E_{1/p^n}-1}{p^{1/p^{n+1}}}$. Then $(p^{1/p}Y_0+1)^p=e^{pX}$, hence
$$Y_0^p+p^{(p-1)/p}Y_0^{p-1}+\tfrac{p-1}{2}p^{(p-2)/p}Y_0^{p-2}+\cdots+p^{1/p}Y_0
=\tfrac{1}{p}(e^{px}-1)$$
which proves that $Y_0$ is integral over $\O_C\{X\}$ and that $Y_0^p=X$ modulo $p^{1/p}$.

The same computations, strating from
$(p^{1/p^{n+1}}Y_n+1)^p=p^{1/p^n}Y_{n-1}+1$, show that $Y_n$ is integral over $\O_C[Y_{n-1}]$
and that $Y_n^p=Y_{n-1}$ modulo $p^{1/p}$.

The result follows.
\end{proof}

\begin{proposition}\label{patch3}
Let $0\to {\bb W}_1\to {\bb E}\to {\bb W}_2\to 0$ be an exact sequence of TVS's
with fixed isomorphims ${\bb W}_1\cong {\bb V}^1$ and ${\bb W}_2\cong {\bb V}^1$.
Let $v\in C$ and 
$h\in {\bb E}(\overline{C\{X\}})$ lifting $vX\in {\bb V}^1(\overline{C\{X\}})$.

{\rm (i)} There exists $\lambda_{v}\in C$ depending only on $v$,
and $g\in {\bb W}_1(\overline{C\{X\}})$ with $\Vert g\Vert\leq
p\sup_{\tau\in T_C}\Vert \tau(h)-h-h(\tau)\Vert$,
such that, if $\ell=h-g$,
we have $(\sigma-1)\ell=\lambda_{v}\psi_1(\sigma)$ for all $\sigma\in H_{C\{X\}}$.

{\rm (ii)} There exists $\lambda_{\bb E}\in C$ such that
$\lambda_v=v\lambda_{\bb E}$ for all $v\in C$,
and ${\bb E}\mapsto\lambda_{\bb E}$ induces
an isomorphism ${\rm Ext}^1_{\rm TVS}({\bb V}^1,{\bb V}^1)\simeq C$.
\end{proposition}
\begin{proof}
$\sigma\mapsto h_\sigma:=(\sigma-1)h$ is a continuous $1$-cocycle on $H_{C\{X\}}$
with values in ${\bb W}_1(\overline{C\{X\}})=\overline{C\{X\}}$.
The standard ``almost \'etale descent'' and ``decompletion'' arguments give
$g \in {\bb W}_1(\overline{C\{X\}})$ with $\Vert g\Vert\leq
p\sup_{\tau\in T_C}\Vert h_\sigma\Vert$, and $\lambda_h\in C\{X\}$
such that $(\sigma-1)h=\lambda_h\psi_1(\sigma)+(\sigma-1)g$, for all $\sigma\in H_{C\{X\}}$;
moreover, $\lambda_h$ is uniquely determined (in other words, $H^1(H_{C\{X\}},\overline{C\{X\}})\simeq
C\{X\}$; the inverse isomorphism being $\lambda\mapsto \lambda\psi_1$).
If $h'$ is another lifting of $vX$, then $\sigma\mapsto h'_\sigma:=(\sigma-1)h'$
differs from $\sigma\mapsto h_\sigma$ by a coboundary, hence $\lambda_h=\lambda_{h'}$;
it follows that we can set $\lambda_v=\lambda_h$ for any choice of $h$.

Now, if $\tau\in T_C$, then $\tau\cdot h-h(\tau)$ is another lifting of $vX$; hence
$\sigma\mapsto (\sigma-1)(\tau\cdot h-h-h(\tau))$ is a coboundary.
We also have $(\sigma-1)(\tau\cdot h-h-h(\tau))=
\tau(\tau^{-1}\sigma\tau-1)h-(\sigma-1)h$; we deduce that
$\tau\cdot(\lambda_h\psi_1(\tau^{-1}\sigma\tau))=\lambda_h\psi_1(\sigma)$, and since
$\psi_1(\tau^{-1}\sigma\tau)=\psi_1(\sigma)$, this implies that $\lambda_h$ is invariant by
translation by $\O_C$, hence $\lambda_h\in C$.

This proves (i).  To prove (ii), we have to prove surjectivity and
injectivity. Surjectivity follows from the
fact that ${\rm Ext}^1_{\rm TVS}({\bb V}^1,{\bb V}^1)$ is a module over ${\rm End}_{\rm TVS}({\bb V}^1)=C$
and the fact that $\lambda_v\neq 0$ if ${\bb E}={\bb B}_2$ and $v\neq 0$,
the injection ${\bb V}^1\to {\bb B}_2$ being $x\mapsto tx$. 
In fact, one checks that
 $\lambda_{\bb E}=1$
for ${\bb E}={\bb B}_2$ by using 
$\ell:=\hat v\log[X_1^\flat]\in{\bb B}_2(\overline{C\{X\}})$
for a lifting
of $vX\in{\bb V}^1(C\{X\})$, where $X_1^\flat=(X_1,X_{1/p},X_{1/p^2},\dots)$ and $\hat v\in\B_2$
is a lifting of $v$: one has $(\sigma-1)\ell=v\psi_1(\sigma)$, if $\sigma\in H_{C\{X\}}$,
by definition of $\psi_1$.

To prove injectivity, it is enough to check that, if $v\neq 0$ and $\lambda_v=0$, then
${\bb E}$ is split;
this requires some preparation and is proven in Remark~\ref{sympa9}. Once injectivity
is proven, it follows that any nontrivial ${\bb E}$ is isomorphic to
${\bb B}_2$, and the existence of $\lambda_{\bb E}$ follows from its existence for ${\bb B}_2$.
\end{proof}

\begin{lemma}\label{sympa2}
For all $\Lambda$:

{\rm (i)} We have

\quad $\bullet$  
$c_\tau:=(\tau-1)\ell-\alpha_\Lambda(\tau(\ell))\in \Lambda\{X\}\subset
{\bb W}_1(\overline{\Lambda\{X\}})$ for all $\tau\in T_{\Lambda}$;

\quad $\bullet$ $c_{\tau}=c_{\sigma\tau}=c_{\tau\sigma}$
for all $\tau\in T_{\Lambda}$ and $\sigma\in H_{\Lambda\{X\}}$.

{\rm (ii)} There exists $u\in C\{X\}\subset {\bb W}_1(\overline{C\{X\}})$ such that
$\ell_v:=\ell-u$ satisfies $\alpha_\Lambda(\tau_1\tau_2(\ell_v)-
\tau_1(\ell_v)-\tau_2(\ell_v)+\ell_v)=0$ 
for all $\tau_1,\tau_2\in T_\Lambda$ {\rm (i.e.~$\ell_v$ is 
an additive lifting of $vX$)}.
\end{lemma}
\begin{proof}
If $\sigma\in H_{\Lambda\{X\}}$, we have $(\sigma-1)c_\tau=(\sigma-1)(\tau-1)\ell=
\tau(\tau^{-1}\sigma\tau-1)\ell-(\sigma-1)\ell=
\lambda_v(\psi_1(\tau^{-1}\sigma\tau)-\psi_1(\sigma))=0$,
which proves the first statement.

Now $c_{\tau\sigma}-c_\tau=\tau((\sigma-1)\ell)-\alpha_{\Lambda}(\tau\sigma(\ell))+
\alpha_{\Lambda}(\tau(\ell))$. But $\sigma(\ell)=\ell+\alpha_C(\sigma(\ell))$;
hence $\alpha_{\Lambda}(\tau\sigma(\ell))-
\alpha_{\Lambda}(\tau(\ell))=\alpha_C(\sigma(\ell))$
and $\tau((\sigma-1)\ell)=\alpha_C(\sigma(\ell))$, which proves
that $c_{\tau\sigma}= c_\tau$.
If one applies this to $\tau^{-1}\sigma\tau\in H_{\Lambda\{X\}}$ instead of $\sigma$,
one gets $ c_{\sigma\tau}= c_\tau$ which proves (i).

Now (i) allows to consider $\tau\mapsto c_\tau$ as a function on $\O_{\Lambda}$
with values in $\Lambda\{X\}$; let us denote $F(X,x(\tau))$ the function corresponding
to $ c_\tau$.
We let $F(X,Y)\in C\{X,Y\}$ be the function corresponding to $ c_\tau$,
with $\tau\in T_{C\{Y\}}$ satisfying $x(\tau)=Y$.
Then if $s_y:\overline{C\{X,Y\}}\to \overline{\Lambda\{X\}}$ sends $Y$ to $y\in\O_\Lambda$
and $X$ to $X$,
we have $F(X,y)=s_y(F(X,Y))$: indeed, from~\cite[Prop.\,5.20]{CB}, we infer that,
if $\tau_y\in T_{\Lambda}$ and $\alpha_\Lambda:\overline{\Lambda\{X\}}\to\Lambda$
are given, with $x(\tau_y)=y$, then there exists
unique $\tau_Y\in T_{C\{Y\}}$ with $x(\tau_Y)=Y$ and $\alpha_{C\{Y\}}:
\overline{C\{X,Y\}}\to \overline{C\{Y\}}$ making the following diagram commute
$$\xymatrix@R=5mm{
\overline{C\{X,Y\}}\ar[r]^-{\tau_Y}\ar[d]^-{s_y}
&\overline{C\{X,Y\}}\ar[r]^-{\alpha_{C\{Y\}}}\ar[d]^-{s_y} 
&\overline{C\{Y\}}\ar[d]^-{s_y}\\
\overline{\Lambda\{X\}}\ar[r]^-{\tau_y}
&\overline{\Lambda\{X\}}\ar[r]^-{\alpha_\Lambda}
&\Lambda
}$$
which gives
$$s_y(F(X,Y))=s_y((\tau_Y-1)\ell-\alpha_{C\{Y\}}(\tau_Y(\ell)))=
(\tau_y-1)\ell-\alpha_{\Lambda}(\tau_y(\ell))=F(X,y)$$
Now, we have 
\begin{equation}\label{ctau}
c_{\tau_1\tau_2}-\tau_1c_{\tau_2}-c_{\tau_1}
=-\alpha_\Lambda(\tau_1\tau_2(\ell)-
\tau_1(\ell)-\tau_2(\ell))
\end{equation}
By construction, $(\tau_1,\tau_2)\mapsto c_{\tau_1\tau_2}-\tau_1c_{\tau_2}-c_{\tau_1}$
is a $2$-cocycle which we view as a $2$-cocycle
on $\O_{\Lambda}$ with values in $\Lambda$.
Using the function $F(X,Y)$, this translates into
$$F(X,Y_1+Y_2)-F(X+Y_1,Y_2)-F(X,Y_1):=f(Y_1,Y_2)\in C\{Y_1,Y_2\}$$
 the $2$-cocycle condition
becoming $f(Y_2,Y_3)-f(Y_1+Y_2,Y_3)+f(Y_1,Y_2+Y_3)-f(Y_1,Y_2)=0$.
Lemma~\ref{sympa1} gives us $u\in C\{X\}$ such that, if 
$$F'(X,Y):=F(X,Y)-(u(X+Y)-u(X)-u(Y))$$ 
then
$F'(X,Y_1+Y_2)-F'(X+Y_1,Y_2)-F'(X,Y_1)=0$.
Setting $\ell_v=\ell-u$ and $c'_\tau:=(\tau-1)\ell_v-\alpha_\Lambda(\tau(\ell_v))$,
formula~(\ref{ctau}) implies $\alpha_\Lambda(\tau_1\tau_2(\ell_v)-
\tau_1(\ell_v)-\tau_2(\ell_v))=0$, as wanted.
\end{proof}

\begin{lemma}\label{sympa1}
If $f\in\O_C\{Y,Z\}$ satisfies $f(Y,Z)-f(X+Y,Z)+f(X,Y+Z)-f(X,Y)=0$,
there exists $u\in p^{-1}\O_C\{X\}$ such that $f(Y,Z)=u(Y+Z)-u(Y)-u(Z)$.
\end{lemma}
\begin{proof}
By substracting $f(0,0)$ to $f$ (and adding $f(0,0)$ to $u$), we can assume $f(0,0)=0$.
Now, setting $Y=0$ in the above cocycle relation, we get $f(0,Z)-f(X,0)=0$, hence
$f(0,Z)=0$ and $f(Y,0)=0$. So we can write
$f(Y,Z)=Yf_1(Z)+Y^2f_2(Z)+\cdots$ Taking derivative with respect to $Y$ and putting $Y=0$,
we get $u'=f_1$ if $u$ exists, 
hence we have to take for $u$ the solution in $C[[X]]$ with $u(0)=0$.
Substrating $u(Y+Z)-u(Y)-u(Z)$ to $f(Y,Z)$, we get $f^{(1)}$ still satisfying the
cocycle relation but with a Taylor expansion
$Y^2f_2^{(1)}(Z)+Y^3f_3^{(1)}(Z)+\cdots$.  An easy induction shows that $f_k^{(1)}=0$
for all $k$ (consider the cocycle relation modulo $(Y,Z)^{k+1}$). Hence
$f(Y,Z)=u(Y+Z)-u(Y)-u(Z)$.  

It remains to check that $u\in p^{-1}\O_C\{X\}$. If $u=\sum_{k\geq 1}a_kX^k$,
and $f(Y,Z)=\sum_{i,j\geq 1}b_{i,j}Y^iZ^j$, then $a_k=\frac{i!}{k(k-1)\cdots(k-i+1)}b_{i,k-i}$
for any $i\in\{1,\dots,k-1\}$.  By taking $i=1$ if $(k,p)=1$, or $i=p^r$ with $r=v_p(k)$ if $r\geq 1$,
 we obtain that $a_k= u_k b_{i,k-1}$ with $u_k\in\Q_p$ and $v_p(u_k)=0$ or $-1$.
Since $b_{i,j}\in\O_C$ and $b_{i,j}\to 0$ when $(i,j)\to\infty$,
it follows that $a_k\in p^{-1}\O_C$ and $a_k\to 0$ when $k\to\infty$. 
\end{proof}

\begin{remark}\label{sympa9}
The {\it rank} of an additive element $\ell$ is the rank of the $\Z_p$-module
$\{(\sigma-1)\ell,\ \sigma\in H_{C\{X\}}\}$.
We have shown that, for all $v\in C$, the additive element $vX\in {\bb W}_2(\overline{C\{X\}})$
(which is of rank $0$) admits an additive lifting $\ell_v$
in ${\bb E}(\overline{C\{X\}})$ of rank~$\leq 1$, and it is of rank $0$
if and only if $\lambda_v=0$, which
is the case if ${\bb E}$ is split.  

To prove injectivity in (ii) of Proposition~\ref{patch3},
it suffices to prove that conversely, if $v\neq 0$ and $\ell_v$ is 
 of rank $0$, then ${\bb E}$ splits.
But $x\mapsto s_v(x):=(\tau_{x/v}-1)\ell_v$, with $\tau_{x/v}\in T_\Lambda$, $x(\tau_{x/v})=x/v$,
gives a functorial splitting $v\O_\Lambda\to {\bb E}(\Lambda)$
which can be extended, by $\Q_p$-linearity to a splitting ${\bb W}_2(\Lambda)\to {\bb E}(\Lambda)$.
Indeed, since $\ell_v$ is of rank~$0$, $s_v(x):=(\tau_{x/v}-1)\ell_v$ only depends on
$x/v$ and is a lifting of $(\tau_{x/v}-1)\cdot vX=x$; functoriality
comes from the fact that, if $\overline u:\overline{\Lambda\{X\}}\to \overline{\Lambda'\{X\}}$
is a lifting of $u:\Lambda\to\Lambda'$ (if $\overline u_1,\overline u_2$ are two
such liftings, there exists $\sigma\in H_{\Lambda\{X\}}$ 
with $\overline u_2=\overline u_1\circ\sigma$; this follows from (iii) of~\cite[Prop.\,5.20]{CB})
and if $\tau'\in T_{\Lambda'}$ there exists (a unique) $\tau\in T_\Lambda$
such that $\overline u\circ\tau=\tau'\circ\overline u$
(again by~\cite[Prop.\,5.20]{CB}).  Applying this to $\tau'$
with $x(\tau')=u(x)/v$
produces $\tau$ with $x(\tau)=x/v$, and functoriality of $s_v$ follows
from
$$u(s_v(x))=\overline u((\tau-1)\ell_v)=(\tau'-1)\ell_v=s_v(u(x))$$
\end{remark}

\begin{remark}\label{patch4}
(i)
Privileging $\psi_1$ as we did can seem fishy, but $\psi_\alpha-\alpha\psi_1$ is a coboundary
in $\overline{C\{X\}}$ (there exists an additive element $f_\alpha\in\overline{C\{X\}}$ such that
$(\sigma-1)f_\alpha=\psi_\alpha(\sigma)-\alpha\psi_1(\sigma)$ for all $\sigma\in H_{C\{X\}}$;
this is a consequence of~\cite[Th.\,10.5]{CB}).

(ii) The space ${\rm Sp}(C\{X\})$ is a $K(\pi,1)$ (it is the closed unit ball, an affinoid),
and $\overline{C\{X\}}$ is (the completion of) a pro-\'etale perfectoid extension of
$C\{X\}$;
hence almost \'etale descent gives
an isomorphism $H^1(H_{C\{X\}},\overline{C\{X\}})\simeq H^1_{\proeet}({\rm Sp}(C\{X\}),\O)$,
and the proof gives an isomorphism
$${\rm Ext}^1_{\rm TVS}({\bb V}^1,{\bb V}^1)
\simeq H^0(T_C,H^1_{\proeet}({\rm Sp}(C\{X\}),\O))\simeq H^0(T_C,C\{X\})=C.$$

%
\end{remark}

\subsection{Extensions of $\Bdr^+$-Modules}\label{PAT3}
\ 

We say that a topological $C$-vector space is {\em reasonable} 
if it is a fr\'echet or a compact inductive
limit of banachs, and if it is {\it separable} (i.e.,~ it contains a dense
sub-$C$-vector space of countable dimension).

Let $W\in C_C$  and let ${\bb W}:={\bb B}_1\wotimes_CW$ be the
functor $\Lambda\mapsto \Lambda\wotimes_CW$ (as a TVS, ${\bb B}_1\simeq{\bb V}^1$, but we are
going to use ${\bb B}_1$ to emphasize that it is also, naturally, a $\Bdr^+$-Module). 
\begin{remark}\label{exact}
If $0\to W_1\to W_2\to W_3\to 0$ is a strict  exact sequence  of reasonable topological $C$-vector spaces then the induced sequence
$$
0\to {\bb W}_1\to {\bb W}_2\to {\bb W}_3\to 0
$$
is an exact sequence of Topological Vector Spaces. Indeed, it suffices to check that, for sympathetic $\Lambda$, the sequence 
$$
0\to \Lambda\wotimes_CW_1\to \Lambda\wotimes_CW_2\to \Lambda\wotimes_CW_3\to 0
$$
is strict exact in $C_{\Q_p}$. But this follows since $W_*$'s are reasonable and $\Lambda$ is banach (see \cite[Sec. 2.1.3]{CDN3} for a discussion of exactness of the tensor product). 
\end{remark}
\subsubsection{Extensions of ${\bb B}_1$ by ${\bb W}$}\label{PAT4}

\begin{remark}\label{patch5}
 If $W\in C_C$ and $w\in W$, one builds the pushout extension $E_w$ of $C$ by $W$, using the
diagram
$$\xymatrix@R=4mm@C=5mm{0\ar[r]&C\ar[r]\ar[d]^-{w}& \B_2\ar[r]\ar[d]&C\ar[r]\ar@{=}[d]&0\\
0\ar[r]&W\ar[r]&E_w \ar[r]&C\ar[r]&0}$$
i.e., $E_w =(\B_2\oplus W)/\{(c,cw),\ c\in C\}$.
It is clear from the description that $E_w $ is a $\bdr^+$-module;
the above construction gives an isomorphism $W\stackrel{\sim}{\to}{\rm Ext}^1_{\bdr^+}(C,W)$.
The associated TVS extension is given by the pushout diagram
$$\xymatrix@R=4mm@C=5mm{0\ar[r]&{\bb B}_1\ar[r]\ar[d]^-{w}
& {\bb B}_2\ar[r]\ar[d]&{\bb B}_1\ar[r]\ar@{=}[d]&0\\
0\ar[r]&{\bb W}\ar[r]&{\bb E}_w \ar[r]&{\bb B}_1\ar[r]&0}$$
\end{remark}
\begin{proposition}\label{patch6}Let $W\in C_C$ be reasonable. 
We have natural isomorphims
$${\rm Hom}_{\rm TVS}({\bb B}_1,{\bb W})\stackrel{\sim}{\leftarrow}{\rm Hom}_{\bdr^+}(C,W)\stackrel{\sim}{\leftarrow}W
\quad{\rm and}\quad
{\rm Ext}^1_{\rm TVS}({\bb B}_1,{\bb W})\stackrel{\sim}{\leftarrow}{\rm Ext}^1_{\bdr^+}(C,W)\stackrel{\sim}{\leftarrow}W.$$
\end{proposition}
\begin{proof}
To prove the statement for $\Hom$, since $W\hookrightarrow {\rm Hom}_{\rm TVS}({\bb B}_1,{\bb W})$, it suffices to show that $W\stackrel{\sim}{\to} {\rm Hom}_{\rm VS}({\bb B}_1,{\bb W})$. Let $\alpha:{\bb B}_1\to{\bb W}$
be a morphism, and let $w=\alpha(1)\in W$. Then, for each $\lambda\in W^\dual$,
the composition $\lambda\circ\alpha:{\bb B}_1\to{\bb B}_1$ is the multiplication by
some element of $C$ and, evaluating at $1$, one gets that $\lambda\circ\alpha$
is the multiplication by $\lambda(w)$.  It follows that $\alpha$ is the map
$a\mapsto aw$ which allows to conclude.

Let us turn to ${\rm Ext}^1$.
As in the proof of Proposition~\ref{patch3}, a TVS-extension
$0\to {\bb W}\to{\bb E}\to{\bb B}_1\to 0$ is controled by the
cohomology class of the $1$-cocycle $\sigma\mapsto (\sigma-1)h$ 
on $H_{C\{X\}}$, where $h\in {\bb E}(\overline{C\{X\}})$
is a lifting of $X\in{\bb B}_1(\overline{C\{X\}})$.
Since all the steps in the proof of Proposition~\ref{patch3} (including Lemma~\ref{sympa1})
involve bounded denominators,
the proof of Proposition~\ref{patch3} goes through, using orthonormal bases,
to treat the case of  ${\bb W}$ being a (separable) banach.

The case of nuclear fr\'echets follows from
the case of banachs by taking projective limits (one can write a fr\'echet $W$
as a projective limits of banachs $W_n$), noticing that 
${\rm Rlim}^1{\rm Hom}_{\rm TVS}({\bb B}_1,{\bb B}_1\wotimes W_n)={\rm Rlim}^1 W_n$ 
vanishes by the topological
Mittag-Leffler criterium.

If $W=\varinjlim_nW_n$, use the fact that $H_{C\{X\}}$ is profinite which implies that any
continuous $1$-cocycle $\sigma \mapsto c_\sigma $ 
on $H_{C\{X\}}$ with values in $W$ takes values in $W_{n_0}$ for some big enough $n_0$
(Lemma~\ref{patch7}); this
allows to reduce this case to the banach case.
\end{proof}

\begin{lemma}\label{patch7}
Let $H$ be a profinite group, and $W=\varinjlim_nW_n$ a compact inductive limit of $C$-banachs,
with $H$ acting continuously on $W_n$ by isometries. 
If $\sigma\mapsto c_\sigma$ is a continuous cocycle on $H$ with values in $W$, there exists
$n_0\in\N$ such that $c_\sigma\in W_{n_0}$ for all $\sigma\in H$.
\end{lemma}
\begin{proof}
The topological closure $B_n$ inside $W$ of the unit ball of $W_n$ is $c$-compact, hence
$U_n:=\sum_{i=0}^np^{i-n}B_n$ is a closed subgroup of $W$, and $W=\cup_n U_n$.
If $H_n=\{\sigma \in H, c_\sigma \in U_n\}$, then $H_n$ is subgroup of $H$
(because of the cocycle condition and the fact that $H$ acts via isometries), which is closed
since $U_n$ is closed,
and $H=\cup_n H_n$. It follows, by Baire, 
that there exists $n_1$ such that $H_{n_1}$ has a non empty
interior, hence is open since it is a subgroup,
 and is of finite index since $H$ is profinite.
Take $n_0\geq n_1$ such that $W_{n_0}$ contains $c_\sigma $ 
for $\sigma $ in a system of representatives
of $H/H_{n_1}$; the cocycle relation implies that $c_\sigma \in W_{n_0}$ 
for all $\sigma \in H$.
\end{proof}

\subsubsection{Extensions of ${\mathbb B}_d$ by a $\Bdr^+$-Module}
Note that $\B_m$ is naturally a banach algebra (with ring of integers the image of $\ainf$).
We say that a topological $\B_m$-module is reasonable
if it is a $\B_m$-fr\'echet 
(i.e., the topology is defined by a countable family of norms of $\B_m$-modules)
or a compact (or rather $c$-compact) inductive
limit of $\B_m$-banachs.

Let $W$ be a  topological $\B_m$-module,
and let ${\bb W}:={\bb B}_m\wotimes_{\B_m}W$ be the
functor $\Lambda\mapsto {\bb B}_m(\Lambda)\wotimes_{\B_m}W$.

\begin{remark}
If $0\to W_1\to W_2\to W_3\to 0$ is a strict  exact sequence  of reasonable topological $\B_m$-modules then the induced sequence
$$
0\to {\bb W}_1\to {\bb W}_2\to {\bb W}_3\to 0
$$
is an exact sequence of ${\bb B}_m$-Modules. Indeed, it suffices to check that, for sympathetic $\Lambda$, the sequence 
$$
0\to  {\bb B}_m(\Lambda)\wotimes_{\B_m}W_1
\to  {\bb B}_m(\Lambda)\wotimes_{\B_m}W
\to  {\bb B}_m(\Lambda)\wotimes_{\B_m}W_2\to 0
$$
is strict exact in $C_{\Q_p}$.  But this follows by reduction to $m=1$, where it is clear because  $W_*$'s are reasonable and $\Lambda$ is banach. 
\end{remark}

\begin{remark}\label{patch13}
 If $W$ is a topological $\B_m$-module and $w\in W$, one builds the pushout extension $E_w$ of $\B_d$ by $W$, using the
diagram
$$\xymatrix@R=4mm@C=5mm{0\ar[r]&\B_m\ar[r]\ar[d]^-{w}& \B_{m+d}\ar[r]\ar[d]&\B_d\ar[r]\ar@{=}[d]&0\\
0\ar[r]&W\ar[r]&E_w \ar[r]&\B_d\ar[r]&0}$$
i.e., $E_w =(\B_{m+d}\oplus W)/\{(t^dc,cw),\ c\in \B_m\}$.
It is clear from the description that $E_w $ is a $\bdr^+$-module; the extension splits if
$w=t^dw'$ and
the above construction gives an isomorphism $W/t^d\stackrel{\sim}{\to}{\rm Ext}^1_{\bdr^+}(\B_d,W)$.
The associated TVS extension is given by the diagram
$$\xymatrix@R=4mm@C=5mm{0\ar[r]&{\bb B}_m\ar[r]\ar[d]^-{w}
& {\bb B}_{m+d}\ar[r]\ar[d]&{\bb B}_d\ar[r]\ar@{=}[d]&0\\
0\ar[r]&{\bb W}\ar[r]&{\bb E}_w \ar[r]&{\bb B}_d\ar[r]&0}$$
\end{remark}

\begin{proposition}\label{patch14}Let $W$ be a reasonable $\B_m$-module. 
We have natural  isomorphisms
$${\rm Hom}_{\rm TVS}({\bb B}_d,{\bb W})\stackrel{\sim}{\leftarrow}{\rm Hom}_{\bdr^+}(\B_d, W)\stackrel{\sim}{\leftarrow}W[t^d]
\quad{\rm and}\quad
{\rm Ext}^1_{\rm TVS}({\bb B}_d,{\bb W})\stackrel{\sim}{\leftarrow}{\rm Ext}^1_{\bdr^+}(\B_d, W)\stackrel{\sim}{\leftarrow}W/t^d,$$  
where $W[t^i]$ denotes the $t^i$-torsion submodule of $W$.
\end{proposition}
\begin{proof}
We argue by d\'evissage, first on $m$ (we have the exact sequence
$0\to W[t]\to W\to W/W[t]\to 0$; if $W$ is a reasonable topological $\B_m$-module,
$W[t]$ and $W/W[t]$ are reasonable topological $\B_{m-1}$-modules), then on $d$.
The case $d=m=1$ is covered by Proposition~\ref{patch6} and Remark~\ref{patch13}.

$\bullet$ Let $0\to W_1\to W\to W_2\to 0$ be an exact sequence of reasonable $\bdr^+$-modules
with $W_1,W_2$ killed by $t^{m-1}$,
and let $0\to {\bb W}_1\to {\bb W}\to {\bb W}_2\to 0$ be the corresponding
sequence of TVS's.  Set $h^i(-)={\rm Ext}^i_{\rm TVS}({\bb B}_1,-)$.
We have a commutative diagram with exact rows:
$$\xymatrix@R=4mm@C=4mm{
0\ar[r]&W_1[t]\ar[r]\ar[d]^-{\wr}&W[t]\ar[r]\ar[d]&W_2[t]\ar[r]\ar[d]^-{\wr}
&W_1/t\ar[r]\ar[d]^-{\wr}&W/t\ar[r]\ar[d]&W_2/t\ar[r]\ar[d]^-{\wr}\ar[r]&0\\
0\ar[r]& h^0({\bb W}_1)\ar[r]& h^0({\bb W})\ar[r]& h^0({\bb W}_2)
\ar[r]& h^1({\bb W}_1)\ar[r]& h^1({\bb W})\ar[r]& h^1({\bb W}_2)}$$
where the vertical isomorphisms come from the induction hypothesis.
It follows that the remaining two vertical arrows  are also isomorphisms, which allows to run the induction on $m$ for $d=1$.

$\bullet$ Now, use the 
long exact sequence of ${\rm Ext}$ groups for the exact sequence 
$0\to {\bb B}_1\lomapr{t^{d-1}}{\bb B}_d\to{\bb B}_{d-1}\to 0$. 
and set $h^i(-)={\rm Ext}_{\rm TVS}^i(-,{\bb W})$.
We have a commutative diagram with exact rows:
$$\xymatrix@R=4mm@C=4mm{
0\ar[r]&W[t^{d-1}]\ar[r]\ar[d]^-{\wr}&W[t^d]\ar[r]\ar[d]&W[t]\ar[r]\ar[d]^-{\wr}
&W/t^{d-1}\ar[r]\ar[d]^-{\wr}&W/t^d\ar[r]\ar[d]&W/t\ar[r]\ar[d]^-{\wr}\ar[r]&0\\
0\ar[r]& h^0({\bb B}_{d-1})\ar[r]& h^0({\bb B}_d)\ar[r]& h^0({\bb B}_1)
\ar[r]& h^1({\bb B}_{d-1})\ar[r]& h^1({\bb B}_d)\ar[r]& h^1({\bb B}_1)}$$
where the vertical isomorphisms come from the induction hypothesis.
It follows that the remaining two vertical arrows are also isomorphisms, which allows to run the induction on $d$ with fixed $m$.

This concludes the proof.
\end{proof}

\subsection{The filtration on a qBC}
Let ${\mathbb W}=({\mathbb W}_1,{\mathbb W}_2)$ be a reasonable $\Bdr^+$-Pair. That is, both ${\bb W}_1$ and ${\bb W}_2$ are reasonable $\Bdr^+$-Modules and of the same type
(i.e., both fr\'echets or both of compact type). 

\begin{lemma}\label{patch19}We have a natural isomorphism 
${\rm Hom}_{\rm TVS}({\bb B}_d,{\bb W})\stackrel{\sim}{\leftarrow}(W_2/W_1)[t^d]$.
\end{lemma}
\begin{proof}
We have the following diagram with exact rows, in which the middle row is
computed in the left heart of the category of topological $\bdr^+$-modules.
The first row is the
the beginning of the long exact sequence
of ${\rm Ext}_{\bdr^+}(\B_d,-)$ (without topology)
associated to the exact sequence $0\to W_1\to W_2\to W_2/W_1\to 0$.
And ${\rm Hom}$ and ${\rm Ext}^1$ in the last row stand for
${\rm Hom}_{\rm TVS}$ and ${\rm Ext}_{\rm TVS}^1$.
$$\xymatrix@R=4mm@C=5mm{
W_1[t^d]\ar[r]\ar[d]^-{\wr}
&W_2[t^d]\ar[r]\ar[d]^-{\wr}
&(W_2/W_1)[t^d]\ar[r]\ar[d]
&W_1/t^d\ar[r]\ar[d]^-{\wr}
&W_2/t^d\ar[d]^-{\wr}\\
{\rm Hom}({\B}_d,{W}_1)\ar[r]\ar[d]^{\wr}
&{\rm Hom}({\B}_d,{W}_2)\ar[r]\ar[d]^{\wr}
& {\rm Hom}({\B}_d,{W})\ar[r]\ar[d]
&{\rm Ext}^1({\B}_d,{W}_1)\ar[r]\ar[d]^{\wr}
&{\rm Ext}^1({\B}_d,{W}_2)\ar[d]^{\wr}\\
{\rm Hom}({\bb B}_d,{\bb W}_1)\ar[r]
&{\rm Hom}({\bb B}_d,{\bb W}_2)\ar[r]
& {\rm Hom}({\bb B}_d,{\bb W})\ar[r]
&{\rm Ext}^1({\bb B}_d,{\bb W}_1)\ar[r]
&{\rm Ext}^1({\bb B}_d,{\bb W}_2)
}$$
By Proposition~\ref{patch14} (and Remark~\ref{patch13}), 
all vertical arrows are isomorphisms except the middle ones.
We conclude by the 5 Lemma that the same is true for the middle arrows.
\end{proof}

\begin{lemma}\label{patch14.5}
{\rm (i)}
If ${\bb W}$ is a reasonable $\Bdr^+$-Module and $W={\bb W}(C)$, then
$${\rm Hom}_{\rm TVS}({\bb U}_d,{\bb W})\stackrel{\sim}{\leftarrow}W, \quad {\rm Ext}^1_{\rm TVS}({\bb U}_d,{\bb W})\stackrel{\sim}{\leftarrow}W, \quad d\in\N.$$

{\rm (ii)} If ${\bb W}=({\bb W}_1,{\bb W}_2)$ is a reasonable $\Bdr^+$-Pair,
then $${\rm Hom}_{\rm TVS}({\bb U}_d,{\bb W})\stackrel{\sim}{\leftarrow}W_2/W_1.$$
\end{lemma}
\begin{proof}For (i), 
choose $m$ such that $t^m$ kills ${\bb W}$. The injection ${\bb U}_d\hookrightarrow{\bb B}_{m+d}$
gives rise to a commutative diagram with exact rows, in which $h^0(-):={\rm Hom}_{\rm TVS}(-,{\bb W})$
and $h^1(-):={\rm Ext}^1_{\rm TVS}(-,{\bb W})$
$$\xymatrix@R=4mm@C=4mm{
0\ar[r]& h^0({\bb B}_d)\ar[r]&h^0({\bb U}_d)\ar[r]&h^0(\Q_pt^d)\ar[r]&
h^1({\bb B}_d)\ar[r]&h^1({\bb U}_d)\ar[r]&h^1(\Q_pt^d)\\
0\ar[r]& h^0({\bb B}_d)\ar[r]\ar@{=}[u]&h^0({\bb B}_{d+m})\ar[r]\ar[u]&h^0(t^d{\bb B}_m)\ar[r]\ar[u]&
h^1({\bb B}_d)\ar[r]\ar@{=}[u]&h^1({\bb B}_{d+m})\ar[r]\ar[u]&h^1(t^d{\bb B}_m)\ar[u]\\
0\ar[r]&W[t^d]\ar[u]_-{\wr}\ar[r]&W\ar[u]_-{\wr}\ar[r]^-{t^d}&W\ar[u]_-{\wr}\ar[r]
&W/t^d\ar[u]_-{\wr}\ar[r]&W\ar[u]_-{\wr}\ar[r]&W\ar[u]_-{\wr}
}$$
The isomorphisms between the second and third rows come from Proposition~\ref{patch14}
and the assumption that $t^m$ kills ${\bb W}$.

Now, $h^1(\Q_pt^d)=0$
and $h^0(t^d{\bb B}_{m})\to h^0(\Q_pt^d)$ is an isomorphism. 
Hence $h^0(\Q_pt^d)\to h^1({\bb B}_d)$ is surjective, and $h^1({\bb U}_d)=0$.
An application of the $5$ lemma shows that $h^0({\bb B}_{d+m})\to h^0({\bb U}_d)$ is
an isomorphism which finishes the proof of (i).

Since (ii) is a direct consequence of (i) and the exact sequence of Ext's, this concludes
the proof.
\end{proof}

\begin{proposition}\label{patch20}
If ${\mathbb Y}$ is a BC of curvature~$>0$ and if ${\mathbb W}$ is a
reasonable $\Bdr^+$-Pair,
then ${\rm Hom}_{\rm TVS}({\mathbb Y},{\mathbb W})=0$.
\end{proposition}
\begin{proof}
${\mathbb Y}$ is a direct sum of BC's of the form ${\mathbb B}_d/\Q_{p^h}$ or ${\bb U}_d/\Q_pt_x^d$ 
with $x\neq \infty$ (i.e.~$\theta(t_x)\neq 0$), hence we may assume
${\mathbb Y}$ to be of one of these forms.  

$\bullet$ If ${\bb Y}={\mathbb B}_d/\Q_{p^h}$, we
have an exact sequence:
$$0\to {\rm Hom}_{\rm TVS}({\mathbb Y},{\mathbb W})\to 
{\rm Hom}_{\rm TVS}({\mathbb B}_d,{\mathbb W})\to {\rm Hom}_{\rm TVS}(\Q_{p^h},{\mathbb W})$$
and it is enough to show the injectivity of the last map.
We have ${\rm Hom}_{\rm TVS}(\Q_{p^h},{\mathbb W})={\mathbb W}(C)^{\oplus h}$
and Lemma~\ref{patch19} shows that
${\rm Hom}_{\rm TVS}({\mathbb B}_d,{\mathbb W})=
{\mathbb W}(C)[t^d]$ which injects into ${\mathbb W}(C)={\rm Hom}_{\rm TVS}(\Q_{p},{\mathbb W})$
which is a quotient of ${\rm Hom}_{\rm TVS}(\Q_{p^h},{\mathbb W})$.

$\bullet$ If ${\bb Y}={\bb U}_d/\Q_pt_x^d$, we
have an exact sequence:
$$0\to {\rm Hom}_{\rm TVS}({\mathbb Y},{\mathbb W})\to
{\rm Hom}_{\rm TVS}({\mathbb U}_d,{\mathbb W})\to {\rm Hom}_{\rm TVS}(\Q_{p}t_x^d,{\mathbb W})$$
The induced map $W\to W$ coming from ${\rm Hom}_{\rm TVS}({\mathbb U}_d,{\mathbb W})
\simeq {\rm Hom}_{\rm TVS}(\Q_{p},{\mathbb W})=W$ is multiplication by $t_x^d$, hence
is an isomorphism since $\theta(t_x)\neq 0$. It follows 
that ${\rm Hom}_{\rm TVS}({\mathbb Y},{\mathbb W})=0$, as wanted.
\end{proof}

\begin{proposition}\label{filt}
Let ${\bb W}$ be a  reasonable qBC, i.e., ${\bb W}_{=0}$ is a reasonable $\Bdr^+$-pair. 

{\rm (i)} ${\bb W}_{>0}$ is the largest sub-BC of ${\bb W}$ of curvature~$>0$.

{\rm (ii)} ${\bb W}_{<0}$ is the largest quotient-BC of ${\bb W}$ of curvature~$<0$.
\end{proposition}
\begin{proof}
Let ${\bb Y}\hookrightarrow{\bb W}$ be a BC of curvature~$>0$.
The map ${\bb Y}\to {\bb W}_{<0}$ is $0$ by Remark~\ref{ppp1}, hence ${\bb Y}\hookrightarrow
{\bb W}_{\geq 0}$. But the induced map ${\bb Y}\to {\bb W}_{=0}$ is $0$
by Proposition~\ref{patch20}, hence ${\bb Y}\hookrightarrow{\bb W}_{>0}$,
which proves (i).

Now, let ${\bb W}\to {\bb Y}$ be a surjective map of TVS's, with ${\bb Y}$
a BC of curvature~$<0$.  As above, the restriction ${\bb W}_{>0}\to{\bb Y}$
is $0$, hence ${\bb W}\to {\bb Y}$ factors through ${\bb W}_{\leq 0}$.

If ${\bb W}_{=0}=({\bb W}_1,{\bb W}_2)$, we
have ${\rm Hom}_{\rm TVS}({\bb W}_{=0},{\bb Y})
\hookrightarrow {\rm Hom}_{\rm TVS}({\bb W}_2,{\bb Y})$,
and a morphism ${\mathbb W}_{=0}\to {\mathbb Y}$ as above is $0$
on any sub-$\Bdr^+$-Module of finite length of ${\mathbb W}_2$; since the union of
these modules is dense in ${\mathbb W}_2$, the induced morphism
${\mathbb W}_2\to {\mathbb Y}$ is identically $0$, 
and so is ${\mathbb W}_{=0}\to {\mathbb Y}$;
hence ${\bb W}\to {\bb Y}$ factors through ${\bb W}_{<0}$,
which proves (ii).
\end{proof}

\subsection{Morphisms of qBC's}
We show here that morphisms between reasonable qBCs of the same type behave as expected. 
\subsubsection{Extensions of reasonable $\Bdr^+$-Modules}
\begin{proposition}\label{iso1a}
If ${W}_1$ and ${W_2}$ are reasonable topological $\B_m$-modules of the same type,
and ${\bb W}_1$ and ${\bb W_2}$ are the associated topological ${\bb B}_m$-Modules,
the natural maps 
\begin{align}
\label{IAS1}
{\rm Hom}_{\bdr^+}(W_2,W_1)\to {\rm Hom}_{\Bdr^+}({\bb W}_2&,{\bb W}_1)
\to{\rm Hom}_{\rm TVS}({\bb W}_2,{\bb W}_1)\\
{\rm Ext}^1_{\bdr^+}(W_2,W_1)\to {\rm Ext}^1_{\Bdr^+}({\bb W}_2&,{\bb W}_1)
\to{\rm Ext}^1_{\rm TVS}({\bb W}_2,{\bb W}_1)\notag
\end{align}
are isomorphisms. 
\end{proposition}
\begin{proof}
Let $\alpha:{\bb W}_2\to {\bb W}_1$ be a morphism of TVS's. To prove
the statement
about Hom's we just have to check that $\alpha$
is $\Bdr^+$-linear.  Linearity can be tested on finite rank
submodules, and a finite rank sub-$\bdr^+$-module $Y$ of $W_2$ is of the form
$\B_{d_1}e_1\oplus\cdots\oplus\B_{d_s}e_s$, hence the sub-$\Bdr^+$-module ${\bb Y}$ generated
by $Y$ is ${\bb B}_{d_1}e_1\oplus\cdots\oplus {\bb B}_{d_s}e_s$. It follows
from Proposition~\ref{patch14} that the restriction of $\alpha$ to ${\bb Y}$
is $\Bdr^+$-linear and  the restriction of $\alpha_C$ to $Y$ is
$\bdr^+$-linear.  This proves that $\alpha_C$ is $\bdr^+$-linear,
and that for any sympathetic $\Lambda$ the restriction of $\alpha_\Lambda$
to $\Bdr^+(\Lambda)\otimes_{\bdr^+}W_2$ is $\Bdr^+(\Lambda)$-linear.
Since $\Bdr^+(\Lambda)\otimes_{\bdr^+}W_2$ is dense in ${\bb W}_2(\Lambda)$,
and since $\alpha_\Lambda$ is continuous, it follows that $\alpha_\Lambda$
is $\Bdr^+(\Lambda)$-linear, and that $\alpha$ is $\Bdr^+$-linear, as wanted.

The statement about ${\rm Ext}^1$'s is much more delicate.
Let ${\bb W}={\bb W}_0$ be a TVS-extension of ${\bb W}_2$ by ${\bb W}_1$, and set $W:={\bb W}(C)$.
It follows from Proposition~\ref{patch14} that $W$ is naturally a $\bdr^+$-module:
indeed it is enough to check that, for elements  $v_1,\dots, v_r$  of $W_2$, if
$Y$ is the $\bdr^+$-submodule they generate, then the inverse image $\widetilde Y$ of $Y$
in $W$ is a $\bdr^+$-module; but $Y=\B_{d_1}e_1\oplus\cdots\oplus \B_{d_s}e_s$
and if we let ${\bb Y}$ be the sub-$\Bdr^+$-Module 
${\bb B}_{d_1}e_1\oplus\cdots\oplus {\bb B}_{d_s}e_s$ of ${\bb W}_2$,
Proposition~\ref{patch14} implies that its inverse image $\widetilde{\bb Y}$ in
${\bb W}$ is a $\Bdr^+$-Module; hence $\widetilde Y=\widetilde{\bb Y}(C)$ is a $\bdr^+$-module.
This gives a map
$${\rm Res}_C:{\rm Ext}^1_{\rm TVS}({\bb W}_2,{\bb W}_1)\to {\rm Ext}^1_{\bdr^+}(W_2,W_1)$$
inverse to the natural map, which proves that ${\rm Res}_C$ is surjective.

It  remains to show that ${\rm Res}_C$ is injective. That is to say: if
${\bb W}$ is an TVS-extension of ${\bb W}_2$ by ${\bb W}_1$, and if 
$W\to W_2$ admits a continuous $\bdr^+$-linear section, then ${\bb W}$ splits.


First assume that $W_1,W_2$ are killed by $t$ (i.e.~are  topological $C$-vector spaces).
If $i=0,1,2$, let $\ell({\bb W}_i)$ be the set of additive elements of rank $0$ of
${\bb W}_i(\overline{C\{X\}})$, i.e. the set of $\ell\in {\bb W}_i(\overline{C\{X\}})$,
such that $\ell(\tau)=0$ if $\tau\in H_{C\{X\}}$ and $\ell(\sigma\tau)=\ell(\sigma)+\ell(\tau)$
for all $\sigma,\tau\in T_C$.  
Since $\ell({\bb B}_1)=\{c X,\ c\in C\}$, the map $v\mapsto \ell_v=Xv$ gives an isomorphism
$W_i\overset{\sim}{\to} \ell({\bb W}_i)$, if $i=1,2$.

 In the other direction, 
note that, if $\ell\in\ell({\bb W}_i)$, then
 $\tau\mapsto\ell(\tau)$ is a continuous group morphism from
$T_C$ to $W_i$, which factors through $T_C/H_{C\{X\}}=\O_C$, hence can be
seen as a continuous morphism of $\Z_p$ modules $\ell:\O_C\to W_i$.
If $i=1,2$, then $\ell\mapsto \ell(1)$ from $\ell({\bb W}_i)$ to $W_i$
is the inverse of the above map.

Now, 
$\ell({\bb W})\to \ell({\bb W}_2)$ is surjective (if $v\in W_2$ and $\ell_v=Xv$,
Remark~\ref{sympa9} -- adapted to the setting of Proposition~\ref{patch6} --
 produces $\ell\in {\bb W}(\overline{C\{X\}})$ lifting $\ell_v$,
additive of rank~$\leq 1$, and of rank $1$ if and only if the extension
of $Cv$ by $W_1$ is not split (as a $\bdr^+$-module)).
Hence $\ell\mapsto \ell(1)$ gives rise to an isomorphism
$\ell({\bb W})\overset{\sim}{\to} W$.
But ${\bb W}(\overline{C\{X\}})$ is reasonable (cf.~Appendix~\ref{extlh}),
and $\ell({\bb W})$ is closed in ${\bb W}(\overline{C\{X\}})$
as can be seen from the definition. It follows that $\ell({\bb W})$
is also reasonable, and the open mapping theorem implies that 
 the above isomorphism is a topological isomorphism.
Composing the section $W_2\to W$ with the inverse of the above isomorphism
produces a continuous $\Q_p$-linear map
$v\mapsto \ell_v$ from $W_2$ to $\ell({\bb W})$,
with the property that $\ell_v(1)=v\in W_2\subset W$.
This also defines a splitting $\ell({\bb W})=\ell({\bb W}_1)\oplus\ell({\bb W}_2)$,
which allow to 
define a $C$-module structure on $\ell({\bb W})$ by setting $c\cdot\ell_v=\ell_{cv}$
($\ell({\bb W}_1)$ is already a $C$-module).
Then, by definition, $(c\cdot\ell_v)(x)=x\,cv=c\,xv=c\,\ell_v(x)\in W_2$, if $x\in\O_C$.

Let now $\Lambda$ be a sympathetic algebra.
The action of $H_{\Lambda\{X\}}$ on $\ell_v$ factors through $H_{C\{X\}}$ since
$\ell_v\in {\bb W}(\overline{C\{X\}})$; hence it is trivial. 
It follows that
$\tau\mapsto s_0(\tau(\ell_v))$ factors through $\O_\Lambda$ and defines
a continuous function $x\mapsto\ell_v(x)$ on $\O_\Lambda$, with values
in ${\bb W}(\Lambda)$.  
We will need the following lemma.

\begin{lemma}\label{sympa3}
$\ell_v:\O_\Lambda\to {\bb W}(\Lambda)$ is a continuous group morphism;
moreover it is $C$-linear in the sense that
$\ell_v(cx)=(c\cdot\ell_v)(x)$ for all $c\in \O_C$
and $x\in\O_\Lambda$.
\end{lemma}
\begin{proof}
Let $\iota_\Lambda:\overline{C\{X\}}\to \overline{\Lambda\{X\}}$ be the
natural injection. Then
$\tau\mapsto (\tau-1)\iota_\Lambda(\ell_v)$ is a continuous $1$-cocycle 
on $T_\Lambda$ with values in
${\bb W}(\overline{\Lambda\{X\}})$.  Now, if $\alpha:\Lambda\to C$ is continuous,
we can extend $\alpha$ (non uniquely) to $\tilde \alpha:\overline{\Lambda\{X\}}\to
\overline{C\{X\}}$ with $\tilde \alpha(X)=X$.
In particular $\tilde \alpha\circ\tau\circ\iota_\Lambda\in T_C$ and its image in $\O_C$
is $\alpha(x(\tau))$ (remember that $x(\tau)\in\Lambda$).

Then $\tilde \alpha((\tau-1)\circ\iota_\Lambda(\ell_v))=\ell_v(\alpha(x(\tau))$.
It follows that $(\tau-1)\iota_\Lambda(\ell_v)-\ell_v(x(\tau))$ is killed by
$\tilde \alpha$ for all $\tilde \alpha$, which implies that it is $0$.
Hence the above cocycle is $\tau\mapsto\ell(x(\tau))$ and is a morphism of
groups since $T_\Lambda$ acts trivially on ${\bb W}(\Lambda)$.
This proves the first statement.

To prove the second statement,
let $c\in\O_C^\dual$, and choose 
an isomorphism $\tilde c:\overline{\Lambda\{X\}}\to \overline{\Lambda\{X\}}$
with $\tilde c(X)=cX$ (this implies that $\tilde c$ induces an automorphism
of $\overline{C\{X\}}$).  Note that, if $\tilde c_1,\tilde c_2$ are two such isomorphisms,
then $\tilde c_1^{-1}\circ\tilde c_2\in H_{\Lambda\{X\}}$. It follows that $\tilde c(\ell_v)$
only depends on $c$ and not on the choice of $\tilde c$.
Fix also $\alpha_\Lambda: \overline{\Lambda\{X\}}\to \Lambda$ with $\alpha_\Lambda(X)=0$

Let us check that $\tilde c(\ell_v)=\ell_{cv}$.
If $\tau\in T_\Lambda$, then $\tilde c^{-1}\tau\tilde c\in T_\Lambda$ and 
\begin{align*}
x(\tilde c^{-1}\tau\tilde c)&=\alpha_\Lambda(\tilde c^{-1}\tau\tilde c(X))=
\alpha_\Lambda(\tilde c^{-1}\tau(cX))\\
&=\alpha_\Lambda(\tilde c^{-1}(c(X+x(\tau))))
=\alpha_\Lambda(X+c\,x(\tau))=c\,x(\tau)
\end{align*}
Hence 
$$\tau(\tilde c(\ell_v))=\tilde c(\tilde c^{-1}\tau\tilde c(\ell_v))=
\tilde c(\ell_v+\ell_v(c\,x(\tau)))=\tilde c(\ell_v)+\ell_v(c\,x(\tau))$$
Now, if $\Lambda=C$, then $\ell_v(c\,x(\tau))=c\,x(\tau)\,v=\ell_{cv}(x(\tau))$
It follows that $\tilde c(\ell_v)-\ell_{cv}$ is invariant by $T_C$, hence belongs to
$W$. Moreover, $\alpha_C\circ \tilde c:\overline{C\{X\}}\to C$ sends $X$ to $0$,
hence there existe $\sigma\in H_{C\{X\}}$ such that $\alpha_C\circ\tilde c=\alpha_C\circ\sigma$;
it follows that $\alpha_C(\tilde c(\ell_v))=\alpha_C(\ell_v)=0$ since $\ell_v$ is additive.
Since $\alpha_C(\ell_{cv})=0$ because $\ell_{cv}$ is also additive, we get that
$\alpha_C(\tilde c(\ell_v)-\ell_{cv})=0$, hence $\tilde c(\ell_v)-\ell_{cv}=0$.
Hence
$$\ell_v(c\, x(\tau))=(\tau-1)\cdot(\tilde c(\ell_v))=(\tau-1)\cdot\ell_{cv}=\ell_{cv}(x(\tau))$$
which concludes the proof of the lemma.
\end{proof}

Let us come back to the proof of the proposition:
we are going to use the lemma to split ${\bb W}(\Lambda)\to{\bb W}_2(\Lambda)$.
Let $(e_i)_{i\in I}$ be an ``orthornormal'' basis of $\O_\Lambda$ over $\O_C$
(i.e.~$\inf_i{v_p(x_i)}\leq v_{\Lambda}(\sum_ix_ie_i)\leq 1+\inf_i{v_p(x_i)}$).
Then any element of ${\bb W}_2(\Lambda)$ can be written uniquely as
$\sum_i e_i v_i$, with $v_i\in W_2$ tending to $0$ when $i\to\infty$.
Now $\ell_{v_i}(e_i)\to 0$ in ${\bb W}(\Lambda)$ (one can find $k_i\in\N$, with $k_i\to\infty$,
such that $p^{-k_i}v_i$ is bounded in $W_2$, and $s_i:\overline{C\{X\}}\to \Lambda$
with $s_i(X)=e_i$, then $p^{-k_i}\ell_{v_i}(e_i)=s_i(\ell_{p^{-k_i}v_i})$ is bounded
in ${\bb W}(\Lambda)$, and $\ell_{v_i}(e_i)\to 0$).
Hence we can set $s(\sum_i e_i v_i)=\sum_i \ell_{v_i}(e_i)$. 
Then we have $s(\sum_j f_jw_j)=\sum_j\ell_{w_j}(f_j)$ for any
family $(f_j)_{j\in J}$ bounded in $\Lambda$ and any family
$(w_j)_{j\in J}$ converging to $0$ in $W_2$. Indeed,
write $f_j=\sum_ia_{i,j}e_i$ and $v_i=\sum_ia_{i,j}w_j$ (then $a_{i,j}$ is bounded
and, for all $j$, $a_{i,j}\to 0$ when $i\to\infty$, which implies that
$v_i\to 0$ when $i\to\infty$).
Then $\sum_j f_jw_j=\sum_{i,j}a_{i,j}e_iw_j=\sum_ie_iv_i$.
But 
$$\sum_j\ell_{w_j}(f_j)=\sum_j\ell_{w_j}(\sum_ia_{i,j}e_i)=
\sum_{i,j}\ell_{a_{i,j}w_j}(e_i)=
\sum_i\ell_{\sum_j a_{i,j}w_j}(e_i)=\sum_i\ell_{v_i}(e_i)$$ 
This shows,
in particular, that $s$ does not depend on the choice of $(e_i)_i$.
It also implies that $s$ is functorial in $\Lambda$.  This finishes the proof
of injectivity of ${\rm Res}_C$
in the case $W_1,W_2$ killed by $t$.

Injectivity of ${\rm Res}_C$ in
the general case follows by d\'evissage, using the long exact sequences of
${\rm Ext}$'s and arguing by induction first on $m_1$, then on $m_2$,
if $W_i$ is killed by $t^{m_i}$.
\end{proof}

\subsubsection{Morphisms of reasonable $\Bdr^+$-Pairs}
\begin{theorem}\label{iso2}
Let ${\bb W}=({\bb W}_1,{\bb W}_2)$ and ${\bb W}'=({\bb W}'_1,{\bb W}'_2)$ be
reasonable $\Bdr^+$-Pairs of the same type, and let $W,W'$ be the associated $\bdr^+$-pairs. Then the natural maps
$${\rm Hom}_{\bdr^+}(W,W')\to{\rm Hom}_{\Bdr^+}({\bb W},{\bb W}')\to{\rm Hom}_{\rm TVS}({\bb W},{\bb W}')$$
 are isomorphisms.
\end{theorem}
\begin{proof}The short exact sequence (in ${\rm LH}(C_{\Q_p})$)
$$
0\to W_1\to W_2\to W \to 0
$$
yields the isomorphism
$${\rm Hom}_{\bdr^+}(W,W')\stackrel{\sim}{\to}{\rm Ker}\big({\rm Hom}_{\bdr^+}(W_2,W')\to {\rm Hom}_{\bdr^+}(W_1,W')\big).$$
Similarly  for ${\rm Hom}_{\Bdr^+}$ and ${\rm Hom}_{\rm TVS}$.  
Hence, we are reduced to showing the claim of the theorem  for $W_1,W_2$;
in other words, we may suppose $W$ to be a reasonable $\bdr^+$-module. 

  Now we have a map of long exact sequences (in the first line we work in the category of $\bdr^+$-pairs
and in the second in the category of Topological Vector Spaces):
$$\xymatrix@R=4mm@C=5mm{
0\ar[r] &  {\rm Hom}({W},{W'}_1)\ar[r]\ar[d]^-{\wr}
&{\rm Hom}({W},{W'}_2)\ar[r]\ar[d]^-{\wr}
& {\rm Hom}({W},{W'})\ar[r]\ar[d]
&{\rm Ext}^1({W},{W'}_1)\ar[r]\ar[d]^-{\wr}
&{\rm Ext}^1({W},{W'}_2)\ar[d]^-{\wr}\\
0\ar[r] & {\rm Hom}({\bb W},{\bb W'}_1)\ar[r]
&{\rm Hom}({\bb W},{\bb W'}_2)\ar[r]
& {\rm Hom}({\bb W},{\bb W'})\ar[r]
&{\rm Ext}^1({\bb W},{\bb W'}_1)\ar[r]
&{\rm Ext}^1({\bb W},{\bb W'}_2)
}$$
with all vertical arrows isomorphisms by Proposition~\ref{iso1a} except the middle one.
It follows, thanks to the 5 Lemma, that the middle
arrow is also an isomorphism.
\end{proof}

\begin{corollary}\label{iso5}
Let ${\bb W}'\to{\bb W}''$ be a morphism of TVS's where ${\bb W}',{\bb W}''$ are qBCs of the same type.
Then ${\bb W}'_{>0}$ maps to ${\bb W}''_{>0}$, ${\bb W}'_{\geq0}$ maps to ${\bb W}''_{\geq0}$,
and the induced map ${\bb W}'_{=0}\to{\bb W}''_{=0}$ is $\Bdr^+$-linear.
\end{corollary}
\begin{proof}
The first part of the claim is proved in the same way as Proposition~\ref{filt}.
The last part follows from Theorem~\ref{iso2}.
%
%
\end{proof}
 
\section{Filtered $(\varphi,N)$-modules}\label{FN}
In this chapter we study  filtered $(\varphi,N,\sg_K)$-modules
over $K$ or $C$ and their relations to the categories of almost $C$-representations
and BC's. In particular, 
we introduce the notion of acyclic $(\phi,N,\G_K)$-modules 
as a generalization of weakly-admissible $(\phi,N,\G_K)$-modules. 
While the $(\phi,N,\G_K)$-modules $(H^i_{\rm HK}(X_C),H^i_{\rm dR}(X))$ 
coming from algebraic geometry 
tend to be weakly admissible those coming from overconvergent geometry tend to  
be only acyclic (as shown later in this paper).

The results of this chapter will be crucial for the proofs of our results towards
the $C_{\rm st}$-conjecture. In particular, Theorem~\ref{new-tate2} and its corollaries
(resp.~Proposition~\ref{baco3.7}) 
will be used to study the pro-\'etale-to-de Rham part of the $C_{\rm st}$-conjecture
for varieties over $K$ (resp.~over $C$). 

\subsection{Filtered $(\varphi,N,\sg_K)$-modules over $K$} \label{FN1}

\subsubsection{Filtered $(\varphi,N,\G_K)$-modules}\ \label{FN2}

$\bullet$ {\it $(\varphi,N,\G_K)$-modules}.
A $(\varphi,N)$-module over $F$ or $F^{\rm nr}$
is a finite dimensional $F$-module or $F^{\rm nr}$-module  $M$ endowed with a bijective Frobenius
$\varphi:M\to M$, semilinear with respect to the absolute Frobenius on $F$ or $F^{\rm nr}$, and a 
linear map $N:M\to M$ satisfying $N\varphi=p\,\varphi N$.

More generally,  a $(\varphi,N,\G_K)$-module over $F^{\rm nr}$ 
is a $(\varphi,N)$-module over $F^{\rm nr}$ endowed
with a smooth\footnote{This means that the stabilizers of elements of $M$ are
open in $\G_K$.} semilinear action of $\G_K$ which commutes with $\varphi$ and $N$.

If $M$ is a $(\varphi,N,\G_K)$-module over $F$ or $F^{\rm nr}$, we define its dual $M^\dual$
as ${\rm Hom}_F(M,F)$ endowed with actions of 
$\varphi$, $N$ and $\G_K$ given by
$$\langle\varphi(\mu),v\rangle=\varphi(\langle\mu,\varphi^{-1}(v)\rangle),
\quad \langle N(\mu),v\rangle=-\langle\mu,N(v)\rangle,\quad
\langle\sigma(\mu),v\rangle=\sigma(\langle\mu,\sigma^{-1}(v)\rangle),
\ {\text{if $\sigma\in\G_K$.}}$$

$\bullet$ {\it Filtered modules}.
A filtered module $(M,{\rm Fil}^\bullet)$ over $K$ is a $K$-module $M$
together with a descending filtration ${\rm Fil}^\bullet$ on $M$ by sub-$K$-modules
${\rm Fil}^iM$, with ${\rm Fil}^iM=M$ if $i\ll 0$
and ${\rm Fil}^iM=0$ if $i\gg 0$.

If $(M, {\rm Fil}^\bullet)$ is a filtered module over $K$,
we define the dual filtered module $(M^\dual, {\rm Fil}_\perp^\bullet)$ 
by endowing 
the $K$-dual $M^\dual={\rm Hom}_K(M,K)$ of $M$, with the filtration
$${\rm Fil}_\perp^iM^\dual=({\rm Fil}^{1-i}M)^{\perp}.$$ 
If the filtration is obvious from the context, we don't indicate it in the notation; for example,
the de Rham cohomology of a variety $X$ over $K$ is a filtered module over $K$ if we
endow it with the Hodge filtration, and will just be denoted by $H^\bullet_{\rm dR}(X)$,
the Hodge filtration being taken for granted.

$\bullet$ {\it Filtered $(\varphi,N,\G_K)$-modules}.
A filtered
$(\varphi,N)$-module $(M,{\rm Fil}^\bullet)$ over $K$ 
is a $(\varphi,N)$-module $M$ over $F$ with
a structure of filtered module over $K$ on $M_K=M\otimes_F K$.

A filtered $(\varphi,N,\G_K)$-module $(M,{\rm Fil}^\bullet)$ over $K$ 
is a $(\varphi,N,\G_K)$-module $M$ over $F^{\rm nr}$ with
a structure of filtered module over $K$ on $M_K=(M\otimes_{F^{\rm nr}}\overline K)^{\G_K}$.

If $(M,{\rm Fil}^\bullet)$ is a filtered $(\varphi,N,\G_K)$-module over $K$,
we define its dual $(M^\dual, {\rm Fil}_\perp^\bullet)$ as the
dual $(\varphi,N,\G_K)$-module $M^\dual$ with the module
$M_K^\dual=(M_K)^\dual$ endowed with the filtration ${\rm Fil}_\perp^\bullet$.

As before, if the filtration is obvious from the context, we will use sometimes
just $M$ to denote a filtered
$(\varphi,N)$-module $(M,{\rm Fil}^\bullet)$ over $K$ or sometimes
$(M,M_K)$ as in the case of de Rham cohomology:
if $X$ is a smooth quasi-compact dagger variety over $K$ then $(H^i_{\rm HK}(X_C),H^i_{\rm dR}(X))$ is a
filtered $(\varphi,N,\G_K)$-module over $K$ thanks to the Hyodo-Kato isomorphism.

\subsubsection{Acyclicity and admissibility}\label{FN3}
If $M$ is a filtered $(\varphi,N,\G_K)$-module over $K$, {\it the rank ${\rm rk}(M)$
of  $M$} is the dimension of $M$ over $F^{\rm nr}$.
If $M$ has rank~$1$, one
defines the {\it degree $\deg(M)$ of $M$} by the formula
$$\deg( M):=t_N(M)-t_H(M),$$
where $t_N(M)$ et $t_H(M)$ are defined by choosing a basis $e$ of $M$ over $F^{\rm nr}$:

\quad $\bullet$ there exists
$\lambda\in (F^{\rm nr})^\dual$ such that $\varphi(e)=\lambda e$, and we set $t_N(M)=v_p(\lambda)$;

\quad $\bullet$
there exists $i\in\Z$, unique, such that $e\in M_K^i-M_K^{i+1}$, and we set
$t_H(M)=i$.

If $M$ has rank $r\geq 2$, then $\det M=\wedge^rM$ is of rank~$1$, and one defines 
{\it the degree of $M$} by:
\begin{align*}
\deg(M):=&\ \deg(\det M)=
t_N(M)-t_H(M),\\
t_N(M):=t_N(\det(M)),\quad&
t_H(M):=t_H(\det M)=\sum_{i\in\Z}i\dim_KM_K^i/M_K^{i+1}.
\end{align*}
Endowed with the rank and degree functions,
the category of filtered $(\varphi,N,\G_K)$-modules over $K$ is a
 Harder-Narasimhan $\otimes$-category.

\begin{definition}
A filtered $(\varphi,N,\G_K)$-module over $K$ is said to be {\it weakly admissible}
if it is semi-stable of slope~$0$ (a reformulation~\cite{faltings}
of the original notion~\cite{Fo78}).
It is said to be {\it acyclic} if its Harder-Narasimhan slopes are~$\geq 0$.
\end{definition}
\begin{remark}
A weakly admissible filtered $(\varphi,N,\G_K)$-module is acyclic; conversely
an acyclic filtered $(\varphi,N,\G_K)$-module is weakly admissible if and only if
it is of degree $0$.
\end{remark}

\begin{lemma}\label{FN4}
The following conditions are equivalent for a filtered $(\varphi,N,\G_K)$-module 
$(M,{\rm Fil}^\bullet)$ over $K$:

{\rm (a)} $(M,{\rm Fil}^\bullet)$ is acyclic.

{\rm (b)} There exists a filtration ${\rm Fil}_1^\bullet$ on $M_K$ such that
${\rm Fil}_1^iM_K\subset {\rm Fil}^iM_K$ for all $i$, and $(M,{\rm Fil}_1^\bullet)$
is weakly admissible.
\end{lemma}
\begin{proof}
To prove (a)$\Rightarrow$(b),
it is enough to show that one can find a filtration 
such that ${\rm Fil}_1^iM_K\subset {\rm Fil}^iM_K$ for all $i$, there exist $i$
with ${\rm Fil}_1^iM_K\neq {\rm Fil}^iM_K$, and $(M,{\rm Fil}_1^\bullet)$ is acyclic.
Indeed the degree of $(M,{\rm Fil}_1^\bullet)$ is then strictly smaller than that
of $(M,{\rm Fil}^\bullet)$. Hence by repeating the process one ends up with
a filtration such that ${\rm Fil}_n^iM_K\subset {\rm Fil}^iM_K$ for all $i$, 
$(M,{\rm Fil}_n^\bullet)$ is acyclic and of degree $0$, hence it  is of
Harder-Narasimhan slope $0$, i.e., it  is weakly admissible.

To construct such a ${\rm Fil}_1^\bullet$, let $M^1$ be the largest subobject
of Harder-Narasimhan slope $0$, and let $M^2$ be the quotient $M/M_1$.
Then $M^2$ has Harder-Narasimhan slope~$>0$, and if we pick up any filtration
${\rm Fil}_1^\bullet$ on $M^2_K$, such that ${\rm Fil}_1^iM^2_K={\rm Fil}^iM^2_K$ for
all $i$ except $i_0$, for which ${\rm Fil}^{i_0}M^2_K/{\rm Fil}^{i_0}_1M^2_K$ is of
dimension~$1$, then $(M^2,{\rm Fil}_1^\bullet)$ has 
Harder-Narasimhan slope~$\geq 0$ since the degree of any subobject 
has decreased by at most $1$ and hence is~$\geq 0$. 
Then defining ${\rm Fil}_1^iM_K$ as the
inverse image of ${\rm Fil}_1^iM^2_K$ in ${\rm Fil}^iM_K$ gives a filtration
with the desired properties.

The converse implication is obvious: the Harder-Narasimhan slope of $(M,{\rm Fil}^\bullet)$
is greater or equal to that of $(M,{\rm Fil}_1^\bullet)$.
\end{proof}

\begin{lemma}\label{FN4.1}
Let $(M,{\rm Fil}^\bullet)$ be an acyclic filtered $(\varphi,N,\G_K)$-module over $K$,
with $\varphi$-slopes in $[0,r]$, and ${\rm Fil}^0M_K=M_K$, ${\rm Fil}^{r+1}M_K=0$.
Then there exists a filtration ${\rm Fil}_1^\bullet$ on $M_K$
such that $(M,{\rm Fil}_1^\bullet)$ is weakly admissible and 
${\rm Fil}_1^iM_K\subset {\rm Fil}^iM_K$ for all $i$, 
${\rm Fil}_1^0M_K=M_K$, ${\rm Fil}_1^{r+1}M_K=0$.
\end{lemma}
\begin{proof}
In the proof of Lemma~\ref{FN4}, the $\varphi$-slopes of $M^2$ are in $[0,r]$, and since
$\deg(M^2)>0$, this implies that ${\rm Fil}^1M_K\neq 0$. Let ${i_0}$ be the largest integer
with ${\rm Fil}^{i_0}M_K\neq 0$; define ${\rm Fil}_1^\bullet$ by ${\rm Fil}_1^{i_0}M_K$
of codimension~$1$ in ${\rm Fil}^{i_0}M_K$, and ${\rm Fil}_1^iM_K={\rm Fil}^iM_K$ is $i\neq 0$.
Then, as in the proof of Lemma~\ref{FN4},
$(M,{\rm Fil}_1^\bullet)$ is acyclic,
${\rm Fil}_1^0M_K=M_K$, ${\rm Fil}_1^{r+1}M_K=0$, and $\deg(M,{\rm Fil}_1^\bullet)<
\deg(M,{\rm Fil}^\bullet)$. Iterating the process gives the wanted filtration.
\end{proof}

\subsubsection{The complex attached to a filtered $(\varphi,N,\sg_K)$-module}\label{FN5}
If $(M,{\rm Fil}^\bullet)$ is a filtered $(\varphi,N,\G_K)$-module over $K$, we set
$$X_{\rm st}(M,{\rm Fil}^\bullet):=(M\otimes_{F^{\rm nr}}\bst)^{N=0,\varphi=1},\quad
X_{\rm dR}(M,{\rm Fil}^\bullet):=(M_K\otimes_K\bdr)/{\rm Fil}^0(M_K\otimes_K\bdr).$$
Then $X_{\rm st}(M,{\rm Fil}^\bullet)$ and $X_{\rm dR}(M,{\rm Fil}^\bullet)$ are
inductive limits of the $X^{(r)}_{\rm st}(M,{\rm Fil}^\bullet)$ 
and $X^{(r)}_{\rm dR}(M,{\rm Fil}^\bullet)$ (defined by replacing $\bst$ and $\bdr$ by
$t^{-r}\bst^+$ and $t^{-r}\bdr^+$),
which are objects of ${\cal C}(\G_K)$.

The complex $X_{\rm st}(M,{\rm Fil}^\bullet)\to X_{\rm dR}(M,{\rm Fil}^\bullet)$
is called~\cite[\S\,5.3]{CCF} the ``fundamental complex associated to $M$''.
Its $H^0$ is denoted by $V_{\rm st}(M,{\rm Fil}^\bullet)$. Hence, we have
an exact sequence
$$0\to V_{\rm st}(M,{\rm Fil}^\bullet)\to X_{\rm st}(M,{\rm Fil}^\bullet)
\to X_{\rm dR}(M,{\rm Fil}^\bullet)$$

\begin{remark}\label{new-tate12}
(i) The cohomology of the fundamental complex is equal
to that of the complex
$X^{(r)}_{\rm st}(M,{\rm Fil}^\bullet)\to X^{(r)}_{\rm dR}(M,{\rm Fil}^\bullet)$, for
$r$ big enough. It follows that its cohomology groups are objects of ${\cal C}(\G_K)$.

(ii) The pair $(X_{\rm st}(M,{\rm Fil}^\bullet), X_{\rm dR}(M,{\rm Fil}^\bullet))$
is a $B$-pair in the sense of Berger~\cite{Be2,Be4}; attached to it is a $\G_K$-equivariant
vector bundle ${\cal E}(M,{\rm Fil}^\bullet)$ on the Fargues-Fontaine curve~\cite[\S\,10.1]{FF}
and the fundamental complex computes the cohomology (not the $\G_K$-equivariant)
of this vector bundle. Since the HN-slope of ${\cal E}(M,{\rm Fil}^\bullet)$ is
that of $(M,{\rm Fil}^\bullet)$, the vector bundle ${\cal E}(M,{\rm Fil}^\bullet)$ has vanishing
$H^1$ if and only if $(M,{\rm Fil}^\bullet)$ is acyclic (see~Theorem~\ref{HN4} and formulas~(\ref{HN4.1})). It follows that,
if $(M,{\rm Fil}^\bullet)$ is acyclic, we have an exact sequence
$$0\to V_{\rm st}(M,{\rm Fil}^\bullet)\to X_{\rm st}(M,{\rm Fil}^\bullet)
\to X_{\rm dR}(M,{\rm Fil}^\bullet)\to 0$$

(iii) If $(M,{\rm Fil}^\bullet)$ is acyclic, then $V_{\rm st}(M,{\rm Fil}^\bullet)$
is a finite dimensional $\Q_p$-vector space if and only if $(M,{\rm Fil}^\bullet)$ is
of slope $0$ (i.e.,~is weakly admissible). This implies it
{\it is admissible}: the natural maps give $\G_K$-equivariant isomorphisms
\begin{align*}
&V_{\rm st}(M,{\rm Fil}^\bullet)\otimes \bst\overset{\sim}{\to} M\otimes_F\bst&&{\text{
of $\bst$-modules commuting with $\varphi$ and $N$,}}\\
&V_{\rm st}(M,{\rm Fil}^\bullet)\otimes \bdr\overset{\sim}{\to} M_K\otimes_K\bdr&&{\text{
of filtered $\bdr$-modules,}}
\end{align*}
and $V:=V_{\rm st}(M,{\rm Fil}^\bullet)$ is a potentially semi-stable representation of $\G_K$, with
$D_{\rm dR}(V)=M_K$ and $D_{\rm pst}(V)=M$.

(iv) If $(M,{\rm Fil}^\bullet)$ is acyclic, but not admissible, the
natural map 
$$ V_{\rm st}(M,{\rm Fil}^\bullet)\otimes \bdr^+\to {\rm Fil}^0(M_K\otimes_K\bdr)$$
is not an isomorphism (the kernel is huge), but it is surjective. Indeed, one can pick a filtration
${\rm Fil}^\bullet_1$ on $M_K$, such that ${\rm Fil}^i_1M_K\subset {\rm Fil}^iM_K$
for all $i$, and $(M,{\rm Fil}_1^\bullet)$ is admissible.
Then we have an exact sequence
\begin{equation}\label{new-tate13}
0\to V_{\rm st}(M,{\rm Fil}_1^\bullet)\to V_{\rm st}(M,{\rm Fil}^\bullet)\to
{\rm Fil}^0(M_K\otimes_K\bdr)/{\rm Fil}_1^0(M_K\otimes_K\bdr)\to 0
\end{equation}
and $\bdr^+\cdot V_{\rm st}(M,{\rm Fil}^\bullet)$
contains $\bdr^+\cdot V_{\rm st}(M,{\rm Fil}_1^\bullet)$ which, 
by admissibility of $(M,{\rm Fil}_1^\bullet)$,
is equal to ${\rm Fil}_1^0(M_K\otimes_K\bdr)$ 
and the above exact sequence gives the desired surjectivity.
\end{remark}

\subsection{Filtered $(\varphi,N,\sg_K)$-modules and almost $C$-representations}\label{FN6}
\subsubsection{The functors $D_{\rm dR}^\dual$ and $D_{\rm st}^\dual$}\label{FN7}
The  definitions of the classical contravariant functors $D_{\rm dR}^\dual$,
 $D_{\rm st}^\dual$ and $D_{\rm pst}^\dual$ for finite dimensional $\Q_p$-representations
of $\G_K$ extend to objects of ${\cal C}(\G_K)$ contrarily to those
of the (more commonly used) covariant functors $D_{\rm dR}$,
 $D_{\rm st}$ and $D_{\rm pst}$.
If $W\in{\cal C}(\G_K)$, set 
$$D_{\rm st}^\dual(W):={\rm Hom}_{\G_K}(W,\bst),\quad
D_{\rm dR}^\dual(W):={\rm Hom}_{\G_K}(W,\bdr).$$
Then $D_{\rm st}^\dual(W)$ is a $(\varphi,N)$-module\footnote{Finite dimensionality comes from the
fact that ${\rm Hom}_{\G_K}(C,\bst)\hookrightarrow {\rm Hom}_{\G_K}(C,\bdr)$ and the last
group is $0$ by point (iv) of Proposition~\ref{twists}. This reduces finite dimensionality
for $W\in{\cal C}(\G_K)$ to the case of finite dimensional $\Q_p$-representations.}
 over $K$ (with
$\langle\varphi(\mu),v\rangle=\varphi(\langle\mu,v\rangle)$ and
$\langle N(\mu),v\rangle=N(\langle\mu,v\rangle)$) 
and $D_{\rm dR}^\dual(W)$ is a filtered module over $K$ (with
${\rm Fil}^iD_{\rm dR}^\dual(W)=\{\mu,\ \mu(W)\subset t^i\bdr^+\}$).
We also define $D_{\rm pst}^\dual(W)$ as:
$$D_{\rm pst}^\dual(W):={\rm Hom}_{\G_K}^{\rm sm}(W,\bst):=\varinjlim_{[L:K]<\infty}
{\rm Hom}_{\G_L}(W,\bst)$$
This is a $(\varphi,N,\G_K)$-module over $F^{\rm nr}$.
\begin{remark}\label{new-tate10}
The natural maps 
$$D_{\rm st}^\dual(W)\otimes_FK\to D_{\rm dR}^\dual(W),
\quad (D_{\rm pst}^\dual(W)\otimes_{F^{\rm nr}}\overline K)^{\G_K}\to D_{\rm dR}^\dual(W)$$
 induced
by the injections $\bst\otimes_FK\hookrightarrow\bdr$ 
and $\bst\otimes_{F^{\rm nr}}\overline K\hookrightarrow\bdr$ are
injective.
\end{remark}

The following result is an extension to acyclic filtered $(\varphi,N)$-modules
of a classical result for admissible filtered $(\varphi,N)$-modules.

\begin{theorem}\label{new-tate2}
If $(M,{\rm Fil}^\bullet)$ is an acyclic filtered $(\varphi,N)$-module over $K$,
then
\begin{align*}
&D_{\rm st}^\dual(V_{\rm st}(M,{\rm Fil}^\bullet))
\simeq M^\dual&&{\text{as $(\varphi,N)$-modules over $F$,}}\\
&D_{\rm dR}^\dual(V_{\rm st}(M,{\rm Fil}^\bullet))
\simeq (M_K^\dual,{\rm Fil}_\perp^\bullet) &&{\text{as filtered $K$-modules.}}
\end{align*}
\end{theorem}
\begin{proof}
Let $V:=V_{\rm st}(M,{\rm Fil}^\bullet)$.
We start by noticing that, thanks to point (iv) of Remark~\ref{new-tate12}, the natural pairing
injects $M^\dual$ into ${\rm Hom}_{\G_K}(V,\bst)$ (as a $(\varphi,N)$-module)
and  $ M_K^\dual$ into ${\rm Hom}_{\G_K}(V,\bdr)$. 
Moreover, the natural map from  $K\otimes_{F}{\rm Hom}_{\G_K}(V,\bst)$
to  ${\rm Hom}_{\sg_K}(V,\bdr)$ is injective (see Remark~\ref{new-tate10}),
and one can deduce that $D_{\rm st}^\dual\to M^\dual$ 
is an isomorphism from the same result for $D_{\rm dR}^\dual$.

Hence, we just have to prove the result for $D_{\rm dR}^\dual$. To do so,
pick  a filtration ${\rm Fil}_1^\bullet$ as in (iv) of Remark~\ref{new-tate12}:
${\rm Fil}^i_1M_K\subset {\rm Fil}^iM_K$ for all $i$ and $(M,{\rm Fil}_1^\bullet)$
is admissible. Let $V:=V_{\rm st}(M,{\rm Fil}^\bullet)$, $V_1:=V_{\rm st}(M,{\rm Fil}_1^\bullet)$,
and $W={\rm Fil}^0(M_K\otimes_K\bdr)/{\rm Fil}_1^0(M_K\otimes_K\bdr)$, so that
the exact sequence~(\ref{new-tate13}) becomes $0\to V_1\to V\to W\to 0$.

Now, $W$ is a direct sum
of factors of the form $t^{k_1}\bdr^+/t^{k_2}\bdr^+$. Hence
$D_{\rm dR}^\dual(W)=0$ thanks to Proposition~\ref{twists},
and we get injections 
$M_K^\dual\hookrightarrow D_{\rm dR}^\dual(V)
\hookrightarrow D_{\rm dR}^\dual(V_1)$.
But $\dim_KD_{\rm dR}^\dual(V_1)\leq\dim_{\Q_p}V_1=\dim_K M_K$ (the first inequality is true
for any finite dimensional $\Q_p$-representation of $\G_K$ and the second equality is true because
$V_1$ is de Rham).
It follows that these injections are in fact isomorphisms, which proves what we
want except for the equality of the filtrations on $M_K^\dual$ and $D^\dual_{\rm dR}(V)$.

So let $\mu\in M_K^\dual$. One can extend $\mu$ to a $\bdr$-linear map
$\bdr\otimes_KM_K\to\bdr$.  Thanks to (iv) of Remark~\ref{new-tate12},
one sees that $\mu(V)\subset t^i\bdr^+$ is equivalent to
$\mu({\rm Fil}^0(M_K\otimes_K\bdr))\subset t^i\bdr^+$.
But ${\rm Fil}^0(M_K\otimes_K\bdr)=\sum_n {\rm Fil}^nM_K\otimes_K t^{-n}\bdr^+$,
and $\mu(M_K)\subset K$, hence $\mu({\rm Fil}^0(M_K\otimes_K\bdr))\subset t^i\bdr^+$
if and only if $\mu({\rm Fil}^nM_K)=0$ for $n\geq 1-i$.  By definition this
translates into $\mu\in {\rm Fil}^i_\perp M_K^\dual$, as wanted.
\end{proof}
\begin{corollary}\label{new-tate2.5}
If $(M,{\rm Fil}^\bullet)$ is an acyclic filtered $(\varphi,N,\G_K)$-module over $K$,
then
\begin{align*}
&D_{\rm pst}^\dual(V_{\rm st}(M,{\rm Fil}^\bullet))
\simeq M^\dual&&{\text{as $(\varphi,N,\G_K)$-modules over $F^{\rm nr}$,}}\\
&D_{\rm dR}^\dual(V_{\rm st}(M,{\rm Fil}^\bullet))
\simeq (M_K^\dual,{\rm Fil}_\perp^\bullet) &&{\text{as filtered $K$-modules.}}
\end{align*}
\end{corollary}
\begin{proof}
As in the proof of the above theorem, 
we have an injection of the right-hand sides into the left-hand sides,
and to check that these are isomorphisms, we just have to bound the dimensions of the left-hand sides.
This can be achieved by passing to a finite extension $L$ of $K$ such that the action of $\G_L$ on
$M$ is unramified, and applying the proposition to a $(\varphi,N)$-module $M(L)$ 
over $F_L$ such that $M=F^{\rm nr}\otimes_{F_L}M(L)$.
\end{proof}

\begin{corollary}\label{new-tate2.7}
Let $(M,{\rm Fil}^\bullet)$ 
be an acyclic filtered $(\varphi,N,\G_K)$-module over $K$, with $\varphi$-slopes in $[0,r]$,
and ${\rm Fil}^0M_K=M_K$, ${\rm Fil}^{r+1}M_K=0$.
Set $$V^r_{\rm st}(M,{\rm Fil}^\bullet)
:={\rm Ker}\big((M\otimes_{F^{\rm nr}}\bst^+)^{N=0,\varphi=p^r} 
\to (M_K\otimes_K\bdr^+)/{\rm Fil}^r\big)$$
Then\footnote{The notations $M\{r\}$ and ${\rm Fil}_\perp^\bullet\{r\}$
mean that the action of $\varphi$ is multiplied by $p^{r}$ and that the filtration
is shifted by $r$: we have ${\rm Fil}_\perp^i\{r\}={\rm Fil}_\perp^{i-r}$.}
\begin{align*}
&D_{\rm pst}^\dual(V^r_{\rm st}(M,{\rm Fil}^\bullet))
\simeq M^\dual\{r\}&&{\text{as $(\varphi,N,\G_K)$-modules over $F^{\rm nr}$,}}\\
&D_{\rm dR}^\dual(V^r_{\rm st}(M,{\rm Fil}^\bullet))
\simeq (M_K^\dual,{\rm Fil}_\perp^\bullet\{r\}) &&{\text{as filtered $K$-modules.}}
\end{align*}
\end{corollary}
\begin{proof}
The conditions imply that $V^r_{\rm st}(M,{\rm Fil}^\bullet)=t^rV_{\rm st}(M,{\rm Fil}^\bullet)$.
Hence $${\rm Hom}_{\G_K}(V^r_{\rm st}(M,{\rm Fil}^\bullet),{\bf B}_?)=
t^{r}{\rm Hom}_{\G_K}(V_{\rm st}(M,{\rm Fil}^\bullet),{\bf B}_?)$$
 and the result follows.
\end{proof}
\begin{example}\label{new-tate23}
If $M$ in Corollary~\ref{new-tate2.7} satisfies ${\rm Fil}^rM_K=M_K$,
then $(M_K\otimes_K\bdr^+)/{\rm Fil}^r=0$. Hence 
$$V^r_{\rm st}(M,{\rm Fil}^\bullet)=
X^r_{\rm st}(M):=(M\otimes_{F^{\rm nr}}\bst^+)^{N=0,\varphi=p^r}$$
The corollary then becomes:
\begin{align*}
&{\rm Hom}_{\G_K}^{\rm sm}(X^r_{\rm st}(M),\bst)\simeq  M^\dual, \quad 
{\text{as $(\varphi,N,\G_K)$-modules over $F^{\rm nr}$,}}\\
&{\rm Hom}_{\G_K}(X^r_{\rm st}(M),t^j\bdr^+)\simeq 
\begin{cases} M_K &{\text{if $j\leq 0$,}}\\ 0 &{\text{if $j\geq 1$.}} \end{cases}
\end{align*} 
\end{example}
  
\begin{lemma}\label{kicius1}
Under the hypothesis of Corollary~\ref{new-tate2.7},
we have 
${\rm Hom}_{\G_K}(V_{\rm st}^r(M,{\rm Fil}^\bullet),C(j))=0$, for  $j\geq r+1$.
\end{lemma}
\begin{proof}
We can write  $V:= V_{\rm st}^r(M,{\rm Fil}^\bullet)$ in the form 
$0\to V_1{\to} V\to W\to 0$, as in the proof of Theorem~\ref{new-tate2},
with  $V_1$ of dimension  ${\rm dim}(M)$, de Rham with Hodge-Tate weights in $[0,r]$,
 and  $W$ a finite type  $\bdr^+$-module, sum of
$t^a\bdr^+/t^b\bdr^+$'s, with $0\leq a\leq b\leq r+1$.

   We have an exact sequence
\begin{align*}
0\to {\rm Hom}_{\scc(\sg_K)}  (W,C(j))  \to {\rm Hom}_{\scc(\sg_K)}(V,C(j))\to
{\rm Hom}_{\scc(\sg_K)}(V_1,C(j))
 \end{align*}
Since $j\geq r+1$, by Lemma~\ref{fog1}, we have  
${\rm Hom}_{\scc(\sg_K)}(t^a\bdr^+/t^b\bdr^+,C(j))=0$ if $b\leq r+1$; hence
${\rm Hom}_{\scc(\sg_K)}  (W,C(j))=0$.
Since the  Hodge--Tate weights of 
$V_1$ are~$\leq r$, we also have ${\rm Hom}_{\scc(\sg_K)}(V,C(j))=0$.

This concludes the proof.
\end{proof}

\subsection{Filtered $(\varphi,N)$-modules over $C$}
\subsubsection{Filtered $(\varphi,N)$-modules and vector bundles}
A filtered $(\phi,N)$-module $(M,M^+_{\dr})$ over $C$
is a pair $(M,M^+_{\dr})$, where:
\begin{enumerate}
\item $M$ is a $(\varphi,N)$-module over $F^{\nr}$;
 \item  $M^+_{\dr}\subset M\otimes_{F^{\nr}}\B_{\dr}$ is a $\B^+_{\dr}$-lattice. 
\end{enumerate}

\vskip.2cm
To such an $(M,M^+_{\dr})$ one can attach a vector bundle ${\cal E}(M,M^+_{\dr})$ on $X_{\rm FF}$, characterized by
$$H^0(X_{\rm FF}\moins\{\infty\},{\cal E}(M,M^+_{\dr}))=(M\otimes_{F^{\rm nr}}\bst)^{N=0,\varphi=1},
\quad{\cal E}(M,M^+_{\dr})\otimes \widehat\O_{X_{\rm FF},\infty} =M_{\rm dR}^+.$$
We set
$${V}_{\rm st}(M,M^+_{\dr}):=H^0(X_{\rm FF},{\cal E}(M,M^+_{\dr}))=
{\rm Ker}\big((M\otimes_{F^{\rm nr}}\bst)^{N=0,\varphi=1}
\to (M\otimes_{F^{\nr}}\B_{\dr})/M^+_{\dr}\big)$$

\begin{definition}
$(M,M^+_{\dr})$ is {\it weakly admissible} if 
${\cal E}(M,M^+_{\dr})$ is semistable, of slope~$0$.
It is {\it acyclic} if the slopes
of ${\cal E}(M,M_{\dr}^+)$
are~$\geq 0$.
\end{definition}

\begin{remark}\label{quasi4.1}
(i) If $(M,{\rm Fil}^\bullet) $ 
is a weakly admissible (resp.~acyclic) filtered  $(\varphi,N)$-module over~$K$,
the induced filtered $(\phi, N)$-module $(M\otimes_FF^{\rm nr}, {\rm Fil}^0(M_K\otimes_{K}\bdr))$ 
over $C$ is weakly admissible (resp.~acyclic).

(ii) It follows from the classification of vector bundles on $X_{\rm FF}$ that
$(M,M_{\rm dR}^+)$ is weakly admissible
if and only if it is admissible (i.e., ${\cal E}(M,M_{\dr}^+)$ is 
a direct sum of  trivial line bundles $\O_{X_{\rm FF}}$).
This translates into the following: $(M,M^+_{\dr})$ is weakly admissible
if and only if ${V}_{\rm st}(M,M^+_{\dr})$ is finite dimensional over $\Q_p$ and
the sequence 
$$0\to {V}_{\rm st}(M,M^+_{\dr})\to (M\otimes_{F^{\nr}}\bst)^{N=0,\varphi=1}\to (M\otimes_{F^{\nr}}\bdr)/M^+_{\dr}\to 0$$
is exact.  Moreover, if this is the case, the triviality of ${\cal E}(M,M^+_{\dr})$
implies that the natural maps
$${V}_{\rm st}(M,M^+_{\dr})\otimes_{\Q_p}\bst\to M\otimes_{F^{\nr}}\bst,\quad
{V}_{\rm st}(M,M^+_{\dr})\otimes_{\Q_p}\bdr^+\to M^+_{\dr}$$
are isomorphisms (for the second map this is also equivalent to 
the map ${V}_{\rm st}(M,M^+_{\dr})\otimes_{\Q_p}\bdr\to M\otimes_{F^{\nr}}\bdr$
being  a filtered isomorphism). 

(iii) The following conditions are equivalent:

\quad $\bullet$ $(M,M_{\dr}^+)$ is acyclic,

\quad $\bullet$
$H^1(X,{\cal E}(M,M_{\dr}^+))=0$ (and hence  ${\cal E}(M,M_{\dr}^+)$ is acyclic),

\quad $\bullet$ $(M\otimes_{F^{\nr}}\bst)^{N=0,\varphi=1}\to (M\otimes_{F^{\nr}}\bdr)/M^+_{\dr}$ 
is surjective.

\quad $\bullet$ There exists a sub-$\bdr^+$ lattice $N_{\dr}^+\subset M_{\dr}^+$
such that $(M,N_{\dr}^+)$ is weakly admissible.

(The first two points are equivalent by Theorem~\ref{HN4} and formulas~(\ref{HN4.1}),
the second and the third are
equivalent because 
$$H^1(X,{\cal E}(M,M_{\dr}^+))={\rm Coker}\big(
(M\otimes_{F^{\nr}}\bst)^{N=0,\varphi=1}\to (M\otimes_{F^{\nr}}\bdr)/M^+_{\dr}\big)$$
 the first
and last points are equivalent by the same arguments as in Lemma~\ref{FN4}.)
\end{remark}

\begin{remark}\label{quasi6}
Since  $(M,M_{\dr}^+)\mapsto {\cal E}(M,M_{\dr}^+)$ commutes with tensor products, and since the
slope of  ${\cal E}\otimes{\cal E}'$ is equal to the sum of slopes of 
 ${\cal E}$ and  ${\cal E}'$, the tensor product of two acyclic filtered 
$\varphi$-modules  is again acyclic.
\end{remark}

\begin{remark}\label{quasi5}
Assume that the $\varphi$-slopes are in $[0,r]$ and 
$M\otimes t^r\bdr^+\subset t^rM_{\rm dR}^+\subset M\otimes \bdr^+$.  
Then the  following conditions are equivalent: 

$\bullet$ $(M,M_{\dr}^+)$ is {acyclic},

$\bullet$ $(M\otimes_{F^{\nr}}\bst)^{N=0,\varphi=p^r}\to (M\otimes_{F^{\nr}}\bdr)/t^rM^+_{\dr}$ is surjective,  

$\bullet$ for all $k\geq 0$, 
$(M\otimes_{F^{\nr}} t^{-k}\bst^+)^{N=0,\varphi=p^r}\to (M\otimes_{F^{\nr}} t^{-k}\bdr^+)/t^rM^+_{\dr}$ is  surjective.

This follows from the fact that $\frac{(M\otimes t^{-k}\bst^+)^{N=0,\varphi=p^r}}{(M\otimes t^{1-k}\bst^+)^{N=0,\varphi=p^r}}
\simeq  \frac{M\otimes t^{-k}\bdr^+}{M\otimes t^{1-k}\bdr^+}$, for $k\geq 1$,
as can be shown, for exemple, by a Dimension of BC argument (we have an injective natural map
from the left to the right that extends to a morphism of BC's of the same Dimension $(\dim M,0)$:
the computation of the Dimension of the term on the right is immediate, that of the term on the left
follows from the Dimension computations of~\cite[Ex.\,5.18]{CN1}; it follows that this is an isomorphism of BC's). 
For the same reasons,
for all $k\geq 0$,
$$t^r{V}_{\rm st}(M,M_{\dr}^+)={\rm Ker}\big((M\otimes_{F^{\nr}} t^{-k}\bst^+)^{N=0,\varphi=p^r}
\to (M\otimes_{F^{\nr}} t^{-k}\bdr^+)/t^rM^+_{\dr}\big)$$
\end{remark}

\subsection{Filtered $(\varphi,N)$-modules and BC's}
The following computations supply  key arguments in the proofs of our pro-\'etale-to-de Rham  comparison theorems in Chapter~\ref{SS7}.  
 
\subsubsection{$(\varphi,N)$-modules}\label{BACO5} 
The following results supply  key arguments 
in the proof of our main comparison theorem (Theorem~\ref{quasi100}).
\begin{lemma}\label{quasi121}
Let  $M$ be a $(\varphi,N)$-module over $F^{\rm nr}$ whose $\varphi$-slopes are in $[0,r]$.
Then the $\Q_p$-module $(M\otimes_{F^{\nr}}\bst^+)^{N=0,\varphi=p^r}$ generates the  $\bdr^+$-module
$M\otimes_{F^{\nr}}\bdr^+$.
\end{lemma}
\begin{proof}
Since  $\bdr^+$ is a local ring, with residue field $C$,
it suffices to show that 
$(M\otimes_{F^{\nr}}\bst^+)^{N=0,\varphi=p^r}$ generates the  $C$-module
$M\otimes_{F^{\nr}} C$.

Now,  the map $x\mapsto x-u\,Nx+\frac{u^2}{2!}\,N^2x-\frac{u^3}{3!}\,N^3x+\cdots$, 
for $u\in \B^+_{\st}$ maping to $\log([p^\flat]/p)$ in $\bdr^+$,  induces
a $\varphi$-equivariant isomorphism $M\otimes_{F^{\nr}}\bcris^+\to (M\otimes_{F^{\nr}}\bst^+)^{N=0}$.
Since  $u$ has image~$0$ in  $C$, we are reduced to proving that 
$(M\otimes_{F^{\nr}}\bcris^+)^{\varphi=p^r}$ generates the  $C$-module
$M\otimes_{F^{\nr}} C$.

Since $\bcris^+$ contains $W(k_C)$, the theorem of  Dieudonn\'e-Manin allows us to reduce to the case 
where $M$ is elementary with slopes  $\frac{a}{h}\leq r$, i.e.,~it is generated 
by  $e_1,\dots,e_h$ with  
$$\varphi(e_1)=e_2,\, \varphi(e_2)=e_3, \dots,\, \varphi(e_h)=p^ae_1.$$
The map  $x\mapsto xe_1+p^{-r}\varphi(x)e_2+\cdots+p^{-(h-1)r}\varphi^{h-1}(x)e_h$
induces an isomorphism of $U_{h,rh-a}:=(\bcris^+)^{\varphi^h=p^{rh-a}}$ with 
$(M\otimes_{F^{\nr}}\bcris^+)^{\varphi=p^r}$.  Hence we are reduced to proving that the image of 
$U_{h,rh-a}$ by  the map $x\mapsto(\theta(x),\theta(\varphi(x)),\dots,\theta(\varphi^{h-1}(x)))$
does not lie inside a proper sub-$C$-module of $C^h$.
Assume that it does. Then there exists $\lambda_0,\dots,\lambda_{h-1}\in C$, not all zero,
such that $\lambda_0\theta(x)+\lambda_1\theta(\varphi(x))+\cdots+\lambda_{h-1}\theta(\varphi^{h-1}(x))=0$
for all $x\in U_{h,rh-a}$.  In particular, 
one can apply this to  $\alpha x$, for  $\alpha\in\Q_{p^h}$.
Setting  $\mu_i(x)=\lambda_i\theta(\varphi^i(x))$, we get 
$\mu_0(x)\alpha+\mu_1(x)\varphi(\alpha)+\cdots+\mu_{h-1}(x)\varphi^{h-1}(\alpha)=0$,
for all $\alpha\in\Q_{p^h}$. Linear independence of characters implies that 
$\mu_i(x)=0$ for all $i$ and  $x$, and we get our contradiction.
\end{proof}

\begin{proposition}\label{new-baco3.5}
Let $M$ be a $(\varphi,N)$-module with $\varphi$-slopes in $[0,r]$, 
and let $${\mathbb X}_{\rm st}^r(M):=(M\otimes_{F^{\nr}}\Bst^+)^{N=0,\varphi=p^r}.$$
Then: 
$${\rm Hom}_{\rm VS}({\mathbb X}_{\rm st}^r(M),\Bdr)\simeq M^\dual\otimes_{F^{\rm nr}}\bdr,
\quad
{\rm Hom}_{\rm VS}({\mathbb X}_{\rm st}^r(M),\Bst)\simeq M^\dual\otimes_{F^{\rm nr}}\bst$$
\end{proposition}
\begin{proof}
${\mathbb X}_{\rm st}^r(M)$ is of curvature~$\leq 0$,
and the conditions on the slopes imply (\cite[Ex.\,5.18]{CN1}) that 
${\rm ht}({\mathbb X}_{\rm st}^r(M))={\rm rk}(M)$.
It follows from Proposition~\ref{baco3} that ${\rm rk}(h({\mathbb X}_{\rm st}^r(M)))={\rm rk}(M)$.
On the other hand, the inclusion $\Bst^+\hookrightarrow\Bdr$ induces a natural map 
$M^\dual\otimes_{F^{\rm nr}}\bdr\to h({\mathbb X}_{\rm st}^r(M))$. 
Lemma~\ref{quasi121} implies that this map is injective. Since the two modules have the same rank
over the field $\bdr$, this natural map is an isomorphism, which provides our first isomorphism.

For the second isomorphism, the injection of the right-hand side into the left-hand side is obvious (it follows, for example, from the first isomorphism).
To prove the converse inclusion, it is enough, granting the first isomorphism,
to show that if $\lambda\in M^\dual\otimes_{F^{\rm nr}}\bdr$ satisfies
$\lambda({X}_{\st}^r(M))\subset \bst$, then $\lambda \in M^\dual\otimes_{F^{\rm nr}}\bst$.
For this, pick a weakly admissible filtration $M_{\dr}^+$ on $M\otimes_{F^{\nr}}\bdr$, and
let $V:=V_{\st}(M,M_{\dr}^+)$ so that 
$V\otimes_{\Q_p}\bst\simeq M\otimes_{F^{\rm nr}}\bst$ (see ~Remark\,\ref{quasi4.1}).
Since $t^rV\subset X_{\st}^r(M)$, we have $\lambda(V)\subset\bst$, and
since $V$ generates $M\otimes_{F^{\rm nr}}\bst$, we have $\lambda(M)\subset \bst$.
This implies $\lambda \in M^\dual\otimes_{F^{\rm nr}}\bst$, as wanted.
\end{proof}

\begin{remark}\label{baco3.51}
The map $x\mapsto x-\frac{u}{1!}Nx+\frac{u^2}{2!}N^2x-\frac{u^3}{3!}N^3x+\cdots$
induces isomorphisms
\begin{align*}
{\mathbb X}_{\rm st}^r(M)\overset{\sim}{\to}
{\mathbb X}_{\rm cris}^r(M):=(M\otimes \Bcris^+)^{\varphi=p^r},\quad
{\rm Hom}_{\rm VS}({\mathbb X}_{\rm st}^r(M),\Bcris^+)\overset{\sim}{\to}
{\rm Hom}_{\rm VS}({\mathbb X}_{\rm cris}^r(M),\Bcris^+)
\end{align*}
Since $\Bst=\Bcris^+[u,\frac{1}{t}]$, we obtain:
\begin{align*}
{\rm Hom}_{\rm VS}({\mathbb X}_{\rm st}^r(M),\Bst)=
&{\rm Hom}_{\rm VS}({\mathbb X}_{\rm st}^r(M),\Bcris^+)\otimes_{\bcris^+}\bst\\
M^\dual\otimes_{F^{\nr}} \bcris^+\subset\ 
&{\rm Hom}_{\rm VS}({\mathbb X}_{\rm cris}^r(M),\Bcris^+)\subset
M^\dual\otimes_{F^{\nr}} t^{-r}\bcris^+
\end{align*}
(For the inclusion on the right, use the fact that $V$ in the proof
of Proposition~\ref{new-baco3.5} is included in $t^{-r}X^r_{\rm cris}(M)$
and $M\subset V\otimes\bcris^+$.)
\end{remark}

 \subsubsection{Filtered $(\phi,N)$-modules} 
\begin{proposition}\label{baco3.7}
Let $(M,M_{\rm dR}^+)$ be an acyclic filtered $(\varphi,N)$-module 
over $C$ with $\varphi$-slopes in $[0,r]$,
with $t^r\bdr^+\otimes M\subset t^rM_{\dr}^+\subset \bdr^+\otimes M$.
Set ${\mathbb M}_{\dr}^+:=\Bdr^+\otimes_{\bdr^+}M_{\dr}^+$ and
$${\mathbb V}_{\rm st}^r(M,M_{\rm dR}^+):=
{\rm Ker}((M\otimes_{F^{\nr}}\Bst)^{N=0,\varphi=p^r}
\to (M\otimes_{F^{\nr}}\Bdr^+)/t^r{\mathbb M}^+_{\dr}).$$
Then 
$${\rm Hom}_{\rm VS}({\mathbb V}_{\rm st}^r(M,M_{\rm dR}^+),\Bdr)
\simeq M^\dual\otimes_{F^{\rm nr}}\bdr,
\quad
{\rm Hom}_{\rm VS}({\mathbb V}_{\rm st}^r(M,M_{\rm dR}^+),\Bst)
\simeq M^\dual\otimes_{F^{\rm nr}}\bst$$
\end{proposition}
\begin{proof}
Let ${\mathbb X}_{\rm dR}^r(M):=(M\otimes_{F^{\nr}}\Bdr^+)/t^r{\mathbb M}^+_{\dr}$;
this is a $\Bdr^+$-module killed by $t^r$.
The hypotheses give us an exact sequence
$$0\to {\mathbb V}_{\rm st}^r(M,M_{\rm dR}^+)\to {\mathbb X}_{\rm st}^r(M)\to {\mathbb X}_{\rm dR}^r(M)\to 0$$
The first isomorphism is then a consequence of Proposition~\ref{new-baco3.5} and vanishing
of ${\rm Hom}_{\rm VS}({\mathbb X}_{\rm dR}^r(M),\Bdr)$ and 
${\rm Ext}^{1,\natural}_{\rm VS}({\mathbb X}_{\rm dR}^r(M),\Bdr)$
(Corollary~\ref{HN12} and Lemma~\ref{baco5}).

For the second isomorphism pick up a weakly admissible filtration $N_{\rm dR}^+$ containing
$M_{\rm dR}^+$ (this is possible because $(M,M_{\rm dR}^+)$ is acyclic),
set $V:={V}_{\rm st}(M,N_{\rm dR}^+)$, and argue as in Proposition~\ref{new-baco3.5}. 
\end{proof}
\begin{remark}\label{baco3.91}
As in Remark~\ref{baco3.51}, we have
\begin{align*}
{\rm Hom}_{\rm VS}({\mathbb V}_{\rm st}^r(M,M_{\dr}^+),\Bst)=
&{\rm Hom}_{\rm VS}({\mathbb V}_{\rm st}^r(M,M_{\dr}^+),\Bcris^+)\otimes_{\bcris^+}\bst\\
M^\dual\otimes_{F^{\nr}}\bcris^+\subset\ 
&{\rm Hom}_{\rm VS}({V}_{\rm st}^r(M,M_{\dr}^+),\Bcris^+)\subset
M^\dual\otimes_{F^{\nr}} t^{-r}\bcris^+
\end{align*}
\end{remark}

\begin{remark}
We could also have argued as in the proof of Theorem~\ref{new-tate2} 
to prove Propositions~\ref{new-baco3.5} and~\ref{baco3.7}.
\end{remark}

\begin{proposition}\label{baco3.9}
Let $(M,M_{\rm dR}^+)$ be an acyclic filtered $(\varphi,N)$-module
over $C$ with $\varphi$-slopes in $[0,r]$, and such that
$M\otimes_{F^{\nr}} t^r\bdr^+\subset t^rM_{\dr}^+\subset M\otimes_{F^{\nr}} \bdr^+$.
Let $k\geq 2r$. Then the natural map of $\bdr^+$-modules
$$M^\dual\otimes_{F^{\nr}}(\bdr^+/t^k)
\to {\rm Hom}_{\rm VS}({\mathbb V}_{\rm st}^r(M,M_{\rm dR}^+),\Bdr^+/t^k)$$
induced by the inclusion ${\mathbb V}_{\rm st}^r(M,M_{\rm dR}^+)\subset M\otimes\Bdr^+$,
has kernel and cokernel killed by $t^{2r}$.
\end{proposition}
\begin{proof}
Choose a $N_{\dr}^+$ with $M\otimes_{F^{\nr}}\bdr^+\subset N_{\dr}^+\subset M_{\dr}^+$
such that $(M,N_{\dr}^+)$ is weakly admissible.
(This is possible by an adaptation of Lemma~\ref{FN4.1} to $(\varphi,N)$-modules over $C$.)
Let $V_1={\mathbb V}_{\rm st}^r(M,N_{\rm dR}^+)$, $V={\mathbb V}_{\rm st}^r(M,M_{\rm dR}^+)$,
and $W=V/V_1$, so that $V_1$ is a finite dimensional $\Q_p$-vector space, 
and $W=t^r{\mathbb M}_{\dr}^+/t^r{\mathbb N}_{\dr}^+$ is a $\Bdr^+$-module killed by $t^r$.

Set $h_k(-):={\rm Hom}_{\rm VS}(-,\Bdr^+/t^k)$. 
Since (we skipped the subscripts $F^{\nr}$)
$$M\otimes \tfrac{t^r\Bdr^+}{t^k\Bdr^+}\subset \tfrac{t^r{\mathbb N}_{\dr}^+}{M\otimes t^k\Bdr^+}
\subset \tfrac{t^r{\mathbb M}_{\dr}^+}{M\otimes t^k\Bdr^+}\subset
M\otimes\tfrac{\Bdr^+}{t^k\Bdr^+}$$ 
and $V_1\subset \tfrac{t^r{\mathbb N}_{\dr}^+}{M\otimes t^k\Bdr^+},$
 $V\subset \tfrac{t^r{\mathbb M}_{\dr}^+}{M\otimes t^k\Bdr^+}$,
we have a commutative diagram
$$\xymatrix@R=4mm@C=6mm{
M^\dual\otimes (t^r\bdr^+/t^k)\ar[r] & h_k(\tfrac{t^r{\mathbb M}_{\dr}^+}{M\otimes t^k\Bdr^+})
\ar[r]\ar[d]&
h_k(\tfrac{t^r{\mathbb N}_{\dr}^+}{M\otimes t^k\Bdr^+})\ar[r]\ar[d] & M^\dual\otimes(t^r\bdr^+/t^k)\\
& h_k(V)\ar[r] &h_k(V_1)},$$
where the composed map $M^\dual\otimes (\bdr^+/t^k)\to M^\dual\otimes(t^r\bdr^+/t^k)$ is
multiplication by $t^r$. 
 Now, since $V_1$ is finite dimensional over $\Q_p$,
and ${\mathbb N}_{\dr}^+=V_1\otimes t^{-r}\Bdr^+$, we have, using Proposition~\ref{HN11}, 
$$h_k(\tfrac{t^r{\mathbb N}_{\dr}^+}{M\otimes t^k\Bdr^+})=
{\rm Hom}_{\Bdr^+}(\tfrac{t^r{\mathbb N}_{\dr}^+}{M\otimes t^k\Bdr^+},\Bdr^+/t^k)
=V_1^\dual\otimes(\bdr^+/t^k)=h_k(V_1).$$
It follows that:

$\bullet$  ${\rm Ker}\big(M^\dual\otimes(\bdr^+/t^k)\to h_k(V)\big)$ is a subobject
of ${\rm Ker}(M^\dual\otimes(\bdr^+/t^k)\overset{t^r}{\to}M^\dual\otimes(t^r\bdr^+/t^k))$,
hence is killed by $t^r$.

$\bullet$ ${\rm Coker}(h_k(V)\to h_k(V_1))$ is a subquotient
of ${\rm Coker}(M^\dual\otimes(\bdr^+/t^k)\overset{t^r}{\to}M^\dual\otimes(t^r\bdr^+/t^k))$,
and hence is killed by $t^r$.
The cokernel of $M^\dual\otimes(\bdr^+/t^k)\to h_k(V_1)$
is also killed by $t^r$ and, since the kernel of $h_k(V)\to h_k(V_1)$ is $h_k(W)$
which is killed by $t^r$, it follows that ${\rm Coker}(M^\dual\otimes(\bdr^+/t^k)\to h_k(V))$
is killed by $t^{2r}$.  

This concludes the proof.
\end{proof}

\section{Comparison theorems: examples and a conjecture}\label{SS2}
 In this Chapter we will formulate a conjecture: the existence of the fundamental diagram for smooth dagger varieties over $C$.  Before doing that though we will first look at  examples of comparison theorems and fundamental diagrams. 
\subsection{Cliffs Notes}Here we make a small digression  with a quick 
  review  of relevant results from \cite{CN3} and \cite{CN4}.  
 \begin{proposition}{\rm (Colmez-Nizio\l, \cite[Th.\,1.1]{CN3}, \cite[Th.\,1.3]{CN4})}
 \label{main00}
 \begin{enumerate}[label={\rm (\arabic*)}]
 \item {\bf Analytic varieties}:
 To any smooth dagger or rigid analytic  variety $X$ over $C$ there are  naturally  associated:
 \begin{enumerate}[label={\rm (\alph*)}]
 \item A pro-\'etale cohomology $\rg_{\proeet}(X,\Q_p(r))\in\sd(C_{\Q_p})$, $r\in \Z$.
 \item A syntomic cohomology $\R\Gamma_{\synt}(X,\Q_p(r))\in\sd(C_{\Q_p})$, $r\in\N$, with a natural  period morphism in $\in\sd(C_{\Q_p})$
    $$
  \alpha_r: \R\Gamma_{\synt}(X,\Q_p(r))\to \R\Gamma_{\proeet}(X,\Q_p(r)),
    $$
 which  is  a strict quasi-isomorphism after truncation $\tau_{\leq r}$. This morphism can be lifted to the derived category of Topological Vector Spaces. 
\item A Hyodo-Kato cohomology\footnote{We take here the Hyodo-Kato cohomology defined in \cite{CN4}.}
$\rg_{\hk}(X)\in \sd_{\phi,N}(F^{\rm nr})$.  We have   natural Hyodo-Kato strict quasi-isomorphisms in $\sd(C_C)$, $\sd(C_{\B^+_{\dr}})$, respectively: 
$$
 \iota_{\hk}: \rg_{\hk}(X)\wh{\otimes}^R_{F^{\nr}}C\stackrel{\sim}{\to} \rg_{\dr}(X),\quad  \iota_{\hk}: \rg_{\hk}(X)\wh{\otimes}^R_{F^{\nr}}\B^+_{\dr}\stackrel{\sim}{\to} \rg_{\dr}(X/\B^+_{\dr}).
 $$
  \item A  distinguished triangle in $\sd(C_{\Q_p})$
  $$
  \R\Gamma_{\synt}(X,\Q_p(r))\lomapr{} [\rg_{\hk}(X)\wh{\otimes}_{F^{\nr}}\B^+_{\st}]^{N=0,\phi=p^r}\lomapr{\iota_{\hk}} \rg_{\dr}(X/\B^+_{\dr})/F^r
  $$
  that can be lifted to the derived category of Topological Vector Spaces. 
  \item {\rm (Local-global compatibility)}
 In the case $X$ has a semistable weak formal model the above constructions are compatible with their analogs defined using the model. 
 \end{enumerate}
 \item {\bf Compatibility}: Let $X$ be a smooth dagger variety over $C$ and let $\wh{X}$ denote its completion. Then there exist  natural compatible morphisms \cite[Sec. 3.2.4]{CN3}
 \begin{align*}
  \iota_{\proeet} :\ &  \rg_{\proeet}(X,\Q_p(r))  \to \rg_{\proeet}(\wh{X},\Q_p(r)),\quad r\in \Z,\\
   \iota :\ & \rg_{\synt}(X,\Q_p(r))  \to \rg_{\synt}(\wh{X},\Q_p(r)),\quad r\in \Z,\\
&  \rg_{\dr}(X)\to \rg_{\dr}(\wh{X}),\quad \rg_{\dr}(X/\B^+_{\dr})\to \rg_{\dr}(\wh{X}/\B^+_{\dr}),\\
 & \rg_{\hk}(X)\to \rg_{\hk}(\wh{X}).
\end{align*}
They are strict quasi-isomorphisms, preserving all the extra structures, if $X$ is partially proper. 
\end{enumerate}
 \end{proposition}
\subsection{Proper rigid analytic  varieties} We start with smooth and proper varieties. 
\subsubsection{Algebraic varieties}
Let $X_K$ be an {\em algebraic} variety over $K$ and set $X=X_{K,\ovk}$. Recall the comparison theorem (recall  that all the cohomology groups involved have finite dimension):
\begin{theorem}\label{nerwy1}Let $r\geq 0$. 
There exists a natural $\bst$-linear Galois equivariant period isomorphism\footnote{Note that there is no assumption on the variety. This formulation of the comparison theorem is due to Beilinson~\cite{BE2}. }
$$
\alpha_{\st}: H^r_{\eet}(X,\Q_p)\otimes_{\Q_p}\bst \simeq H^r_{\hk}(X)\otimes_{F^{\nr}}\bst ,
$$
that preserves the Frobenius and the monodromy operators, and induces a filtered  isomorphism 
$$
\alpha_{\dr}: H^r_{\eet}(X,\Q_p)\otimes_{\Q_p}\bdr
 \simeq H^r_{\dr}(X_K)\otimes_{K}\bdr.
$$
\end{theorem}
In particular, we have the  natural isomorphism
$$
H^r_{\eet}(X,\Q_p)\simeq (H^r_{\hk}(X)\otimes_{F^{\nr}}\bst)^{\phi=1, N=0} \cap F^0(H^r_{\dr}(X_K)\otimes_{K}\bdr),\quad\text{as a  $\G_K$-module,}
$$
as well as the natural isomorphisms
\begin{align}
\label{iso11}
&{\rm Hom}^{\rm sm}_{\sg_K}(H^r_{\eet}(X,\Q_p),\bst)\simeq H^r_{{\rm HK}}(X)^\dual,\quad{\text{as a  $(\varphi,N,\sg_K)$-module,}}\\
&{\rm Hom}_{\sg_K}(H^r_{\eet}(X,\Q_p),\bdr)\simeq H^r_{{\rm dR}}(X_K)^\dual,\quad{\text{as a filtered $K$-module}}.\notag
\end{align}

\subsubsection{Proper rigid analytic varieties over $K$}
Let $X_K$ be a proper smooth rigid analytic variety over $K$. Let  $X=X_{K,C}$. The following result generalizes  \cite[Cor.\,1.10]{CN1}, where   semistable reduction was assumed. 
We note that all involved cohomology groups have finite dimension: for \'etale cohomology this is the result of Scholze \cite[Th.\,1.1]{Sch}; for Hyodo-Kato cohomology this follows from the Hyodo-Kato isomorphism and finiteness of de Rham cohomology. 
\begin{theorem} \label{nerwy11}Let $r\geq 0$. 
There exists a natural $\bst$-linear and $\sg_K$-equivariant period isomorphism
$$
\alpha_{\st}: H^r_{\eet}(X,\Q_p)\otimes_{\Q_p}\bst \simeq H^r_{\hk}(X)\otimes_{F^{\nr}}\bst ,
$$
that preserves the Frobenius and the monodromy operators, and induces a filtered   isomorphism 
$$
\alpha_{\dr}: H^r_{\eet}(X,\Q_p)\otimes_{\Q_p}\bdr
 \simeq H^r_{\dr,\ovk}(X)\otimes_{\ovk}\bdr.
$$
 \end{theorem}
In particular,  have the following natural  isomorphisms 
\begin{align}
\label{iso12}
&{\rm Hom}^{\rm sm}_{\G_K}(H^r_{\eet}(X,\Q_p),\bst)\simeq H^r_{{\rm HK}}(X)^\dual,\quad{\text{as a  $(\varphi,N,\sg_K)$-module,}}\\
&{\rm Hom}_{\G_K}(H^r_{\eet}(X,\Q_p),\bdr)\simeq H^r_{{\rm dR}}(X)^\dual,\quad{\text{as a filtered $K$-module}}.\notag
\end{align}
\begin{proof} Take $s > r$. To define the period maps consider the following composition
\begin{align*}
\alpha_{\st}(s): H^r_{\eet}(X,\Q_p(s))\xrightarrow[\sim]{\alpha_s^{-1}}H^r_{\synt}(X,s)\to (H^r_{\hk}(X)\otimes_{F^{\nr}}\bst^+)^{N=0,\phi=p^s}\lomapr{p^{-s}} H^r_{\hk}(X)\otimes_{F^{\nr}}\bst,
\end{align*}
where $\alpha_s$ is the period map from \cite[Cor.\,6.9]{CN4}. 
Set
$$
\alpha_{\st}:=t^{-s}\alpha_{\st}(s)\varepsilon^s, \quad \alpha_{\dr}:=\iota_{\rm BK}^{-1}(\iota_{\hk}\otimes\iota)\alpha_{\st},
$$
where $\varepsilon$ is the generator of $\Z_p(1)$ corresponding to $t$.

   We follow the proof of \cite[Cor.\,1.10]{CN1}, which uses
 BC's. We will sketch the arguments and refer the reader to \cite{CN1} for details. We start with stating the isomorphism
 $$
 (H^r_{\dr}(X)\otimes_{K}\B^+_{\dr})/F^s \stackrel{\sim}{\to} (H^r(\rg_{\dr}(X)\otimes_{K}\B^+_{\dr})/F^s), 
 $$
 which follows from the degeneration of the Hodge-de Rham spectral sequence\footnote{Alternatively, one can use an argument analogous to the one we use in the proof of Theorem~\ref{nerwy111} below.} (proved by Scholze as a corollary of  the de Rham comparison theorem \cite{Sch}). Using this isomorphism, we obtain 
  the long exact sequence (obtained from the definition of overconvergent syntomic cohomology)
\begin{align}
\label{long}
\to  (H^{r-1}_{\dr}(X)\otimes_{\ovk}\B^+_{\dr})/F^s & \to H^r_{\synt}(X,\Q_p(s))\to  (H^r_{\hk}(X)\otimes_{F^{\nr}}\B^+_{\st})^{N=0,\phi=p^s} \\
  &     \to (H^r_{\dr,\ovk}(X)\otimes_{\ovk}\B^+_{\dr})/F^s \to \notag
\end{align}

We recall that the groups $H^r_{\hk}(X), H^r_{\dr}(X)$ have finite dimension over $F^{\nr}$ and $K$, respectively. 
The above long exact sequence yields short exact sequences
$$
0\to H^r_{\synt}(X,\Q_p(s))\to  (H^r_{\hk}(X)\otimes_{F^{\nr}}\B^+_{\st})^{N=0, \phi=p^s}\to     (H^r_{\dr,\ovk}(X)\otimes_{\ovk}\B^+_{\dr})/F^s\to 0
$$
To prove this, we observe that the map
$f_r:(H^r_{\hk}(X)\otimes_{F^{\nr}}\B^+_{\st})^{N=0,\phi=p^s}\to     (H^r_{\dr,\ovk}(X)\otimes_{\ovk}\B^+_{\dr})/F^s$ is the evaluation on $C$ of a map
of BC's. But, the syntomic cohomology group $H^r_{\synt}(X_C,\Q_p(s))$, $r\leq s$, 
is a  finite dimensional $\Q_p$-vector space since 
we have the quasi-isomorphism \cite[6.10]{CN4} with \'etale cohomology.
This implies that the cokernel of $f_i$, viewed as a map of BC's,
 is of Dimension $(0,d_i)$.
On the other hand, the Space $(H^r_{\dr,\ovk}(X)\otimes_{\ovk}\B^+_{\dr})/F^s$ is a successive extension of $C$-vector spaces. 
The theory of BC's implies now that the map $(H^{r}_{\dr,\ovk}(X)\otimes_{\ovk}\B^+_{\dr})/F^s\to 
{\rm Coker}\,f_r$ is zero, hence ${\rm Coker}\,f_r=0$,
as wanted. 

  We have the Hyodo-Kato isomorphism (see \cite[Cor.\,4.32]{CN4}) 
  $$\iota_{\hk}:H^r_{\hk}(X)\otimes_{F^{\nr}}C\simeq H^r_{\dr}(X). $$
Taking $\sg_K$-smooth vectors of both sides (note that $X$ is quasi-compact) we get the Hyodo-Kato isomorphism
  $$H^r_{\hk}(X)\otimes_{F^{\nr}}
 \ovk\simeq H^r_{\dr}(X_K)\otimes_K\ovk.$$ Hence the pair $(H^r_{\hk}(X), H^r_{\dr}(X_K))$ 
is a $(\varphi,N,\sg_K)$-filtered module (in the sense of Fontaine). The above short exact 
sequence and a ``weight" argument shows that ${{\mathbf V}_{\st}}(H^r_{\hk}(X), H^i_{\dr}(X_K))\simeq H^r_{\eet}(X,\Q_p)$. Here ${{\mathbf V}_{\st}}(-)$ is Fontaine's functor from filtered Frobenius modules to Galois representations. 
 The short exact sequence and $C$-dimension count give also that $t_N(H^r_{\hk}(X))=t_H(H^r_{\dr}(X_K))$, where $t_N(D)=v_p(\det \phi)$ and $t_H(D)=\sum_{i\geq 0}i\dim_K( F^iD/F^{i+1}D)$. 
 The theory of BC's now implies that the pair $(H^r_{\hk}(X), H^r_{\dr}(X_K))$ is weakly admissible from which  our theorem follows. 
\end{proof}
\begin{remark}
The isomorphisms (\ref{iso11}) and (\ref{iso12}) are strict if we put the natural
 topology on the $\Hom$'s. 
   \end{remark}
\subsubsection{Proper rigid analytic varieties over $C$}\label{niziol2}
Let $X$ be a smooth and proper rigid analytic variety over $C$.
 Its $p$-adic \'etale cohomology is of finite rank by \cite[Th.\,1.1]{Sch}.
 Its Hyodo-Kato cohomology is of finite rank by the Hyodo-Kato isomorphism and finiteness of de Rham cohomology.
 Its $\B^+_{\dr}$-cohomology is free of finite rank over $\B^+_{\dr}$, by the comparison isomorphism with Hyodo-Kato cohomology.
\begin{theorem} \label{nerwy111}Let $r\geq 0$. 
There exists a natural $\bst$-linear period isomorphism
$$
\alpha_{\st}: H^r_{\eet}(X,\Q_p)\otimes_{\Q_p}\bst \simeq H^r_{\hk}(X)\otimes_{F^{\nr}}\bst ,
$$
that preserves the Frobenius and the monodromy operators, and induces a filtered   isomorphism 
\begin{equation}
\label{BMS1}
\alpha_{\dr}: H^r_{\eet}(X,\Q_p)\otimes_{\Q_p}\bdr
 \simeq H^r_{\dr}(X/\B^+_{\dr})\otimes_{\B^+_{\dr}}\bdr.
 \end{equation}
Here, the filtration on $H^r_{\dr}(X/\B^+_{\dr})$ is defined by $$F^iH^r_{\dr}(X/\B^+_{\dr}):=\im (H^r(F^i\rg_{\dr}(X/\B^+_{\dr}))\to H^r_{\dr}(X/\B^+_{\dr})).
$$
 \end{theorem}
 \begin{remark}
A de Rham comparison isomorphism as in (\ref{BMS1}) was constructed earlier in \cite[Th.\,13.1]{BMS1}. 
It did not treat filtrations.
 \end{remark}
\begin{proof} Take $s > r$. The period maps are defined as in Theorem~\ref{nerwy11} 
but by dropping the map $\iota_{\rm BK}$. We note that 
$H^r(\rg_{\dr}(X/\B^+_{\dr})/F^s)$ is a BC, a successive extension of $C$-vector spaces of finite rank. This follows from the fact that  the distinguished triangle 
\cite[3.28]{CN4} yields the distinguished triangle
$$
\bigoplus_{i\leq s}\rg(X,\Omega^i)(s-i)[-i]\to \rg_{\dr}(X/\B^+_{\dr})/F^{s+1}\to \rg_{\dr}(X/\B^+_{\dr})/F^s
$$
and $\rg_{\dr}(X/\B^+_{\dr})/F^1\simeq \rg(X,\so_X)$ by \cite[3.27]{CN4}.
Having that, the same arguments as in the case of proper varieties over $K$ yield short exact sequence   
\begin{equation}\label{wieczor1}
0\to H^r_{\eet}(X,\Q_p(s))\to (H^r_{\hk}(X)\otimes_{F^{\nr}}\bst^+)^{N=0,\varphi=p^s}\to H^r(\rg_{\dr}(X/\B^+_{\dr})/F^s)\to 0
\end{equation}

   Moreover, these arguments  show that the canonical map
$$
H^r_{\dr}(X/\B^+_{\dr})/F^s\to H^r(\rg_{\dr}(X/\B^+_{\dr})/F^s)$$
is an isomorphism.  This  can be proved  in the following way. 
 By  (\ref{wieczor1}),  we have a surjection
$$ (H^r_{\hk}(X)\otimes_{F^{\nr}}\bst^+)^{N=0,\varphi=p^s}\twoheadrightarrow  H^r(\rg_{\dr}(X/\B^+_{\dr})/F^s).
$$
Since the  above map 
 factors through  the natural map 
 \begin{equation}
\label{injective}
  H^r_{\dr}(X/\B^+_{\dr})/F^s\to H^r(\rg_{\dr}(X/\B^+_{\dr})/F^s)
\end{equation}
that latter is surjective as well. But it is also injective. Indeed, we have the distinguished triangle
$$
 F^s \rg_{\dr}(X/\B^+_{\dr})\to   \rg_{\dr}(X/\B^+_{\dr})\to  \rg_{\dr}(X/\B^+_{\dr})/F^s
$$
It yields the long exact sequence of cohomology groups
$$
\to  H^r(F^s\rg_{\dr}(X/\B^+_{\dr}))\stackrel{f_r}{\to}   H^r_{\dr}(X/\B^+_{\dr})\to   H^r(\rg_{\dr}(X/\B^+_{\dr})/F^s)\to 
$$
Since $F^sH^r_{\dr}(X/\B^+_{\dr})=\im f_r$, the map in (\ref{injective}) is injective. We are done.

The two isomorphisms in our theorem follow now from Remark~\ref{quasi4.1}, using the last part of Remark~\ref{quasi5} (take $M=H^r_{\hk}(X), M^+=F^0(H^r_{\dr}(X/\B^+_{\dr})\otimes_{\B^+_{\dr}}\B_{\dr})$).
\end{proof}

\begin{remark}
(i)
One can restate the theorem as follows (${\cal E}(-,-)$ is the associated vector bundle on 
$X_{\rm FF}$):
$$H^r_{\eet}(X,\Q_p)\simeq H^0(X_{\rm FF},{\cal E}(H^r_{\hk}(X),H^r_{\dr}(X/\B^+_{\dr}))).$$

(ii) From Theorem~\ref{nerwy111}, we get  natural  isomorphisms 
\begin{align*}
&{\rm Hom}_{\rm VS}(H^r_{\eet}(X,\Q_p),{\mathbb B}_{\st})
\simeq {\rm Hom}_{F^{\rm nr}}(H^r_{{\rm HK}}(X),\bst),\quad{\text{as $\bst$-modules,}}\\
&{\rm Hom}_{\rm VS}(H^r_{\eet}(X,\Q_p),{\mathbb B}_{\dr})\simeq 
{\rm Hom}_{\bdr^+}(H^r_{{\rm dR}}(X/\bdr^+),\bdr),\quad{\text{as $\bdr$-modules}}.
\end{align*}
\end{remark}

\subsection{Dagger Stein varieties and dagger  affinoids}
Having the comparison result from \cite[Cor.\,6.9]{CN4}, we can now deduce a (simplified) fundamental diagram for pro-\'etale cohomology from the one for overconvergent syntomic cohomology. 
\begin{theorem}{\rm (Simplified fundamental diagram)}\label{affinoids}
Let  $X$ be a smooth dagger Stein variety  or a smooth dagger  affinoid over $C$. 
 Let $r\geq 0$. There is  a natural map of strictly exact sequences in $\sd(C_{\Q_p})$
  $$
\xymatrix@C=.6cm@R=.6cm{
0\to  \Omega^{r-1}(X)/\kker d\ar[r]\ar@{=}[d] & H^r_{\proeet}(X,\Q_p(r))\ar[d]^{\tilde{\beta}} \ar[r] & (H^r_{\hk}(X)\wh{\otimes}_{F^{\nr}}\B^+_{\st})^{N=0,\phi=p^r}\ar[d]^{\iota_{\hk}\otimes\theta} \to 0\\
0\to  \Omega^{r-1}(X)/\kker d \ar[r]^-d & \Omega^r(X)^{d=0} \ar[r] & H^r_{\dr}(X)\longrightarrow 0
}
$$
Moreover, 
$H^r_{\proeet}(X,\Q_p(r))$ is Fr\'echet or LB, respectively, 
the vertical maps are strict and  have closed images, and
$ \kker\tilde{\beta}\simeq  (H^r_{\hk}(X)\wh{\otimes}_{F^{\nr}}\B^+_{\st})^{N=0,\phi=p^{r-1}}$. 
\end{theorem}
\begin{proof}
We define $\tilde{\beta}:=p^{-r}\beta\alpha_{r}^{-1}$, using \cite[Cor.\,6.9]{CN4} and   \cite[Prop.\,5.13]{CN4}; the twist by $p^{-r}$ being added to make this map compatible with symbols.
The theorem  follows immediately from \cite[Cor.\,6.9]{CN4} and 
 \cite[Prop.\,5.13]{CN4}.
\end{proof}

\subsection{The fundamental diagram} We will now introduce the fundamental diagram, look at some examples, where it appears, and, finally,  state a conjecture concerning it. 
\subsubsection{Examples}\label{propervar}
We start with examples. 
\vskip.2cm
\noindent $\bullet$
{\em Proper varieties.}
\begin{corollary}
Let $X$ be a smooth proper variety over $C$. 
We have the  bicartesian diagram 
 \begin{equation}
 \label{nerwy2}
 \xymatrix{
 H^r_{\eet}(X,\Q_p(r))\ar[r]^-{\alpha_{\st}(r)} \ar[d]^-{\alpha_{\dr}(r)} & (H^r_{\hk}(X)\otimes_{F^{\nr}}\bst^+)^{\phi=p^r, N=0}\ar[d]^-{\iota_{\hk}\otimes\iota} \\
 F^rH^r_{\dr}(X/\B^+_{\dr})\ar[r]^-{\can} & H^r_{\dr}(X/\B^+_{\dr})
 }
 \end{equation}
 \end{corollary}
 Recall that this means that this is a pushout and pullback diagram, or, that the  sequence
 \begin{equation}\label{wieczor2}
 0\to  H^r_{\eet}(X,\Q_p(r))\verylomapr{\alpha_{\st}(r)\oplus\alpha_{\dr}(r)}  (H^r_{\hk}(X)\otimes_{F^{\nr}}\bst^+)^{N=0,\phi=p^r}\oplus  F^rH^r_{\dr}(X/\B^+_{\dr})\verylomapr{\iota_{\hk}\otimes\iota+\can} 
 H^r_{\dr}(X/\B^+_{\dr})\to 0
 \end{equation}
 is exact. 
 \begin{proof} Follows immediately from the short exact sequence (\ref{wieczor1}).
  \end{proof}
 \begin{remark} (i) The passage 
  the other way, from diagram (\ref{nerwy2}) to Theorem~\ref{nerwy111}   is also possible: the exact sequence (\ref{wieczor2}) yields the exact sequence (\ref{wieczor1}) and we can finish as in the proof of Theorem~\ref{nerwy111}.
  
   (ii) The natural map $ H^r(F^r\rg_{\dr}(X/\B^+_{\dr})) \to F^rH^r_{\dr}(X/\B^+_{\dr})$ is an isomorphism. 
  \end{remark}

\noindent $\bullet$
{\em Dagger Stein varieties and dagger  affinoids.}
\begin{corollary}\label{affine11}Let $X$ be a smooth dagger Stein variety or a smooth dagger affinoid over $C$.
Let $r\geq 0$. We have 
the  bicartesian diagram in $\sd(C_{\Q_p})$  (recall that all cohomologies are classical) 
 \begin{equation}
 \label{main-diagram}
 \xymatrix{
 H^r_{\eet}(X,\Q_p(r))\ar[r]^-{\alpha_{\st}(r)} \ar[d]^-{\alpha_{\dr}(r)} & (H^r_{\hk}(X)\wh{\otimes}_{F^{\nr}}\bst^+)^{N=0,\phi=p^r}\ar[d]^-{\iota_{\hk}\otimes\iota} \\
 H^r F^r\rg_{\dr}(X/\B^+_{\dr})\ar[r] & H^r_{\dr}(X/\B^+_{\dr})
 }
\end{equation}
\end{corollary}
\begin{proof}
Consider the following diagram of maps of distinguished triangles in $\sd(C_{\Q_p})$
$$
\xymatrix{
\R\Gamma_{\synt}(X,\Q_p(r))\ar[r]^-{f_1}\ar@{.>}[d]^{{\beta}} & [\R\Gamma_{\hk}(X)\wh{\otimes}_{F^{\nr}}\B^+_{\st}]^{N=0,\phi=p^r}
\ar[d]^-{\iota_{\hk}\otimes\iota}\ar[r]^-{\iota_{\hk}\otimes\iota} & \R\Gamma_{\dr}(X/\B^+_{\dr})/F^r\ar@{=}[d]\\
F^r\R\Gamma_{\dr}(X/\B^+_{\dr})\ar[r]^-{f_2}&  \R\Gamma_{\dr}(X/\B^+_{\dr})\ar[r] & \R\Gamma_{\dr}(X/\B^+_{\dr})/F^r
}
$$
Since $H^i(\R\Gamma_{\dr}(X/\B^+_{\dr})/F^r)=0$ for $i\geq r$, the maps $f_1, f_2$ are surjective in degrees $\geq r$. It suffices to show that $\kker f_1\stackrel{\sim}{\to} \kker f_2$ in degree $r$. 
We have the following commutative diagram
\begin{equation*}
 \xymatrix@R=6mm{
0\ar[r] &  \kker f_1\ar[r] \ar[d] & H^r_{\synt}(X,\Q_p(r))\ar[r]^-{f_1} \ar[d] & (H^r_{\hk}(X)\wh{\otimes}_{F^{\nr}}\bst^+)^{N=0,\phi=p^r}\ar[d]^{\iota_{\hk}\otimes\iota} \\
0\ar[r] &  \kker f_2\ar[r] \ar[d]_{\wr}^{\vartheta}  & H^r F^r\rg_{\dr}(X/\bdr^+)\ar[r]^{f_2}\ar@{->>}[d]^{\vartheta} & H^r_{\dr}(X/\bdr^+)\ar@{->>}[d]^{\vartheta}\\
  0\ar[r] &  \kker f_{2,C}\ar[r] & H^r F^r\rg_{\dr}(X)\ar[r]^{f_{2,C}} & H^r_{\dr}(X)
 }
\end{equation*}
We claim that the map $ \kker f_2\to  \kker f_{2,C}$ is an isomorphism.
Indeed, we compute
\begin{align*}
& \kker f_2  \stackrel{\sim}{\leftarrow} \ccoker (H^{r-1}_{\dr}(X_K)\wotimes_K\B^+_{\dr}\to H^{r-1} ((\R\Gamma_{\dr}(X_K)\wh{\otimes}_K\B^+_{\dr})/F^r)) \stackrel{\sim}{\to} 
(\Omega^{r-1}(X_K)/\kker d) \wh{\otimes}_KC,\\
& \kker f_{2,C}  \stackrel{\sim}{\leftarrow}  (\Omega^{r-1}(X_K)/\kker d) \wh{\otimes}_KC.
\end{align*}
\end{proof}

\subsubsection{Conjecture} We will formulate now a conjecture 
describing  pro-\'etale cohomology in terms of the de Rham complex. 
\begin{conjecture}\label{quasi1} 
{\rm ($C_{\st}$-conjecture)}  
Let $X$ be a smooth dagger variety over $C$. 
Let $r\geq i$. The commutative diagram 
$$
\xymatrix@C=.4cm@R=.6cm{
\wt{H}^i_{\proeet}(X,\Q_p(r))\ar[r]\ar[d]
&(\wt{H}^i_{\rm HK}(X)\wotimes_{F^{\nr}}\B^+_{\st})^{N=0,\varphi=p^r}\ar[d]\\
\wt{H}^iF^r\rg_{\rm dR}(X/\B^+_{\dr})\ar[r]
& \wt{H}^i_{\rm dR}(X/\B^+_{\dr})
}
$$
is  bicartesian. That is, the following sequence
\begin{equation}\label{quasi19}
0\to \wt{H}^i_{\proeet}(X,\Q_p(r))\to (\wt{H}^i_{\rm HK}(X)\wotimes_{F^{\nr}}\B^+_{\st})^{N=0,\varphi=p^r}\oplus \wt{H}^iF^r\rg_{\rm dR}(X/\B^+_{\dr})\to \wt{H}^i_{\rm dR}(X/\B^+_{\dr})\to 0
\end{equation}
is exact in ${\rm LH}(C_{\Q_p})$.
\end{conjecture}
\begin{remark}\label{quasi2}
(i) This conjecture is known so far in the following cases:
\begin{itemize}
\item  $X$ is proper (see Examples~\ref{propervar}).  In this case,  the two horizontal arrows are injective.
\item  $X$ is Stein or affinoid (see Examples~\ref{propervar}).  In this case, the two horizontal arrows are surjective
and their kernels are $\Omega^{r-1}(X)/\kker \,d$.
\end{itemize}

(ii) Let $X$ be a smooth dagger variety over $C$. If $r\geq i$, set 
\begin{align*}
&\wt{H}^{r,i}:=\wt{H}^i_{\proeet}(X,\Q_p(r)) 
&&\wt{X}^{r,i}:=(\wt{H}^i_{\rm HK}(X)\wotimes_{F^{\nr}}\B^+_{\st})^{N=0,\varphi=p^r}\\
&\wt{F}^{r,i}:=\wt{H}^i(F^r\rg_{\rm dR}(X/\B^+_{\dr}))
&&\wt{B}^i:=\wt{H}^i_{\rm dR}(X/\B^+_{\dr})
\end{align*}
and denote by $H^{r,i}$, $X^{r,i}$, etc. the images of $\wt{H}^{r,i}$, $\wt{X}^{r,i}$, etc. in
$C_{\Q_p}$.
Note that the $\wt{X}^{r,i}$'s and $\wt{B}^i$'s are classical, i.e., $\wt{X}^{r,i}\simeq {X}^{r,i}$
and $\wt{B}^i\simeq B^i$.

We have a commutative diagram with exact rows  and columns:
$$\xymatrix@R=4mm@C=6mm{&0\ar[d]&0\ar[d]\\
\wt{H}^{r,i}\ar[r]\ar[d]^{\wr}& \wt{F}^{r,i}\oplus \wt{X}^{r,i}
\ar@<5mm>[d]_-t\ar@<-6mm>[d]^-{\wr}_-t\ar[r]&\wt{B}^i\ar[d]_-t\\
\wt{H}^{r+1,i}\ar[r]& \wt{F}^{r+1,i}\oplus \wt{X}^{r+1,i}
\ar@<5mm>[d]\ar[r]&\wt{B}^i\ar[d]\\
&{\phantom{XXX}H^i_{\rm HK}(X)\wotimes_{F^{\nr}} C}\ar@<5mm>[d]\ar[r]^-{\sim}_-{\iota_{\hk}}&H^i_{\rm dR}(X)\ar[d]\\
&{\phantom{XXX}0}&0}$$
The vertical maps are multiplications by $t$ (on pro-\'etale cohomology, this corresponds
to the Tate twist); for the isomorphism $X^{r+1,i}/tX^{r,i}\simeq H^i_{\rm HK}(X)\wotimes_{F^{\nr}} C$,
see Remark~\ref{quasi5}; the isomorphism $\wt{F}^{r+1,i}\simeq t\wt{F}^{r,i}$ comes from
the fact that $\tau_{\leq i}F^{r+1}\rg_{\rm dR}(X/\B^+_{\dr})=tF^r\rg_{\rm dR}(X/\B^+_{\dr})$ since $i\leq r$;
the bottom isomorphism is the Hyodo-Kato map.  It follows
that, for fixed $i$, the conjecture for $r$ is equivalent to the conjecture for $r+1$.
Hence it is enough to prove it for one $r\geq i$ (for example $r=i$).

(iii) Since the $\wt{B}^i$'s are actually classical, and we have a long exact sequence
$$\cdots\to \wt{X}^{r,i-1}\oplus \wt{F}^{r,i-1}\to \wt{B}^{i-1}\to
\wt{H}^{r,i}\to \wt{X}^{r,i}\oplus \wt{F}^{r,i}\to \wt{B}^i\to\cdots$$
it is enough, thanks to an induction on $i$, to prove surjectivity of
${X}^{r,i}\oplus {F}^{r,i}\to {B}^i$: this will show
that $\wt{X}^{r,i}\oplus \wt{F}^{r,i}\to \wt{B}^i$ is surjective (since $\wt{B}^i\simeq B^i$)
and that the long exact sequence splits into short exact sequences, as wanted.

(iv) We will prove this conjecture for many small varieties  in Theorem~\ref{quasi100} below. 
\end{remark}

\section{Equivalent conjectures}

Let $X$ be a smooth dagger variety over $C$.  In this Chapter, we will first state and discuss four conjectures, a priori unrelated, on the cohomology of  $X$:

\begin{itemize}[label=$\bullet$]
\item Conjecture~\ref{quasi1} (already stated above) describes the $p$-adic  pro-\'etale cohomology of $X$ in terms of the de Rham complex.
\end{itemize}

 The remaining conjectures assume $X$ to be quasi-compact: 
 
\begin{itemize}[label=$\bullet$]
\item Conjecture~\ref{quasi7.1} gives a restriction on the Hodge filtration on the de Rham cohomology in terms of the slopes of Frobenius on the Hyodo-Kato cohomology.

\item Conjecture~\ref{quasi10} says that,  even if huge, the pro-\'etale cohomology groups ${\mathbb H}^r_{\proeet}(X,\Q_p)$ have  nevertheless $\Q_p$-dimension equal to 
$\dim_C H^r_{\rm dR}(X)$.

\item Conjecture~\ref{pos1} says that  ${\mathbb H}^r_{\proeet}(X,\Q_p)$ is of curvature~$\leq 0$.
\end{itemize}

   Next, we proceed to studying these conjectures in  the case $X$ is quasi-compact. We  use: 

\begin{itemize}[label=---]

\item  the period quasi-isomorphism between pro-\'etale cohomology and syntomic cohomology 
from~\cite[6.10]{CN4}, 

\item the canonical  distinguished triangle involving syntomic cohomology from \cite[5.12]{CN4}, 

\item delicate properties of BC's (scattered through~\cite{CB,CF,lebras,plut}
and recalled in   Section~\ref{BACO} for the convenience of the reader).
\end{itemize}

These ingredients allow us to prove that the four conjectures above are, in fact, 
equivalent  (see~Proposition ~\ref{pos10.3}  as well as Corollary~\ref{pos20} 
for  a precise statement).

\subsection{Conjectures}
\subsubsection{Acyclicity of  de Rham cohomology} For a smooth dagger variety $X$ over $C$, we set 
\begin{enumerate}
\item $F^iH^r_{\rm dR}(X/\B^+_{\dr}):=\im(H^r(F^i\rg_{\dr}(X/\B^+_{\dr}))\to H^r_{\rm dR}(X/\B^+_{\dr}))$;
\item $H^r_{\rm dR}(X/\B_{\dr}):=
H^r_{\rm dR}(X/\B^+_{\dr})\wotimes_{\B^+_{\dr}}\B_{\dr}$ and equip it with the induced filtration. 
\end{enumerate}
\begin{conjecture}\label{quasi7.1}
Let  $X$ be a smooth quasi-compact dagger variety over $C$. For all $r$,  the map
$$(H^r_{\rm HK}(X)\wotimes_{F^{\nr}}\B^+_{\st})^{N=0,\varphi=p^r}\to H^r_{\rm dR}(X/\B^+_{\dr})/F^r$$
is surjective, i.e.\footnote{See Remark~\ref{quasi5}.}, 
the filtered $(\varphi,N)$-module $(H^r_{\rm HK}(X),F^0H^r_{\rm dR}(X/\B_{\dr}))$ is acyclic.
\end{conjecture}

\begin{remark}\label{quasi8} If  $X$ is proper, 
the pair $(H^r_{\rm HK}(X), F^0H^r_{\rm dR}(X/\B_{\dr}))$ is weakly admissible, hence 
$X$ verifies the conjecture.
If  $X$ is  affinoid, then $H^r_{\rm dR}(X/\B^+_{\dr})/F^r=0$
and  $X$ clearly satisfies the conjecture.
\end{remark}

\subsubsection{Curvature and height of pro-\'etale cohomology}  
Let $X$ be a smooth quasi-compact dagger variety over $C$.
We show below (Corollary~\ref{pos4.11}) that ${\mathbb H}^r_{\proeet}(X,\Q_p)$ is naturally a qBC.
\begin{conjecture}\label{pos1} Then, for all $r\geq 0$,
${\mathbb H}^r_{\proeet}(X,\Q_p)$ has curvature~$\leq 0$.
\end{conjecture}

\begin{conjecture}\label{quasi10}
For all~$r$,
$${\rm ht}({\mathbb H}^r_{\proeet}(X,\Q_p))=\dim_C H^r_{\rm dR}(X).$$
\end{conjecture}
\begin{remark}\label{quasi11}
(i) If $X$ is proper, Conjecture~\ref{quasi10} is a theorem: 
we have the exact sequence of BC's (see Section~\ref{niziol2})
$$
0\to H^r_{\proeet}(X,\Q_p(r))\to (H^r_{\hk}(X)\otimes_{F^{\nr}}\bst^+)^{N=0,\varphi=p^r}
\to H^r(\rg_{\dr}(X/\B^+_{\dr})/F^r)\to 0
$$
We know that
$H^r_{\proeet}(X,\Q_p)$ has finite dimension over $\Q_p$; hence  its height is  
equal to its dimension over 
$\Q_p$. We also know that the slopes of Frobenius on $H^r_{\hk}(X)$ are $\leq r$, 
which implies that ${\rm ht }((H^r_{\hk}(X)\otimes_{F^{\nr}}\bst^+)^{N=0,\varphi=p^r})
=\dim_{F^{\nr}}H^r_{\hk}(X)$. Now, 
the above short exact sequence implies that 
${\rm ht} ({\mathbb H}^r_{\proeet}(X,\Q_p)) =\dim_{F^{\nr}}H^r_{\hk}(X)$ and, 
since by the Hyodo-Kato isomorphism $\dim_{F^{\nr}}H^r_{\hk}(X)=\dim_C H^r_{\rm dR}(X)$,
 we have
$$\dim_{\Q_p}(H^r_{\proeet}(X,\Q_p))=\dim_C H^r_{\rm dR}(X),$$
as wanted.

(ii) Does there exist non proper dagger varieties $X$ such that 
$H^r_{\proeet}(X,\Q_p)$ is finite dimensional over $\Q_p$ for all $r$?
Already for $r=1$ and $X$ connected, one needs $\O(X)=C$.
\end{remark}

\subsection{Equivalence of conjectures}\label{EQUIV}
In this section we assume $X$ to be a smooth dagger variety over $C$ such that
$H^i_{\rm dR}(X/\B^+_{\dr})$ is (free) of finite rank over $\bdr^+$ for all $i$ 
(for example, $X$ could be quasi-compact,
or a naive interior\footnote{See \cite[proof of Cor. 3.6]{CN4} for a precise definition.} of a quasi-compact,
or the analytification of an algebraic variety). We call such varieties {\em small}. 
Conjectures~\ref{quasi1}, \ref{quasi7.1}, \ref{pos1}, and~\ref{quasi10} make sense in this, slightly more general, set-up. 
\subsubsection{The key diagram}\label{short-notation}
Fix $r$ and, for $i\leq r$, set
\begin{align*}
&H^{r,i}:= H^i_{\proeet}(X,\Q_p(r)),
&&X^{r,i}:=(H^{i}_{\rm HK}(X)\wotimes_{F^{\nr}}\B^+_{\st})^{N=0,\varphi=p^r},\\
&F^{r,i}:= H^i(F^r\rg_{\rm dR}(X/\B^+_{\dr})),
&&{\rm DR}^{r,i}:=H^i(\rg_{\rm dR}(X/\B^+_{\dr})/F^r),\\
&B^i:=H^i_{\rm dR}(X/\B^+_{\dr}) && {\rm Fil}^{r,i}:={\rm Im}(F^{r,i}\to B^i)
\end{align*}
We also denote by
$${\mathbb H}^{r,i},\quad {\mathbb X}^{r,i},\quad {\mathbb F}^{r,i},\quad {\mathbb {DR}}^{r,i},\quad {\mathbb B}^i,\quad \FIL^{r,i}$$
the associated TVS's.

The isomorphism between pro-\'etale and syntomic cohomologies (see~\cite[6.19]{CN4}) yields a commutative diagram with exact rows:
$$\xymatrix@C=.4cm@R=.4cm{H^{r,i-1}\ar[r]\ar[d]
&X^{r,i-1} \ar[r]\ar[d]&{\rm DR}^{r,i-1}\ar[r]\ar@{=}[d]
&H^{r,i}\ar[r]\ar[d]
&X^{r,i} \ar[r]\ar[d]
&{\rm DR}^{r,i}\ar@{=}[d]\\
F^{r,i-1}\ar[r]\ar[d]
&B^{i-1}\ar[r]\ar[d]&{\rm DR}^{r,i-1}\ar[r]\ar@{=}[d]
&F^{r,i}\ar[r]\ar[d]
&B^i\ar[r]\ar[d]
&{\rm DR}^{r,i}\ar@{=}[d]\\
F^{r,i-1}/t^r\ar[r]
&B^{i-1}/t^r{\rm Fil}^{r,i-1}\ar[r]&{\rm DR}^{r,i-1}\ar[r]
&F^{r,i}/t^r\ar[r]
&B^i/t^r{\rm Fil}^{r,i}\ar[r]
&{\rm DR}^{r,i}
}$$
The bottom sequence is exact because $t^rF^{r,i}\to t^r{\rm Fil}^{r,i}$ is an isomorphism
(multiplying by $t^r$ kills the big $\bdr^+$-torsion in $F^{r,i}$, and $B^i$
is a free finite rank $\B^+_{\dr}$-module thanks to the Hyodo-Kato isomorphism), hence we
 have an isomorphism:
\begin{equation}
\label{needed}
\kker(F^{r,i}\to F^{r,i}/t^r)\stackrel{\sim}{\to}\kker(B^i\to B^i/t^r{\rm Fil}^{r,i}).
\end{equation}
   All spaces in the diagram are $C$-points of TVS's and those 
in the top and the bottom rows are $C$-points of qBC's  (this is clear for all of them except for
$H^{r,i}$ (and $H^{r,i-1}$), for which this is proven in Corollary~\ref{pos4.11}),
and the above diagram lifts to a diagram of TVS's 
(see \cite[Sec. 7]{CN4}).

\subsubsection{The canonical filtration on ${\bb H}^{r,i}$}
If $n=0,1$, set:
$${\mathbbm h}^n(H^i_{\rm HK}\{-r\}):={\bb H}^n(X_{\rm FF},{\cal E}(H^i_{\rm HK}(X)\{-r\},
F^rH^i_{\rm dR}(X/\bdr))).$$

\begin{lemma}\label{pos4}
For all $i\leq r$, we have the following isomorphisms:
\begin{align*}
{\mathbbm h}^0(H^i_{\rm HK}\{-r\})&\simeq  {\rm Ker}\big({\bb X}^{r,i}\oplus\FIL^{r,i}\to {\bb B}^i\big)\simeq
{\rm Ker}\big({\bb X}^{r,i}\to {\bb B}^i/\FIL^{r,i}\big)\\
{\mathbbm h}^1(H^i_{\rm HK}\{-r\})&\simeq {\rm Coker}\big({\bb X}^{r,i}\oplus\FIL^{r,i}\to {\bb B}^i\big)\simeq
{\rm Coker}\big({\bb X}^{r,i}\to {\bb B}^i/\FIL^{r,i}\big)\\
{\rm Ker}\big({\bb H}^{r,i}\to {\bb X}^{r,i}\oplus {\bb F}^{r,i}\big) & \simeq
{\rm Coker}\big({\bb X}^{r,i-1}\oplus {\bb F}^{r,i-1}\to {\bb B}^{i-1}\big)
\simeq {\mathbbm h}^1(H^{i-1}_{\rm HK}\{-r\})
\end{align*}
\end{lemma}
\begin{proof} The first two isomorphisms are clear. 
Using  the snake lemma in the following commutative diagram with exact rows we get the last isomorphism:
\begin{equation}\label{wiesia101}\xymatrix@C=.4cm@R=.4cm{
0\ar[r]&{\bb X}^{r,i-1}/{\bb H}^{r,i-1}\ar[r]\ar[d]
&{\bb{DR}}^{r,i-1}\ar[r]\ar@{=}[d]
&{\rm Ker}\big({\bb H}^{r,i}\to {\bb X}^{r,i}\big) \ar[r]\ar[d]&0\\
0\ar[r]&{\bb B}^{i-1}/\FIL^{r,i-1}\ar[r]
&{\bb{DR}}^{r,i-1}\ar[r]&{\rm Ker}\big({\bb F}^{r,i}\to {\bb B}^i\big)\ar[r]
&0}
\end{equation}
\end{proof}

\begin{corollary}\label{pos4.11}
{\rm (i)} ${\bb H}^{r,i}$ is a qBC. Moreover:

$\bullet$ $({\bb H}^{r,i})_{>0}\simeq {\mathbbm h}^1(H^{i-1}_{\rm HK}\{-r\})$,

$\bullet$ $({\bb H}^{r,i})_{<0}\simeq {\mathbbm h}^0(H^{i}_{\rm HK}\{-r\})$.

$\bullet$ $({\bb H}^{r,i})_{=0}\simeq {\rm Ker}\big({\bb F}^{r,i}\to {\bb B}^i\big)=
({\bb F}^{r,i})_{\Bdr^+{\text{-tors}}}$.

{\rm (ii)} We have an exact sequence
$$0\to {\mathbbm h}^1(H^{i-1}_{\rm HK}\{-r\})\to {\bb H}^{r,i}\to
{\bb F}^{r,i}\oplus {\bb X}^{r,i}\to {\bb B}^i\to {\mathbbm h}^1(H^i_{\rm HK}\{-r\})\to 0$$
\end{corollary}
\begin{proof}By diagram \eqref{wiesia101}, 
${\bb H}^{r,i}$ has a natural filtration with successive quotients 
${\rm Coker}\big({\mathbb X}^{r,i-1}\to {\mathbb B}^{i-1}/\FIL^{r,i-1}\big)$, 
${\bb {DR}}^{r,i-1}/({\bb B}^{i-1}/{\mathbbm{Fil}}^{r,i-1})$ 
and ${\rm Ker}({\bb X}^{r,i}\to{\bb{DR}}^{r,i})$.

The first is equal to ${\mathbbm h}^1(H^{i-1}_{\rm HK}\{-r\})$, hence is 
a BC of curvature~$> 0$ being the ${\bb H}^1$ of a vector bundle on $X_{\rm FF}$
(one can also combine Proposition~\ref{lacompagnie} and Lemma~\ref{quasi121}); 
the last
is a BC of curvature $<0$ (as a sub-BC of a BC of curvature~$<0$), equal to 
${\mathbbm h}^0(H^{i}_{\rm HK}\{-r\})$ (the map ${\bb X}^{r,i}\to{\bb{DR}}^{r,i}$ factors
through ${\bb X}^{r,i}\to {\bb B}^i$);
the middle one is a $\Bdr^+$-Pair isomorphic to ${\rm Ker}\big({\bb F}^{r,i}\to {\bb B}^i\big)$
thanks to diagram \eqref{wiesia101}.  This proves (i).

(ii) follows from the last two points of Lemma~\ref{pos4}.
\end{proof}

\subsubsection{Equivalence of the conjectures}
\begin{proposition}\label{pos10.3} The following properties are equivalent:

\quad {\rm (a)} The diagram in Conjecture~\ref{quasi1} is bicartesian.

\quad {\rm (b)}
${\mathbbm h}^1(H^i_{\rm HK}\{-r\})=0$, for $i=r-1$ and $i=r$.

\quad {\rm (c)} $(H^{i}_{\rm HK}(X), F^0H^{i}_{\rm dR}(X/\B_{\dr}))$ 
is acyclic, for $i=r-1$ and $i=r$.

\quad {\rm (d)} 
${\mathbb H}^{r,r}$ and\footnote{Here ${\mathbb H}^{r,r+1}$ stands for syntomic
cohomology since we are outside of the range of validity of the comparison with pro-\'etale
cohomology.} ${\mathbb H}^{r,r+1}$ have curvature~$\leq 0$.

\quad {\rm (e)} ${\rm ht}({\mathbb H}^{r,r})=\dim_{C} H^r_{\rm dR}(X)$.
\end{proposition}
\begin{proof}
The equivalence between (a) and (b) is a direct consequence of (ii) of Corollary~\ref{pos4.11}.
That between (b) and (c) is the invariance of acyclicity by ``Tate twist''.

The equivalence between (b) and (d) is a direct consequence of Corollary~\ref{pos4.11}.

It remains to prove that (b) and (e) are equivalent. But, by definition
${\rm ht}({\mathbb H}^{r,r})={\rm ht}(({\mathbb H}^{r,r})^{>0})+{\rm ht}(({\mathbb H}^{r,r})^{<0})$,
and one gets from Corollary~\ref{pos4.11} and Remark~\ref{EP} that
\begin{align*}
{\rm ht}({\mathbb H}^{r,r})&={\rm ht}({\mathbbm h}^1(H^{r-1}_{\rm HK}\{-r\}))+
{\rm ht}({\mathbbm h}^0(H^r_{\rm HK}\{-r\}))\\
&={\rm ht}({\mathbbm h}^1(H^{r-1}_{\rm HK}\{-r\}))+\dim H^r_{\rm HK}(X)+
{\rm ht}({\mathbbm h}^1(H^{r}_{\rm HK}\{-r\}))
\end{align*}
and the equivalence between (b) and (e) follows from the Hyodo-Kato isomorphism and the fact that
${\rm ht}({\mathbbm h}^1(H^{i}_{\rm HK}\{-r\}))< 0$ if ${\mathbbm h}^1(H^{i}_{\rm HK}\{-r\})\neq 0$.
\end{proof}

\begin{corollary}\label{pos20}
If $X$ is quasi-compact,
Conjectures~\ref{quasi1}, \ref{quasi7.1}, \ref{pos1}, and~\ref{quasi10}
are equivalent. 
\end{corollary}
\begin{proof}
Equivalence of (a) and (e) from Proposition~\ref{pos10.3} shows that, 
already for a fixed  $r$,  Conjectures
~\ref{quasi1} and~\ref{quasi10} are equivalent. 
Equivalence between (a) and (c) show that they are equivalent to
Conjecture~\ref{quasi7.1} (but one has to vary $r$). And equivalence between (c) and (d),
together with the twist invariance of (ii) of Remark~\ref{quasi2} (to get back to a range
where syntomic is isomorphic to pro\'etale) show that they are equivalent 
to Conjecture~\ref{pos1}.
\end{proof}
\begin{remark}
\label{pos201}
It is easy to see that Corollary~\ref{pos20} holds, more generally, if  $H^i_{\dr}(X/\bdr^+)$ is 
of finite rank over $\bdr^+$ for all $i$. The proof is the same as in the quasi-compact case.
\end{remark}

\subsection{Decomposing syntomic cohomology}\label{Decomp}
The purpose of this section is to show that the canonical filtration
on ${\bb H}^{r,i}(X)$ (see Corollary~\ref{pos4.11} and its proof) splits (noncanonically) if $r\geq \dim X$. 
This result will not be used later. 
As remarked in Remark~\ref{nonsplit}, is it not impossible that such a splitting exists for
all qBC.

 In what follows, $\rg_{\rm cris}(X)$ is a functorial complex representing log-crystalline
cohomology of $X$ (built from an \'etale  hypercovering of $X$ by semi-stable formal schemes),
and $\rg_{\rm dR}(X/\bdr^+)$ is obtained from $\rg_{\rm cris}(X)$ by completion. Below we take the $\breve{C}$-version of Hyodo-Kato cohomology. 
\begin{lemma}\label{decomp1}
Let $X$ be a quasi-compact dagger variety over $C$ and let $r\geq\dim X$.

{\rm (i)} $H^i([\rg_{\rm cris}(X)]^{\varphi=p^r})
\simeq (H^i_{\rm HK}(X)\otimes\bst^+)^{N=0,\varphi=p^r}$ as BC's.

{\rm (ii)}
$[\rg_{\rm cris}(X)]^{\varphi=p^r}$ is quasi-isomorphic to
the direct sum of its cohomology groups.
\end{lemma}
\begin{proof}
Claim (i) is proved in~\cite[Chap.\,7]{CN4}. 
Recall that we have a (not Galois equivariant) quasi-isomorphism
$$\rg_{\rm cris}(X)\simeq \rg_{\rm HK}(X)\wh{\otimes}^{\R}_{\breve C}\bcris^+.$$
Fix $i$. Then $H^i_{\rm HK}(X)$ is a finite dimensional $\breve{C}$-vector space
with a semi-linear and bijective $\varphi$. Using Dieudonn\'e-Manin's theorem,
we can write $H^i_{\rm HK}(X)\simeq\oplus_{j\in J_i} D_{j}$ with $D_{j}$ irreducible,
of slope $\frac{a_j}{h_j}\leq \inf(i,\dim X)$, with $(a_j,h_j)=1$: i.e., $D_{j}=\breve{C}e_{j}\oplus
\breve{C}\varphi(e_{j})\oplus\cdots\oplus \breve{C}\varphi^{h_j-1}(e_{j})$ with
$\varphi^{h_j}(e_{j})=p^{a_j}e_j$. Since $r\geq\dim X$, we have
$$(D_{j}\otimes_{\breve{C}}\bcris^+)^{\varphi=p^r}=\{xe_j+p^{-r}\varphi(xe_j)+\cdots+p^{-r(h_j-1)}\varphi^{h_j-1}(xe_j)|
\varphi^{h_j}(x)=p^{rh_j-a_j}x\}$$
 i.e., $(D_{j}\otimes_{\breve{C}}\bcris^+)^{\varphi=p^r}\simeq
{\mathbb U}_{h_j,rh_j-a_j}e_j$.
Choose $\tilde e_j\in Z^i(\rg_{\rm cris}(X))$ whose image in $H^i_{\rm HK}(X)\otimes_{\breve{C}}\bcris^+$ is $e_j$.
If $\lambda\in {\mathbb U}_{h_j,rh_j-a_j}$, then 
$$(1-\tfrac{\varphi}{p^r})\big(\lambda \tilde e_j+p^{-r}\varphi(\lambda\tilde e_j)
+\cdots+p^{-r(h_j-1)}\varphi^{h_j-1}(\lambda\tilde e_j)\big)=\lambda(\tilde e_j-p^{-a_j}\varphi^{h_j}(\tilde e_j)).$$Since $\varphi^{h_j}(e_{j})=p^{a_j}e_j$, it follows that $\tilde e_j-p^{-a_j}\varphi^{h_j}(\tilde e_j)=d x_j$.
Hence $(\lambda_j)_j\mapsto (\lambda_j e_j,\lambda_j x_j)$ maps
$\oplus_{j \in J_j}{\mathbb U}_{h_j,rh_j-a_j}$ to a subcomplex of the total complex of the above mapping fiber,
with all arrows $0$ and all cohomology groups $0$ except $H^i$ which is, by construction,
indentified with $H^i([\rg_{\rm cris}(X)]^{\varphi=p^r})$.
The lemma follows.
\end{proof}
\begin{remark}\label{decomp2}
The arrow 
$$\oplus_i (H^i_{\rm HK}(X)\otimes_{\breve{C}}\bst^+)^{N=0,\varphi=p^r}[-i]\to 
[\xymatrix{\rg_{\rm cris}(X)\ar[r]^-{1-{\varphi}/{p^r}}
&\rg_{\rm cris}(X)}]$$
 that we constructed above depends on the choice of the $\tilde e_j$.
This arrow can be used to pin down a subspace of $Z^i(\rg_{\rm dR}(X/\bdr^+))$ isomorphic
to $H^i_{\rm dR}(X/\bdr^+)$, namely the $\bdr^+$-module generated by the images of
$\lambda \tilde e_j+p^{-r}\varphi(\lambda\tilde e_j)
+\cdots+p^{-r(h_j-1)}\varphi^{h_j-1}(\lambda\tilde e_j)$ for $\lambda\in \Q_{p^{h_i}}$
and $j\in J_i$.
\end{remark}

\begin{proposition}\label{decomp3}Let $r\geq \dim X$. 
We have a decomposition
$${\mathbb H}^i_{\rm syn}(X,r)\simeq {\mathbbm h}^1(H^{i-1}_{\rm HK}\{-r\})\oplus
{\mathbbm h}^0(H^{i}_{\rm HK}\{-r\}) \oplus (Y^{i-1},Z^i)$$
where
%
%
%
$(Y^{i-1},Z^i)$ is a (reasonable) $\Bdr^+$-Pair.
\end{proposition}
\begin{proof}
Let us write syntomic cohomology as the fiber
$$\rg_{\rm syn}(X,r)=\big[[\rg_{\rm cris}(X)]^{\varphi=p^r}
\oplus F^r\rg_{\rm dR}(X/\bdr^+)\to \rg_{\rm dR}(X/\bdr^+)\big]$$
and replace $[\rg_{\rm cris}(X)]^{\varphi=p^r}$ by the direct sum of its cohomology groups,
the induced map to $\rg_{\rm dR}(X/\bdr^+)$ being induced by the choice of the $\tilde e_j$ as above
(the induced map on cohomology does not depend on this choice).
 It follows that $H^i_{\rm syn}(X,r)$ is $H^1$ of the following complex
(where $\rg_{\rm dR}^s$ and $F^r\rg_{\rm dR}^s$ are the terms of degree $s$ in the respective complexes
and ${\rm HK}^s:=(H^s_{\rm HK}(X)\otimes\bst^+)^{N=0,\varphi=p^r}$)
$${\rm HK}^{i-1}\oplus F^r\rg_{\rm dR}^{i-1}\oplus\rg_{\rm dR}^{i-2}\lomapr{d^{i-1}_{\rm tot}}
{\rm HK}^{i}\oplus (F^r\rg_{\rm dR}^{i})^{d=0}\oplus\rg_{\rm dR}^{i-1}\lomapr{d^{i}_{\rm tot}}
(\rg_{\rm dR}^i)^{d=0}.$$
Now, $H^1$ does not change if we mod out by the complex $\rg_{\rm dR}^{i-2}\to {\rm Im}\,d$ to obtain the complex
\begin{equation}\label{splitting}{\rm HK}^{i-1}\oplus F^r\rg_{\rm dR}^{i-1}\lomapr{d^{i-1}_{\rm tot}}
{\rm HK}^{i}\oplus (F^r\rg_{\rm dR}^{i})^{d=0}\oplus\rg_{\rm dR}^{i-1}/({\rm Im} \,d)\lomapr{d^{i}_{\rm tot}}
(\rg_{\rm dR}^i)^{d=0}.
\end{equation}
But
$\rg_{\rm dR}^{i-1}/ ({\rm Im}\,d)$ is an extension of $\rg_{\rm dR}^{i-1}/ ({\rm Ker}\,d)$
by $H^{i-1}_{\rm dR}$ (where $H^s_{\rm dR}:=H^s_{\rm dR}(X/\bdr^+)$). Such an extension
is split $\B^+_{\dr}$-linearly  since $H^{i-1}_{\rm dR}$ is free of finite rank over $\B^+_{\dr}$ and $\rg_{\rm dR}^{i-1}/ ({\rm Ker}\,d)$ is
separated: pick a set of  continuous $\bdr^+$-linear forms on $\rg_{\rm dR}^{i-1}/ ({\rm Im}\,d)$ that
form a basis of $\bdr^+$-linear forms on $H^{i-1}_{\rm dR}$ (these exist because all the
spaces involved are of compact type hence reflexive), then the common kernel
of these linear forms is a supplementary subspace of $H^{i-1}_{\rm dR}$ that maps homeomorphically
to $\rg_{\rm dR}^{i-1}/ ({\rm Ker}\,d)$ by the open mapping theorem.
We can then mod out \eqref{splitting} by $0\to \rg_{\rm dR}^{i-1}/ ({\rm Ker}\,d)\to {\rm Im}\,d$
and obtain that $H^i_{\rm syn}(X,r)$ is $H^1$ of the complex:
$${\rm HK}^{i-1}\oplus F^r\rg_{\rm dR}^{i-1}\to
{\rm HK}^{i}\oplus (F^r\rg_{\rm dR}^{i})^{d=0}\oplus H_{\rm dR}^{i-1}
\to H^i_{\rm dR.}$$

  We have splittings ($F^rH^{i}_{\rm dR}$ and $F^rH^{i-1}_{\rm dR}$ are free of finite rank over $\B^+_{\dr}$,
so it is easy to lift  them)
\begin{align*}
(F^r\rg_{\rm dR}^{i})^{d=0}&\simeq F^rH^{i}_{\rm dR}\oplus \hat{Z}^i,
\quad{\text{where $\hat{Z}^i={\rm Ker}((F^r\rg_{\rm dR}^{i})^{d=0}\to H^{i}_{\rm dR})$}};\\
F^r\rg_{\rm dR}^{i-1}/{\rm Ker}\,d_{\rm tot}&\simeq F^rH^{i-1}_{\rm dR}\oplus \hat{Y}^{i-1},
\quad{\text{with $\hat{Y}^{i-1}=F^r\rg_{\rm dR}^{i-1}/{\rm Ker}\,d$}}.
\end{align*}
It follows that
$$H^i_{\rm syn}(X,r)\simeq (\hat{Y}^{i-1},\hat{Z}^i)\oplus {\rm Ker}({\rm HK}^{i}\oplus F^rH^{i}_{\rm dR}\to
H^{i}_{\rm dR})\oplus {\rm Coker}\,({\rm HK}^{i-1}\oplus F^rH^{i-1}_{\rm dR}\to H^{i-1}_{\rm dR})$$
The Ker (resp.~Coker) term gives us the ${\mathbbm h}^0$ (resp.~${\mathbbm h}^1$) term
in the statement of the proposition.
Now the image of $\hat{Y}^{i-1}\to \hat{Z}^i$ contains $t^r \hat{Z}^i$ (and even $t^r\rg_{\rm dR}^i$);
hence we can replace $(\hat{Y}^{i-1},\hat{Z}^i)$ by 
the $\Bdr^+$-Pair $(Y^{i-1},Z^i)$ with $Y^{i-1}=\hat{Y}^{i-1}/t^r$ and
$Z^i=\hat{Z}^i/t^r$. 
\end{proof}

\section{De Rham-to-pro-\'etale comparison theorems}\label{SS4}
  
Using Corollary~\ref{pos20} we will now prove  Conjecture~\ref{quasi1} for some quasi-compact smooth dagger varieties and then, via a limit argument, for a product of a smooth proper and a smooth Stein variety. 

\subsection{Small varieties}  \label{small1}
We will say that a smooth dagger variety  
{\it has de Rham slopes $\geq 0$} if it is small and satisfies Conjecture~\ref{quasi7.1} 
for all $r\ge 0$. We are going to prove that various dagger varieties have de Rham slopes $\geq 0$.
By Corollary~\ref{pos20} this implies that they satisfy  Conjectures~\ref{quasi1},  \ref{pos1}, and~\ref{quasi10}.
\begin{theorem}\label{quasi100}
Let $X$ be a smooth dagger variety over $C$, which has one of the
following (non mutually exclusive) forms:

\quad {\rm (a)} a product of two    varieties with de Rham slopes $\geq 0$, 

\quad {\rm (b)} the  analytification of an algebraic variety,

\quad {\rm (c)} an  almost proper smooth variety (i.e., the complement in a proper 
smooth rigid analytic variety of a closed subvariety). 

Then $X$ has de Rham slopes $\geq 0$.

%
%
%
%
\end{theorem}
\begin{proof} We start with claim (a). For   a quasi-compact smooth dagger variety $T$ over $C$, let $M^i(T):=(H^i_{\hk}(T),F^0H^i_{\dr}(T/\B_{\dr}))$ 
be the associated filtered $(\phi,N)$-modules. Let $X,Y$ be small smooth dagger varieties over $C$ that have de Rham slopes $\geq 0$.
 We claim that $X\times Y$ has de Rham slopes $\geq 0$  as well. 

 Indeed, the Hyodo-Kato and de Rham cup products induce  a pairing
$$
M^i(X)\otimes M^j(Y)\to M^{i+j}(X\times Y),\quad i,j\in\N.
$$
By functoriality of the functor $\se(-)$ it gives a pairing
$$
\gamma: \se(M^i(X))\otimes_{\so} \se(M^j(Y))\to \se(M^i(X)\otimes M^j(Y))\to \se(M^{i+j}(X\times Y)).
$$
Since we have a K\"unneth formula for Hyodo-Kato cohomology (reduce, via the Hyodo-Kato quasi-isomorphism,  to K\"unneth formula for de Rham cohomology), the map $\gamma$ yields an exact sequence of coherent sheaves on $X_{\FF}$
$$
0\to  \oplus_{i+j=n}\se(M^i(X))\otimes_{\so} \se(M^j(Y))\lomapr{\gamma} \se(M^{n}(X\times Y))\to \sff\to 0,
$$
where $\sff$ is supported at $\infty$. Applying cohomology we get the exact sequence
$$
\oplus_{i+j=n} H^1(X_{\FF},\se(M^i(X))\otimes_{\so} \se(M^j(Y)))\lomapr{\gamma} H^1(X_{\FF},\se(M^{n}(X\times Y)))\to H^1(X_{\FF},\sff)\to 0
$$
Since $H^1(X_{\FF},\sff)=0$ and the tensor product of two acyclic vector bundles is acyclic (see Remark~\ref{quasi6}), this yields $H^1(X_{\FF},\se(M^{n}(X\times Y)))=0$, as wanted.

\vskip1mm
 We pass now to claim (b). Let $X$ be a smooth algebraic variety over $C$ and denote by $X^{\rm an}$ its analytification. 
We claim that $X^{\rm an}$ has de Rham slopes $\geq 0$. 
Indeed, the filtered $(\phi,N)$-module $(H^i_{\hk}(X), F^0H^i_{\dr}(X/\B_{\dr}))$, $i\in\N$,  is admissible by \cite{BE2}, hence acyclic. The natural map 
$$\gamma: M^i(X):=(H^i_{\hk}(X), F^0H^i_{\dr}(X/\B_{\dr}))\to (H^i_{\hk}(X^{\rm an}), F^0H^i_{\dr}(X^{\rm an}/\B_{\dr}))=M^i(X^{\rm an})$$ is an isomorphism
on the first factor by Hyodo-Kato GAGA (see \cite{Shao}). Applying the functor $\se(-)$, we get an exact sequence of coherent sheaves on $X_{\FF}$
$$
0\to \se(M^i(X))\lomapr{\gamma} \se(M^i(X^{\rm an}))\to \sff\to 0,
$$
where $\sff$ is a coherent sheaf supported at $\infty$. Taking cohomology we get a short exact sequence
$$
H^1(X_{\FF}, \se(M^i(X)))\lomapr{\gamma} H^1(X_{\FF},\se(M^i(X^{\rm an})))\to 0
$$
Since $H^1(X_{\FF}, \se(M^i(X)))=0$ this yields $H^1(X_{\FF},\se(M^i(X^{\rm an})))=0$, as wanted.

\vskip1mm
 Concerning claim (c), 
let $X$ be a proper smooth rigid analytic variety over $C$ and let $Z\subset X$ be a closed subvariety. Let $U=X\setminus Z$ be the complement. We want to show  that $U$ is has de Rham slopes $\geq 0$. By resolution of singularities  (see \cite{BM97}) there exists a proper morphism $\pi: X^{\prime}\to X$ such that $\pi^{-1}(Z)_{\rm red}$ with its canonical reduced closed subspace structure is a simple normal crossing divisor. We have $U=X^{\prime}\setminus D$, hence, we may assume that $Z=D$ 
is  a divisor with simple  normal crossings in $X$.

Let $i\leq r$. Consider the following map of exact sequences in ${\rm LH}(C_{\Q_p})$
$$
\xymatrix@R=6mm@C=4mm{
 H^{i-1}_{\dr}(X^{\times}/\B^+_{\dr})\ar[r]^{0}\ar[d]^{\wr}  & H^i_{\eet}(U,\Q_p(r))\ar[r] \ar[d] 
& *\txt{\Small{ $(H^i_{\hk}(U)\wh{\otimes}_{F^{\nr}}\B^+_{\st})^{N=0,\phi=p^r}$}\\ $\oplus$ \\ 
\Small{$H^i(F^r\R\Gamma(X^{\times}/\B^+_{\dr}))$}}\ar[r] \ar[d]& H^i_{\dr}(X^{\times}/\B^+_{\dr})\ar[r] \ar[d]^{\wr} & 0\\
 H^{i-1}_{\dr}(U/\B^+_{\dr})\ar[r]  & \wt{H}^i_{\proeet}(U,\Q_p(r))\ar[r]^-{f_{i,1}} 
& *\txt{\Small{$(H^i_{\hk}(U)\wh{\otimes}_{F^{\nr}}\B^+_{\st})^{N=0,\phi=p^r}$}\\ $\oplus$\\
 \Small{$\wt{H}^i(F^r\R\Gamma(U/\B^+_{\dr}))$}}\ar[r]^-{f_{i,2}} & H^i_{\dr}(U/\B^+_{\dr})
}
$$
The top sequence is exact by \cite{Shao} (see \cite[Cor. 1.10]{CN1} for the case of semistable reduction); here $H^{i-1}_{\dr}(X^{\times}/\B^+_{\dr})$ denotes  $\B^+_{\dr}$-cohomology of the log-space $X^{\times}$ with its   logarithmic filtration. We note that
the group $H^i(F^r\R\Gamma(X^{\times}/\B^+_{\dr}))$ is a finite length  $\B^+_{\dr}$-module. 
The diagram shows   that the map $f_{i,2}$ in the bottom sequence is surjective. 
Since this works for all $i\leq r\in\N$, it implies that that the map $f_{i,1}$ is injective, as wanted. 
\end{proof}

\begin{remark}
The above proof shows that, in the case of almost proper $U$ over $C$,  the difference between the \'etale and pro-\'etale cohomology is the same as the difference between the Hodge filtration and the naive Hodge filtration:
$$
[ H^i_{\eet}(U,\Q_p(r)) \to \wt{H}^i_{\proeet}(U,\Q_p(r))]\simeq [H^i(F^r\R\Gamma(X^{\times}/\B^+_{\dr}))\to \wt{H}^i(F^r\R\Gamma(U/\B^+_{\dr}))],\quad i\leq r.
$$
\end{remark}
\subsubsection{Complements of tubular neighbourhoods} 
We will give two basic examples of complements of tubular neighbourhoods in proper smooth dagger  
varieties that have de Rham slopes $\geq 0$. It should be possible to vastly generalize these examples,
but this would require other techniques than the ones used in the paper.

\vskip1mm
  For the first example, we assume everything has semistable models.
Recall the following result of Grosse-Kl\"onne \cite[Cor. 3.8]{GKdR}
 \begin{proposition} \label{GK1}
 \begin{enumerate}
 \item  Let $\varpi\in\so_K$ be a uniformizer. Let $\mathcal{Z}\hookrightarrow  \sx$  be a closed immersion
of weak formal $\so_K$-schemes. Assume that locally on $\sx$ there exist \'etale morphisms of
weak formal $\so_K$-schemes
$$q : \sx \to {\rm Spf}(\so_K \{ X_1, . . . ,X_n \} /(X_1 . . . X_r -\varpi)),$$
such that $\mathcal{Z} = \cap_{j=r+1}^{m}V (q^*X_j )$, for some $1 \leq r \leq m \leq n$. Then the natural map
$$\R\Gamma_{\dr}(\sx_K \setminus \mathcal{Z}_K) \to \R\Gamma_{\dr}(]\sx_k \setminus \mathcal{Z}_k[_{\sx} )$$
is a filtered  quasi-isomorphism.
\item In the setting of  (1), we can replace $\mathcal{Z} = \cap_{j=r+1}^{m}V (q^*X_j )$ with  $\mathcal{Z} = \cup_{j=r+1}^{m}V (q^*X_j )$.
 \end{enumerate}
 \end{proposition}
 In the setting of Proposition~\ref{GK1}, assume that $\sx_K$ is proper and set $U:=]\sx_k \setminus \mathcal{Z}_k[_{\sx}$. We note that $U$ is quasi-compact. The space  $U_C$ has de Rham slopes $\geq 0$ because,  via the Hyodo-Kato isomorphism  and Proposition~\ref{GK1}, the canonical map
 $$
 \R\Gamma_{\hk}(\sx_K \setminus \mathcal{Z}_K) \to \R\Gamma_{\hk}(]\sx_k \setminus \mathcal{Z}_k[_{\sx} )
 $$
 is a quasi-isomorphism, the analogous map for $\B_{\dr}$-cohomology is a filtered quasi-isomorphism,  
and we can use  claim (c) of Theorem~\ref{quasi100} to deduce that the left-hand side has de Rham slopes~$\geq 0$
from which it follows that the right-hand side
has de Rham slopes~$\geq 0$.

\vskip1mm
   For the second example, 
we don't assume the existence of semistable models, but we remove a sufficiently small
tubular neighbourhood
of a smooth subvariety.
We have the following setting: 
 \begin{proposition}\label{ponie10} 
Let $X$ be a projective space over $K$ and let $Z$ be a smooth scheme-theoretic intersection $Z_1\times_X\cdots\times_XZ_c$ of $c$  hypersurfaces. Then, there exists an element $\epsilon_0\in|K^*|$ such that, for any $\epsilon<\epsilon_0$, the canonical  morphisms
  $$
 \R\Gamma_{\dr}(X\setminus Z) \to  \R\Gamma_{\dr}(X\setminus Z(\epsilon)),\quad   \R\Gamma_{\hk}(X_C\setminus Z_C)\to  \R\Gamma_{\hk}(X_C\setminus Z(\epsilon)_C)
  $$ are  quasi-isomorphism (compatible with all the structures).  Here $Z(\epsilon)$ is the tubular neighbourhood  defined in \cite[proof of Prop. 8.7]{Sch0}, \cite[Sec. 4]{BKV}.
  
  Similarly for $Z$ replaced by the union of $Z_i$'s. 
 \end{proposition}
 \begin{proof}
 The Hyodo-Kato case follows from the de Rham case via the Hyodo-Kato isomorphism. For the de Rham case, note that, locally, we have an isomorphism
 $$
 Y(\epsilon)\simeq Y\times_K {\mathbb B}(\epsilon)^c,
 $$
 where $ {\mathbb B}(\epsilon)$ denotes the open disc of radius $\epsilon$. 
 Hence  the quasi-isomorphsm in our  proposition follows from \cite[Prop. 3.7]{GKdR}.  
 
   For the union of $Z_i$'s, we use the above case and a \v{C}ech analytic descent. 
 \end{proof}
 
 The dagger variety $U(\epsilon)_C:=X_C\setminus Z(\epsilon)_C$ has de Rham slopes $\geq 0$   
via  Proposition~\ref{ponie10} and   claim (c) of  Theorem~\ref{quasi100}. 

\subsection{Proper times Stein} In this section we prove Conjecture~\ref{quasi1} for products of proper smooth and Stein varieties. 
\subsubsection{Limit arguments} This is a preliminary computation for  treating the case of partially proper varieties. 
The general set-up is the following: 
if $X$ is a smooth dagger variety over $C$, we choose a covering $\{X_n\}_{n\in\N}$ by
increasing
quasi-compact dagger subvarieties, and we want to deduce the result for $X$ from the result
for the $X_n$'s. To make formulas less cumbersome, let us introduce some notation.

Set 
\begin{align}\label{nota10}
& X^{r,i}:=
(H^i_{\hk}(X)\wh{\otimes}^R_{F^{\nr}}\B^+_{\st})^{N=0,\phi=p^r}, & & X^{r,i}_n:=
(H^i_{\hk}(X_n)\wh{\otimes}^R_{F^{\nr}}\B^+_{\st})^{N=0,\phi=p^r},\\
&F^{r,i}:= H^i(F^r\rg_{\dr}(X/\bdr^+)), && F^{r,i}_{n}:=H^i(F^r\rg_{\dr}(X_n/\bdr^+)),\notag\\
&A^{r,i}:={\rm Ker}(F^{r,i}\to H^i_{\rm dR}(X/\bdr^+)), 
&&A^{r,i}_{n}:={\rm Ker}(F^{r,i}_{n}\to H^i_{\rm dR}(X_n/\bdr^+)),\notag\\
&\overline F^{r,i}:=F^{r,i}/A^{r,i},&&\overline F_n^{r,i}:=F_n^{r,i}/A_n^{r,i}.\notag
\end{align}
So that, by definition,
\begin{equation}\label{nota11}
\overline F^{r,i}\simeq F^rH^i_{\rm dR}(X/\bdr^+),
\quad \overline F_n^{r,i}\simeq F^rH^i_{\rm dR}(X_n/\bdr^+).
\end{equation}
In particular,
\begin{equation*}
\overline F^{0,i}\simeq  H^i_{\rm dR}(X/\bdr^+),
\quad \overline F_n^{0,i}\simeq H^i_{\rm dR}(X_n/\bdr^+).
\end{equation*}

  We want to prove the exactness of:
\begin{equation}\label{nota31}
0\to H^{r,i} \to X^{r,i}\oplus F^{r,i}_{n}\to \overline F^{0,i}\to 0
\end{equation}
Assume that we have the exact sequences 
\begin{equation}
\label{wieczor5}
0\to H^{r,i}_{n} \to X^{r,i}_n\oplus F^{r,i}_{n}\to \overline F^{0,i}_{n}\to 0
\end{equation}
Hence, to prove the exactness of  (\ref{nota31}), we need to control various $\R^1\wlim_n$.
\begin{remark}\label{setup1}
(i) 
By \cite[Prop. 5.8]{CN4} and \cite[Prop. 5.12]{CN4}, respectively, we have
$$\R^1\wlim_n X^{r,i}_n=0,\quad \R^1\wlim_n \overline F^{r,i}_{n} =0$$ 
(\cite[Prop. 5.12]{CN4} proves the result for $\overline F^{0,i}_{n}$; the general case
follows because 
$\overline F^{r,i}_{n}$ contains $t^r\overline F^{0,i}_{n}$, 
and $\overline F^{r,i}_{n}/t^r\overline F^{0,i}_{n}$ is a BC).

(ii) We have two exact sequences and isomorphisms
\begin{align*}
0\to \R^1\wlim_n F^{r,i-1}_{n}\to F^{r,i}\to \wlim_n F^{r,i}_{n}\to 0,\\
0\to \R^1\wlim_n H^{r,i-1}_{n}\to H^{r,i}\to \wlim_n H^{r,i}_{n}\to 0,\\
X^{r,i}\simeq \wlim_n X^{r,i}_n,\quad \overline F^{0,i}\simeq  \wlim_n \overline F^{0,i}_{n}.
\end{align*}

(iii)
Applying $\wlim_n$ to (\ref{wieczor1}) we obtain the exact sequence 
\begin{align*}
0\to \wlim_n H^{r,i}_{n}& \to \wlim_n X^{r,i}_n\oplus \wlim_nF^{r,i}_{n} \to \wlim_n 
\overline F^{0,i}_{n}
\to \R^1\wlim_n H^{r,i}_{n} \to \R^1\wlim_n F^{r,i}_{n}\to 0
\end{align*}
Hence, if we can prove that $\R^1\wlim_n H^{r,i}_{n}=0$, for all $i$,
we get that $H^{r,i}\simeq  \wlim_n H^{r,i}_{n}$, that (\ref{nota31}) is exact and, as a bonus, that
$\R^1\wlim_n F^{r,i}_{n}=0$
and $F^{r,i}\simeq \wlim_n F^{r,i}_{n}$.
\end{remark}
Proving directly the vanishing of the $\R^1\wlim_n H^{r,i}_{n}$'s for all $i$
does not look too promising since pro-\'etale cohomology is not that easy to deal with, but one
can try to prove statements about the filtered de Rham complexes that imply this vanishing. 
Set
$$
{\rm DR}^{r,i}:=H^i(\rg_{\dr}(X/\B^+_{\dr})/F^r),
\quad {\rm DR}^{r,i}_n:=H^i(\rg_{\dr}(X_n/\B^+_{\dr})/F^r).
$$
\begin{lemma}\label{imply1}
We have the  implications
$$(\R^1\wlim_n {\rm DR}^{r,i-1}_n=0)\Longrightarrow (\R^1\wlim_n A^{r,i}_n=0)
\Longrightarrow (\R^1\wlim_n H^{r,i}_n=0).$$
\end{lemma}
\begin{proof}
We have a long exact sequence
$$ \cdots \overline F^{0,i-1}_n\to {\rm DR}^{r,i-1}_n\to F^{r,i}_n\to \overline F^{0,i}_n\to\cdots$$ 
It follows that $A^{r,i}_n$ is a quotient of ${\rm DR}^{r,i-1}_n$. This proves the first
implication.

For the second implication, note that $A^{r,i}_n$ is naturally a subspace of $H^{r,i}_n$;
denote by $\overline H^{r,i}_n$ the quotient. 
Since $t^r\overline F^{0,i}_n$ is a subspace of $\overline F^{r,i}_n$,
we have an exact sequence
$$0\to \overline H^{r,i}_n\to X^{r,i}_n\oplus (\overline F^{r,i}_n/t^r\overline F^{0,i}_n)\to
(\overline F^{0,i}_n/t^r\overline F^{0,i}_n)\to 0$$
in which all the terms, except $\overline H^{r,i}_n$ are BC's. This shows that
$\overline H^{r,i}_n$ is a BC as well\footnote{To be more precise,
by \cite[Sec. 7.3]{CN4},  the exact sequence (\ref{wieczor5})  can be lifted to the category of Vector Spaces. Moreover, since $H^i_{\rm dR}(X_n)$ is finite
dimensional, all the terms in this exact sequence  are BC's:
this is clear for $\overline{\mathbb F}^{0,i}/t^r\overline{\mathbb F}^{0,i}$ and for
${\mathbb X}^{r,i}_n$; it is true for $\overline{\mathbb F}^{r,i}/t^r\overline{\mathbb F}^{0,i}$
because it is
is a sub-$\Bdr^+$-module of $\overline{\mathbb F}^{0,i}/t^r\overline{\mathbb F}^{0,i}$, 
and, finally, it is true
for $\overline {\mathbb H}^{r,i}$ as the kernel of a morphism between BC's.}, 
hence $\R^1\wlim_n \overline H^{r,i}_n=0$
by Remark~\ref{Mittag}.
It follows that we have an exact sequence
$$0\to \lim_n A^{r,i}_n\to  \lim_n H^{r,i}_n\to \lim_n\overline H^{r,i}_n\to 
\R^1\lim_n A^{r,i}_n\to \R^1\lim_n H^{r,i}_n\to 0$$
which proves the second implication.
\end{proof}
\begin{remark}\label{imply2}
The above proof yields that: 
$$\R^1\wlim_n \overline H^{r,i}_n=0$$
\end{remark}

\subsubsection{Proper times Stein} 

 Let $X$ be the dagger variety $X:=Y\times_CS$, where $Y$ and $S$ are  smooth dagger varieties over $C$, $Y$ is proper and $S$ is Stein.   We cover $S$ with a strictly increasing admissible covering $\{S_n\}_{i\in\N}$ by open dagger affinoids. Set $X_n:=Y\times_CS_n$.
\begin{proposition} \label{above1} For $i\geq 0$, we have 
$$\R^1\lim_n H^i_{\proeet}(X_n,\Q_p)=0. 
$$ 
\end{proposition}
\begin{proof} Since $\R\lim_nH^i_{\proeet}(X_n,\Q_p)=\R\lim_nH^i_{\proeet}(\wh{X}_n,\Q_p)$, we can replace $X_n$, $S_n$ with their completions. 
Since $\wh{S}_n$'s are quasi-compact, we can pass to \'etale cohomology.  
 Using Lemma~\ref{kunn} below and Theorem~\ref{affinoids} we get
\begin{align*}
\R^1\lim_n H^i_{\eet}(X_n,\Q_p) & \simeq \bigoplus_{j\geq 0}H^j_{\eet}(Y,\Q_p)\otimes_{\Q_p}\R^1\lim_n H^{i-j}_{\eet}(\wh{S}_n,\Q_p)=0,
\end{align*}
 as wanted. 
\end{proof}
\begin{lemma}\label{kunn}
Let $Y, T$ be  smooth rigid analytic varieties over $C$. Assume that $Y$ is proper. Let $i\geq 0$. 
We have the K\"unneth  isomorphism 
\begin{equation}
\label{Kunn}
\bigoplus_{j\geq 0}H^j_{\eet}(Y,\Q_p)\otimes_{\Q_p}H^{i-j}_{\eet}(T,\Q_p)\stackrel{\sim}{\to} H^i_{\eet}(Y\times_CT,\Q_p).
\end{equation}
\end{lemma}
\begin{proof}\ 

$\bullet$ {\em Step1: the integral K\"unneth formula.} We claim that the canonical map
$$
\R\Gamma_{\eet}(Y,\Z_p)\otimes_{\Z_p}^{L}\R\Gamma_{\eet}(T,\Z_p)\to \R\Gamma_{\eet}(Y\times_CT,\Z_p)
$$
is a quasi-isomorphism. The proof of K\"unneth formula from \cite[Tag 0F13]{SP} goes through once we have the quasi-isomorphism
\begin{equation}
\label{Kun11}
\R p_{2,*}\Z_p\stackrel{\sim}{\leftarrow} \underline{\R\Gamma_{\eet}(Y,\Z_p)},
\end{equation}
where $p_{2}: Y\times_CT\to T$ is the canonical projection.  This  is a consequence of the proper base change theorem of Bhatt-Hansen \cite[Th. 1.6]{BH}.   To recall the proof from \cite[Tag 0F13]{SP}: we apply $\R\Gamma_{\eet}(T,-)$ to 
the quasi-isomorphism (\ref{Kun11}) and get a quasi-isomorphism
$$
\R\Gamma_{\eet}(Y\times_CT,\Z_p)\stackrel{\sim}{\leftarrow} \R\Gamma_{\eet}(T,\underline{\R\Gamma_{\eet}(Y,\Z_p)}).
$$
Now, we use the fact that the complex $ \R\Gamma_{\eet}(T,\underline{\R\Gamma_{\eet}(Y,\Z_p)})$ is perfect  and get a quasi-isomorphism
$$
 \R\Gamma_{\eet}(T,\underline{\R\Gamma_{\eet}(Y,\Z_p)}) \stackrel{\sim}{\leftarrow}\R\Gamma_{\eet}(T,\Z_p)\otimes_{\Z_p}^L\R\Gamma_{\eet}(Y,\Z_p).
$$
This finishes the proof of the integral K\"unneth formula.
  
\vskip1mm
 $\bullet$ {\em Step 2: passing to $\Q_p$-coefficients. } From Step 1 and the fact that the complex $ \R\Gamma_{\eet}(Y,\Q_p)$ is perfect, we get the canonical quasi-isomorphism
 $$
 \R\Gamma_{\eet}(Y,\Q_p)\otimes^L_{\Q_p}\R\Gamma_{\eet}(T,\Q_p)\stackrel{\sim}{\to }\R\Gamma_{\eet}(Y\times_CT,\Q_p).
$$
Taking cohomology of both sides we get the isomorphism in the proposition. 
\end{proof}
\begin{remark}
(i) Proposition~\ref{above1} implies the isomorphism
$$
H^i_{\proeet}(X,\Q_p)\stackrel{\sim}{\to} \lim_nH^i_{\eet}(X_n,\Q_p).
$$
\hskip.4cm(ii)
The  proof of Proposition~\ref{above1} yields also the K\"unneth  isomorphism
$$
\bigoplus_{j\geq 0}H^j_{\proeet}(Y,\Q_p)\otimes_{\Q_p} H^{i-j}_{\proeet}(S,\Q_p)\stackrel{\sim}{\to}  H^i_{\proeet}(X,\Q_p).
$$
This K\"unneth decomposition is true because we are in the very special case where the cohomology
groups of one of the factors are finite dimensional. 
Without this assumption, such a decomposition already fails in the simplest case where both $X$ and $Y$
are the open unit ball in dimension $1$.
\end{remark}
\begin{proposition} \label{limits1}
Conjecture~\ref{quasi1} holds for $X=Y\times_CS$ as above. 
\end{proposition}
\begin{proof}
 Let $i\leq r$. By Remark~\ref{quasi2}, we reduce to showing  that the  sequence of algebraic cohomologies
\begin{equation*}
0\to  {H}^i_{\proeet}(X,\Q_p(r))\to ({H}^i_{\rm HK}(X)\wotimes_{F^{\nr}}\B^+_{\st})^{N=0,\varphi=p^r}\oplus {H}^iF^r\rg_{\rm dR}(X/\B^+_{\dr})\to {H}^i_{\rm dR}(X/\B^+_{\dr})\to 0
\end{equation*}
is exact. Granting Proposition~\ref{above1}, this follows from (iii) of Remark~\ref{setup1} and Section~\ref{small1}.
%
\end{proof}

\section{Pro-\'etale--to--de Rham  comparison theorems}\label{SS7} 
In this chapter we propose a recipe to extract,
from the pro-\'etale cohomology of varieties defined over $C$,
the  Hyodo-Kato and de Rham cohomologies (as modules over the relevant rings)
and, for  varieties 
defined over $K$, to extract also Frobenius, monodromy, and the naive Hodge filtration. 
\subsection{The pro-\'etale--to--de Rham $C_{\rm st}$-conjecture for varieties over $K$}
 In this section, we study the following conjecture extracting, for analytic spaces over $K$,
the  Hyodo-Kato and de Rham cohomologies
from the pro-\'etale cohomology. This extends to $p$-adic analytic spaces
the $C_{\rm st}$-conjecture of Fontaine. 
\begin{conjecture}\label{cst1}
Let $X$ be a smooth dagger variety over $K$. We have natural strict isomorphisms:
\begin{align*}
&{\rm Hom}^{\rm sm}_{\G_K}(H^i_{\proeet}(X_C,\Q_p),\bst)
\simeq H^i_{{\rm HK}}(X_C)^\dual,\quad{\text{as  a  $(\varphi,N,\G_K)$-module,}}\\
&{\rm Hom}_{\G_K}(H^i_{\proeet}(X_C,\Q_p),\bdr)
\simeq H^i_{{\rm dR}}(X)^\dual,\quad{\text{as  a filtered $K$-module}}.
\end{align*}
\end{conjecture}
\begin{remark}\label{cst2}
Our approach uses syntomic cohomology, which gives a description of
$H^i_{\proeet}(X_C,\Q_p(r))$ for $r\geq i$, rather than that of $H^i_{\proeet}(X_C,\Q_p)$.
Hence we are going to consider the following equivalent form of Conjecture~\ref{cst1}
(where the $\{r\}$ has the same meaning as in Corollary~\ref{new-tate2.7}):
\begin{align*}
&{\rm Hom}^{\rm sm}_{\G_K}(H^i_{\proeet}(X_C,\Q_p(r)),\bst)
\simeq H^i_{{\rm HK}}(X_C)^\dual\{r\},\quad{\text{as a $(\varphi,N,\G_K)$-module,}}\\
&{\rm Hom}_{\G_K}(H^i_{\proeet}(X_C,\Q_p(r)),\bdr)
\simeq H^i_{{\rm dR}}(X)^\dual\{r\},\quad{\text{as a filtered $K$-module}}.
\end{align*}
\end{remark}

Our main result is the following:
\begin{theorem}
Conjecture~\ref{cst1} holds for  smooth dagger varieties over $K$ with de Rham slopes~$\geq 0$.
\end{theorem}
\subsubsection{Dagger affinoids} The case of affinoids is included in the case of quasi-compact varieties but the proof
is considerably simpler.
We note that it suffices to show that we have natural
isomorphisms since the weak topology on the Hom-spaces is Hausdorff.
   Let 
$$M:=H^i_{\hk}(X_C),\quad X_{\st}^i(M):=(M\wh{\otimes}^R_{F^{\nr}}\B^+_{\st})^{N=0,\phi=p^i},
\quad M_K=(\overline K\otimes_{F^{\nr}}M)^{\G_K}\simeq H^i_{{\rm dR}}(X).$$
The last isomorphism follows from the Hyodo-Kato isomorphism \cite[Th.\,4.27]{CN4}. 
  Recall that we have the  exact sequence (see  Theorem~\ref{affinoids})
 \begin{equation}
 \label{seq1}
 0\to (\Omega_X^{i-1}/\kker d)\wh{\otimes}_KC\to H^i_{\proeet}(X_C,\Q_p(i))\to X_{\st}^i(M)\to 0.
 \end{equation}
Applying $\Hom_{\sg_K}(-, \B_{\st})$ and $\Hom_{\sg_K}(-, \B_{\dr})$ to it, we get the following  exact sequences
\begin{align*}
0\to &  \Hom^{\rm sm}_{\sg_K}(X_{\rm st}^i(M), \B_{\st})  \to \Hom^{\rm sm}_{\sg_K}(H^i_{\proeet}(X_C,\Q_p(i)), \B_{\st})\to \Hom^{\rm sm}_{\sg_K}((\Omega^{i-1}_{X}/\kker\,d)\widehat\otimes_K C, \B_{\st}),\\
0\to  & \Hom_{\sg_K}(X_{\rm st}^i(M), \B_{\dr})  \to \Hom_{\sg_K}(H^i_{\proeet}(X_C,\Q_p(i)), \B_{\dr})\to \Hom_{\sg_K}((\Omega^{i-1}_{X}/\kker\,d)\widehat\otimes_K C, \B_{\dr}).
\end{align*}
\begin{lemma}\label{new-niedziela}We have 
$$
\Hom^{\rm sm}_{\sg_K}((\Omega^{i-1}_{X}/\kker\,d)\widehat\otimes_K C, \B_{\st})=0,\quad \Hom_{\sg_K}((\Omega^{i-1}_{X}/\kker\,d)\widehat\otimes_K C, \B_{\dr})=0.
$$
\end{lemma}
\begin{proof}
Since $\Hom^{\rm sm}_{\sg_K}((\Omega^{i-1}_{X}/\kker\,d)\widehat\otimes_K C, \B_{\st})\simeq \colim_{L/K}\Hom_{\sg_L}((\Omega^{i-1}_{X}/\kker\,d)\widehat\otimes_K C, \B_{\st})$ and 
 $\B_{\st}\hookrightarrow \B_{\dr}$, it suffices to show that
$ \Hom_{\sg_L}((\Omega^{i-1}_{X}/\kker\,d)\widehat\otimes_K C, \B_{\dr})=0$, for a finite extension $L$ of $K$. 

But 
$\Hom_{\sg_L}(C, \B_{\dr})=0$ (Proposition~\ref{twists} (iv)); 
hence
$\Hom_{\sg_L}((\Omega^{i-1}_{X}/\kker\,d)\otimes_K C, \B_{\dr})=0$.
Since our maps are requested to be continuous and
$(\Omega^{i-1}_{X}/\kker\,d)\otimes_K C$ is dense in $(\Omega^{i-1}_{X}/\kker\,d)\wh{\otimes}_K C$, 
 we get the wanted vanishing.
 \end{proof}

  This lemma yields   isomorphisms
 \begin{align*}
   \Hom^{\rm sm}_{\sg_K}(X_{\rm st}^i(M), \B_{\st})   & \stackrel{\sim}{\to} \Hom^{\rm sm}_{\sg_K}(H^i_{\proeet}(X_C,\Q_p(i)), \B_{\st}),\\
\Hom_{\sg_K}(X_{\rm st}^i(M), \B_{\dr})  & \stackrel{\sim}{\to} \Hom_{\sg_K}(H^i_{\proeet}(X_C,\Q_p(i)), \B_{\dr}).
 \end{align*}
The filtration on $M_K$ is concentrated in degree $i$ 
(i.e., ${\rm Fil}^iM_K=M_K$ and ${\rm Fil}^{i+1}M_K=0$). Hence we can use Example~\ref{new-tate23}
of Corollary~\ref{new-tate2.7} to finish the proof of Conjecture~\ref{cst1}
in the case of dagger affinoids.
\subsubsection{Quasi-compact  dagger varieties with de Rham slopes $\geq 0$}
Let $X$ be a quasi-compact smooth dagger variety over $K$.
Fix $r\geq i$. Set:
\begin{align*}
&\wt{H}^{r,i}:=\wt{H}^i_{\proeet}(X_C,\Q_p(r)), \quad 
\wt{F}^{r,i}:=\wt{H}^i(F^r(\rg_{\rm dR}(X)\wotimes_K\bdr^+)), \\
&{X}^{r,i}:=(H^i_{\rm HK}(X_C)\wotimes_{F^{\nr}}\bst^+)^{N=0,\varphi=p^r},\quad 
 B^i:=H^i_{\rm dR}(X)\wotimes_K\bdr^+.
\end{align*}
Let $\wt{A}^{r,i}$ be the kernel of the canonical map $\wt{F}^{r,i}\to B^i$. Then
$\wt{A}^{r,i}$ is
 also canonically a subgroup of $\wt{H}^{r,i}$.
Let 
$$\overline H^{r,i}:=\wt{H}^{r,i}/\wt{A}^{r,i},\quad 
\overline F^{r,i}:=\wt{F}^{r,i}/\wt{A}^{r,i}.$$ 
Note that $\wt{F}^{r,i}/\wt{A}^{r,i}$ is a subgroup of $B^i$, hence it is classical. 
\begin{lemma}\label{nioule1}
We have, for all $i\leq r$,
$${\rm Hom}_{\G_K}^{\rm sm}(\wt{A}^{r,i},\bst)=0,\quad{\rm Hom}_{\G_K}(\wt{A}^{r,i},\bdr)=0.$$
\end{lemma}
\begin{proof} It suffices to show that 
$${\rm Hom}_{\G_K}^{\rm sm}(A^{r,i},\bst)=0,\quad{\rm Hom}_{\G_K}(A^{r,i},\bdr)=0.$$
And, for that, it is enough to prove the second statement.  
Let $${\rm DR}^{r,i}:=H^i((\bdr^+\wotimes\rg_{\dr}(X))/F^r)$$
so that we have a long exact sequence 
$\cdots\to{\rm DR}^{r,i-1}\to F^{r,i}\to B^i\to {\rm DR}^{r,i}\to\cdots$ which shows that
$A^{r,i}$ is a quotient of ${\rm DR}^{r,i-1}$. Hence it is enough
to prove the same statement for ${\rm DR}^{r,i}$, with $i<r$.

Now ${\rm DR}^{r,i}$ is the $i$-th hypercohomology group of the complex
$${\rm DR}^{r,\bullet}:=
\big((\bdr^+/t^r)\wotimes_K\so_X\to (\bdr^+/t^{r-1})\wotimes_K\Omega^1_X\to\cdots\to
(\bdr^+/t)\wotimes_K\Omega^{r-1}_X\big).$$
Choose a covering of $X$ by dagger affinoids, and
denote by $Z^{r,i}$ the group of $i$-cocycles of the \v{C}ech double complex associated 
to this covering. Since ${\rm DR}^{r,i}$ is a quotient of $Z^{r,i}$, it is enough to prove
that ${\rm Hom}_{\G_K}(Z^{r,i},\bdr)=0$. 

Denote by $Z^{j,i}_K$ the group of $i$-cocycles of the \v{C}ech double complex associated 
to the above covering and the complex
$${\rm DR}^{j,\bullet}_K:=\big(\so_X\to\Omega^1_X\to\cdots\to \Omega^{j-1}_X\big).$$
We are going to prove that $\sum_{j\leq r}(t^{r-j}\bdr^+/t^r\bdr^+)\otimes_K Z^{j,i}_K
\to Z^{r,i}$ has dense image. This will allow us to conclude since
${\rm Hom}_{\G_K}((t^{r-j}\bdr^+/t^r\bdr^+)\otimes_K Z^{j,i}_K,\bdr)=0$ by Proposition~\ref{twists}
and our maps are assumed to be continuous.

To prove this density, choose a Banach basis over $K$ of the $K$-Banach $\bdr^+/t^r$ of the form
$(t^je_n)_{0\leq j\leq r-1,\,n\in\N}$ (pick  a family $e_n$ of elements of $\bdr^+/t^r$
whose images in $\bdr^+/t=C$ form a Banach basis of $C$ over $K$).  Then one can use
$(t^je_n)_{0\leq j\leq k-1,\,n\in\N}$ as a Banach basis of $\bdr^+/t^k$, if $k\leq r$.
This makes it possible to decompose ${\rm DR}^{r,\bullet}$ as a completed direct
sum of the complexes $t^je_n\otimes {\rm DR}^{r-j,\bullet}_K$'s, and then $Z^{r,i}$ is the completion
of the sum of the $t^je_n\otimes Z^{r-j,i}_K$'s.
\end{proof}

It follows from Lemma~\ref{nioule1} that
we have isomorphisms
$${\rm Hom}^{\rm sm}_{\G_K}(\overline H^{r,i},\bst)\overset{\sim}{\to} {\rm Hom}_{\G_K}^{\rm sm}(\wt{H}^{r,i},\bst),
\quad
{\rm Hom}_{\G_K}(\overline H^{r,i},\bdr)\overset{\sim}{\to} {\rm Hom}_{\G_K}(\wt{H}^{r,i},\bdr) $$
 Since $X$ is quasi-compact for which
Conjecture~\ref{quasi1} holds and 
we have the exact sequence (\ref{quasi19}):
$$0\to \wt{H}^{r,i}\to
X^{r,i}\oplus \wt{F}^{r,i}\to B^i\to 0.$$
This induces
exact sequences
\begin{equation}\label{exact1}
0\to\overline H^{r,i}\to X^{r,i}\oplus \overline F^{r,i}
\to B^i\to 0, \quad
0\to\overline H^{r,i}\to X^{r,i}
\to B^i/\overline F^{r,i}\to 0, \quad
\end{equation}
Since $\overline F^{r,i}$ is a subgroup of $B^i$, all the terms in these sequences are classical. 
This identifies topologically  $\overline H^{r,i}$ with $V_{\rm st}^r(H^i_{\rm HK}(X_C),H^i_{\rm dR}(X))$. 
Hence we can use Corollary~\ref{new-tate2.7} to finish the proof of Conjecture~\ref{cst1}
in the case of quasi-compact  smooth dagger varieties with de Rham slopes~$\geq 0$.

 \subsubsection{Other  small varieties with de Rham slopes $\geq 0$}
The proof in the case of other  small varieties with de Rham slopes $\geq 0$
is  the same as in the quasi-compact case
using the fact that Conjecture~\ref{quasi1} holds in that case.

%

\subsection{The pro-\'etale--to--de Rham $C_{\rm st}$-conjecture for varieties over $C$}
The following theorem shows that one can recover de Rham cohomology (without the Hodge filtration)
and Hyodo-Kato cohomology (without actions of $\varphi$ and $N$) from
pro-\'etale cohomology for varieties over $C$, despite the absence of Galois action.
\begin{theorem}\label{miracle3.5}
Let $X$ be a   smooth dagger variety over $C$ with de Rham slopes $\geq 0$.   Let $i\geq 0$. 
\begin{align*}
& {\rm Hom}_{\rm TVS}({\mathbb H}^i_{\proeet}(X,\Q_p),\Bst)\simeq
{\rm Hom}_{F^{\rm nr}}(H^i_{\rm HK}(X),\bst),\quad{\text{as a  $\bst$-module}},\\
& {\rm Hom}_{\rm TVS}({\mathbb H}^i_{\proeet}(X,\Q_p),\Bdr)\simeq
{\rm Hom}_{\bdr^+}(H^i_{\rm dR}(X/\bdr^+),\bdr),\quad{\text{as a  $\bdr$-module}}.
\end{align*}
\end{theorem}
\begin{proof}
If $M$ is a qBC or, more generally, a Topological Pair, we set
$$h(M):={\rm Hom}_{\rm TVS}(M,\Bdr). $$

Fix $r\geq i$. Set:
\begin{align*}
&{\mathbb H}^{r,i}:={\mathbb H}^i_{\proeet}(X_C,\Q_p(r)), \quad 
{\mathbb F}^{r,i}:=H^i(F^r(\rg_{\rm dR}(X/\bdr^+)\wotimes_{\bdr^+}\Bdr^+)), \\
&{\mathbb X}^{r,i}:=(H^i_{\rm HK}(X_C)\wotimes_{F^{\nr}}\Bst^+)^{N=0,\varphi=p^r},\quad 
 {\mathbb B}^i:=H^i_{\rm dR}(X/\bdr^+)\wotimes_{\bdr^+}\Bdr^+.
\end{align*}
Let ${\mathbb A}^{r,i}$ be the kernel of the canonical map ${\mathbb F}^{r,i}\to {\mathbb B}^i$. Then
${\mathbb A}^{r,i}$ is a $\Bdr^+$-Pair; it is 
 also canonically a subgroup of ${\mathbb H}^{r,i}$.
Let 
$$\overline {\mathbb H}^{r,i}:={\mathbb H}^{r,i}/{\mathbb A}^{r,i},\quad 
\overline {\mathbb F}^{r,i}:={\mathbb F}^{r,i}/{\mathbb A}^{r,i}.$$ 
Since ${\mathbb A}^{r,i}$ is a $\Bdr^+$-Pair, it contains
a dense submodule which is an inductive limit of finite length $\Bdr^+$-Modules,
hence 
$h({\mathbb A}^{r,i})=0$ by Corollary~\ref{HN12}, and
we have an isomorphism
$$h(\overline {\mathbb H}^{r,i})\overset{\sim}{\to} h({\mathbb H}^{r,i}).$$
We also have natural sequences
\begin{align*}
0\to{\mathbb H}^{r,i}\to
{\mathbb X}^{r,i}\oplus {\mathbb F}^{r,i}\to {\mathbb B}^i\to 0,
\quad &
0\to {\mathbb H}^{r,i}\to
{\mathbb X}^{r,i}\oplus {\mathbb F}^{r,i}\to {\mathbb B}^i\to 0\\
0\to\overline {\mathbb H}^{r,i}\to {\mathbb X}^{r,i}
&\to {\mathbb B}^i/\overline {\mathbb F}^{r,i}\to 0
\end{align*}

Since $X$ is small, the last sequence is a sequence of BC's. It is
exact because passing to $C$-points yields the sequence
$0\to\overline H^{r,i} \to {X}^{r,i}
\to {B}^i/\overline {F}^{r,i}\to 0$ which was proven to be exact,
as a consequence of the validity of
Conjecture~\ref{quasi1} (see (\ref{exact1})).
This identifies $\overline{\mathbb H}^{r,i}$ with 
${\mathbb V}^r_{\rm st}(H^i_{\rm HK}(X),t^{-r}\overline{F}^{r,i})$,
which makes it possible to use Proposition~\ref{baco3.7} to prove
Theorem~\ref{miracle3.5}.
\end{proof}

\begin{remark}\label{recover}
(i) If $X$ is proper, then ${\mathbb H}_{\proeet}^i(X,\Q_p)$ 
is a finite dimensional $\Q_p$-vector space with
no extra structure. This shows that it is hopeless to try to recover the actions of $\varphi$
and $N$ on $H^i_{\rm HK}(X)$ or the filtration on $H^i_{\rm dR}(X/\bdr^+)$ using 
Theorem~\ref{miracle3.5}. 

(ii) On the other hand, if $X$ is a dagger affinoid, we know what the filtration
on $H^i_{\rm dR}(X/\bdr^+)$ is: we have $F^{i+k}H^i_{\rm dR}(X/\bdr^+)=t^kH^i_{\rm dR}(X/\bdr^+)$.
Also, as we have seen
\begin{equation}\label{max}
(H^i_{\rm HK}(X)\otimes_{F^{\nr}}\Bst^+)^{N=0,\varphi=p^i}\simeq 
(H^i_{\rm HK}(X)\otimes_{F^{\nr}}\Bcris^+)^{\varphi=p^i}. 
\end{equation}
Using Lemma~\ref{quasi121}, one sees that 
\begin{align*}
H^i_{\rm HK}(X)^\dual \otimes_{F^{\nr}}\bcris^+
&\subset {\rm Hom}_{\rm TVS}((H^i_{\rm HK}(X)\otimes_{F^{\nr}}\Bcris^+)^{\varphi=p^i},\Bcris^+)\\
&\subset\{\lambda\in H^i_{\rm HK}(X)^\dual \otimes_{F^{\nr}}\bcris,\ \varphi^n(\lambda)\in
H^i_{\rm HK}(X)^\dual \otimes_{F^{\nr}}\bdr^+,\  \forall n\geq 0\}\\
&=H^i_{\rm HK}(X)^\dual \otimes_{F^{\nr}}\bcris^+.
\end{align*}
Since, by the proof of Theorem~\ref{miracle3.5},
 $${\rm Hom}_{\rm TVS}({\mathbb H}^i(X,\Q_p(i)),\Bcris^+)\simeq
{\rm Hom}_{\rm TVS}((H^i_{\rm HK}(X)\otimes_{F^{\nr}}\Bcris^+)^{\varphi=p^i},\Bcris^+),$$
one sees that one
can recover the action of $\varphi$ on $H^i_{\rm HK}(X)\otimes_{F^{\nr}}\breve{C}$
(by tensoring with $\breve C$ above $\bcris^+$ and dualizing).

(iii) If $X$ is a general smooth dagger variety over $C$, point (ii) suggests 
that one might be able to  recover a sheafified version of the actions
of $\varphi$ on $H^i_{\rm HK}(X)$ and of the filtration on $H^i_{\rm dR}(X/\bdr^+)$.
Recovering the action of $N$ seems out of reach by these methods.
\end{remark}

\appendix

\section{Extensions of topological vector spaces}\label{extlh}
We list here few basic facts concerning extensions of Banach and Fr\'echet spaces for which we could not find references. 
\begin{proposition}\label{extlh1}
Let $0\to W_1\to W\to W_2\to 0$ be an exact sequence in ${\rm LH}(C_K)$.
If $W_1,W_2$ are banachs, so is $W$.
\end{proposition}
\begin{proof}
Let $U$ be an open lattice in $W$, small enough so that $U\cap W_1\subset W_1^+$ (this exists
since any open lattice in $W_1$ is the trace of an open lattice in $W$ by strictness of the sequence).
Then the image of $U$ in $W_2$ is open (by strictness of the sequence), hence contains
$p^nW_2^+$ for $n$ big enough.  Changing $U$ into its intersection with the preimage of $p^nW_2^+$,
we can assume that this image is $p^nW_2^+$, and changing $U$ into $U+W_1^+$, we can assume
that we have an exact sequence $0\to W_1^+\to U\to p^nW_2^+\to 0$.  Let $W'$ be the banach obtained
by completing $W$ for the norm induced by $U$; hence $W'=\varprojlim_n W/p^nU$, and we have a
strict exact sequence $0\to W_1\to W'\to W_2\to 0$ as well as a continuous morphism
$W\to W'$ which makes the diagram
$$\xymatrix@R=4mm@C=5mm{0\ar[r]&W_1\ar[r]\ar@{=}[d] &W\ar[r]\ar[d] &W_2\ar[r]\ar@{=}[d] &0\\  
0\ar[r] &W_1\ar[r] &W'\ar[r] &W_2\ar[r] &0}$$
commutative. We conclude by the $5$ Lemma that the middle map is an isomorphism
in ${\rm LH}(C_K)$ hence $W$ is a banach.  
\end{proof}
\begin{remark}\label{extlh2}
If $K$ is spherically complete or if $W_1,W_2$ are separable (i.e. have a
dense subspace of countable dimension) then the above sequence splits.
It seems, on the other hand, that if $\widehat{\C_p}$ is a spherical completion
of $\C_p$, the sequence of $\C_p$-banachs
$0\to\C_p\to \widehat{\C_p}\to \widehat{\C_p}/\C_p\to 0$
does not split. In particular the $\C_p$-dual of $\widehat{\C_p}$ is $0$.
\end{remark}

\begin{proposition}\label{extlh3}
Let $0\to W_1\to W\to W_2\to 0$ be an exact sequence in ${\rm LH}(C_K)$.
If $W_1,W_2$ are fr\'echets, so is $W$.
\end{proposition}
\begin{proof}
If $F$ is a fr\'echet and $(U_n)_n$ is a basis of neighborhoods of $0$ consisting
of open lattices with $U_{n+1}\subset U_n$ (we can always achieve that by replacing
$U_n$ by $U_0\cap U_1\cap\dots\cap U_n$), then the natural map
$F\to \varprojlim_n F/U_n$ is an isomorphism.

Now, for any open lattices $U_1\subset W_1$ and $U_2\subset W_2$ we can construct, as above,
an open lattice $U\subset W$ with $U\cap W_1\subset U_1$ and the image of $U$ in $W_2$
contained in $U_2$.  This makes it possible to construct
a family of neighborhoods 
of $0$ in $W$ consisting
of open lattices with $U_{n+1}\subset U_n$, such that the family $(U_{2,n})$ of images
of the $U_n$ are a basis of neighborhoods of $0$ in $U_2$, and the $U_{1,n}:=W_1\cap U_n$
are a basis of neighborhoods of $0$ in $U_1$. Let $W':=\varprojlim_n W/U_n$.
Then $W'$ is a fr\'echet (can define the topology by the distance $d(x,y)=p^{-n}$
where $n$ is the smallest integer with $x-y\notin U_n$),
and we have an exact sequence 
$0\to \varprojlim_n W_1/U_{1,n}\to W'\to \varprojlim_n W_2/U_{2,n}\to 0$ of fr\'echets
(with the two projective limits being isomorphic to $W_1$ and $W_2$ by the above discussion)
as well as a continuous morphism
$W\to W'$ which makes the diagram 
$$\xymatrix@R=4mm@C=5mm{0\ar[r]&W_1\ar[r]\ar[d]^-{\wr} &W\ar[r]\ar[d] &W_2\ar[r]\ar[d]^-{\wr} &0\\      
0\ar[r] &W_1\ar[r] &W'\ar[r] &W_2\ar[r] &0}$$ 
commutative in ${\rm LH}(C_K)$  It follows that $W\to W'$ is an isomorphism
in ${\rm LH}(C_K)$ hence $W$ is a fr\'echet.
\end{proof}

\begin{proposition}\label{extlh5}
Let $0\to W_1\to W\to W_2\to 0$ be an exact sequence in ${\rm LH}(C_K)$.
If $W_1,W_2$ are of compact type and $K$ is spherically closed or
$W_1,W_2$ are separable, so is $W$.
\end{proposition}
\begin{proof}
Passing to stereotypical duals, we get an exact sequence in ${\rm LH}(C_K)$
$$0\to W_2^\dual\to W^\dual\to W_1^\dual\to {\rm Ext}^1(W_2,K)\to\cdots$$
But ${\rm Ext}^1(W_2,K)=0$ since $W_2$ is separated (see \cite[Lemma 2.5]{CGN}).
Since $W_1^*,W_2^*$ are nuclear fr\'echets,  $W^*$ is a fr\'echet by Proposition~\ref{extlh3},
and it  is nuclear by \cite[Prop. 3.8]{OV}.
Dualizing back, and denoting by $W',W'_1,W'_2$ the stereotypical duals 
of $W^\dual, W_1^\dual, W_2^\dual$, we get a diagram
 \begin{equation}
 \label{kolo12}
 \xymatrix@R=4mm@C=5mm{0\ar[r]&W_1\ar[r]\ar[d]^-{\wr} &W\ar[r]\ar[d] &W_2\ar[r]\ar[d]^-{\wr} &0\\
0\ar[r] &W'_1\ar[r] &W'\ar[r] &W'_2\ar[r] &0},
\end{equation}
which is commutative in ${\rm LH}(C_K)$ where the left and right vertical arrows are isomorphisms
since $W_1,W_2$ are of compact type and we have assumed that $K$ is spherically
complete or $W_1,W_2$ are separable (otherwise we could have $W'_i=0$...).  Hence the rows in diagram \eqref{kolo12} are exact. 
It follows that the map $W\to W'$ is an isomorphism
${\rm LH}(C_K)$ hence $W$ is of compact type.
\end{proof}

\begin{remark}\label{extlh4}
The sequences in Propositions~\ref{extlh3} and~\ref{extlh5} have no reason to split, 
even if $K$ is a finite extension of $\Q_p$.
For example, if ${\rm LA}(\Z_p)$ and ${\rm LC}(\Z_p)$
are the spaces of locally analytic (resp.~locally constant) functions
$\Z_p\to K$, the derivation $\frac{d}{dx}:{\rm LA}(\Z_p)\to{\rm LA}(\Z_p)$ is surjective
but Kohlhaase  showed  in \cite[Cor. 4.3]{Kol} that it does not have a continuous section (no continuous 
primitivation), i.e.,
the sequence of spaces of compact type $$0\to {\rm LC}(\Z_p)\to
{\rm LA}(\Z_p)\to{\rm LA}(\Z_p)\to 0$$ does not split.

By duality, this says that if ${\cal R}^+$ denotes the ring of analytic functions
on the open unit ball (dual to ${\rm LA}(\Z_p)$ via Amice transform), 
and $t=\log(1+T)$, then the sequence 
of fr\'echets $$0\to t{\cal R}^+\to{\cal R}^+\to {\cal R}^+/t\to 0$$ does not split.
\end{remark}


\begin{thebibliography}{99}

\bibitem{And} Y.\,Andr\'e, 
{\em Slope filtrations}.
 Confluentes Math.\,{\bf 1} (2009), 1--85.

\bibitem{ALB}  J.~Anch\"utz, A-C.~Le Bras, 
{\em A Fourier Transform for Banach-Colmez spaces.}  	
\url{arXiv:2111.11116 [math.AG]}, to appear in JEMS.

 \bibitem{BE2} A.\,Beilinson, 
{\em On the crystalline period map}. Preprint,
\url{arXiv:1111.3316v4}. This is an extended version of  Camb. J. Math. {\bf 1} (2013), 1--51. 

\bibitem{Be2} 
{L.\,Berger},
{\em Construction de $(\varphi,\Gamma)$-modules : repr\'esentations $p$-adiques et $B$-paires.}
Algebra Number Theory~{\bf 2} (2008),  91--120. 

\bibitem{Be4}
{L.\,Berger},
{\em Presque $\C_p$-repr\'esentations et $(\varphi,\Gamma)$-modules.}
J. Inst. Math. Jussieu~{\bf 8} (2009), 653--668. 


 \bibitem{BH} B.\,Bhatt, D.\,Hansen, 
{\em The six functors for Zariski-constructible sheaves in rigid geometry}.
\url{arXiv:2101.09759 [math.AG]}.

 \bibitem{BMS1} B.\,Bhatt, M.\,Morrow, P.\,Scholze, 
{\em Integral $p$-adic Hodge Theory}.  
Publ. IHES~{\bf 128}  (2018), 219--397.

\bibitem{BM97} E.~Bierstone, P.~D.~Milman,  
{\em Canonical desingularization in characteristic zero by blowing up the maximum strata of a local invariant.}
Invent. Math.~{\bf 128} (1997), 207--302.

\bibitem{BKV} F.~Binda, H.~Kato, A.~Vezzani, 
{\em On the $p$-adic weight-monodromy conjecture for complete intersections in toric varieties.}  
\url{arXiv:2207.00369 [math.AG]}, preprint 2022.

 \bibitem{CB} { P.\,Colmez}, 
{\em Espaces de Banach de dimension finie}.
 J. Inst. Math. Jussieu~{\bf 1} (2002), 331--439.

\bibitem{CF} { P.\,Colmez}, 
{\em Espaces vectoriels de dimension finie et repr\'esentations de de Rham}.
Ast\'erisque~{\bf 319} (2008), 117--186.

\bibitem{dense} P.\,Colmez,
{\em Une construction de $\bdr$}.
 Rend. Sem. Mat. Univ. Padova~{\bf 128} (2012), 109--130.

\bibitem{CDN3}  P.\,Colmez, G.\,Dospinescu, W.\,Nizio\l, 
{\em Cohomology of $p$-adic Stein spaces}.  
Invent. Math. {\bf 219} (2020), 873--985.

\bibitem{CGN}
{P.\,Colmez, S.\,Gilles, W.\,Nizio{\l}},
{\em Arithmetic duality  for the pro-\'etale cohomology of $p$-adic analytic curves.}
\url{arXiv:2308.07712 [math.NT]}, preprint 2023.

\bibitem{CCF} P.\,Colmez, J.-M.\,Fontaine, 
{\em Construction des repr\'esentations $p$-adiques semi-stables}. 
Invent. Math. {\bf 140} (2000), 1--43.

\bibitem{CN1} P.\,Colmez, W.\,Nizio\l, 
{\em Syntomic complexes and $p$-adic nearby cycles}. 
Invent. Math. {\bf 208} (2017), 1--108.

 \bibitem{CN3} P.\,Colmez, W.\,Nizio\l, 
{\em On $p$-adic comparison theorems for rigid analytic spaces, I}. 
M\"unster J. Math.~{\bf 13} (2020) (Special Issue: In honor of Ch. Deninger), 445--507. 

  \bibitem{CN4} P.\,Colmez, W.\,Nizio\l, 
{\em On the cohomology of $p$-adic analytic spaces, I: The basic comparison theorem}.  
J. Algebraic Geometry~{\bf 34} (2025), 1--108.

\bibitem{OV} S.~Dierolf,  W.~Roelcke,  
{\em On the three-space-problem for topological vector spaces}.
 Collect. Math. 32, p. 13--35 (1981).

\bibitem{Fa94} {G.\,Faltings},
{\em Almost \'etale extensions}.
 Ast\'erisque~{\bf 279} (2002), 185--270.

\bibitem{faltings}
{G.\,Faltings},
{\em Mumford-Stabilit\"at in der algebraischen Geometrie.}
In {\it Proceedings of the International Congress of Mathematicians (Z\"urich, 1994), Vol.\,1},
(1995), 648--655, Birkh\"auser.

 \bibitem{FF} L.\,Fargues, J.-M.\, Fontaine, 
{\em Courbes et fibr\'es vectoriels en th\'eorie de Hodge $p$-adique}.
 Ast\'erisque~{\bf 406} (2018), 51--382. 

\bibitem{Fo78} {J.-M.\,Fontaine},
{\em Modules galoisiens, modules filtr\'es et anneaux de Barsotti-Tate}.
 Ast\'erisque~{\bf 65} (1979), 3--80.

\bibitem{Fo82} {J.-M.\,Fontaine},
{\em Sur certains types de repr\'esentations $p$-adiques
du groupe de Galois d'un corps local; construction d'un anneau de Barsotti-Tate}.
 Ann. of Math.\,{\bf 115} (1982), 529--577.

\bibitem{Fo83} {J.-M.\,Fontaine},
{\em Cohomologie de de Rham, cohomologie cristalline et repr\'esentations $p$-adiques}.
{\it Algebraic geometry (Tokyo/Kyoto, 1982)}, 86--108,
Lecture Notes in Math.\,{\bf 1016}, Springer, 1983.

\bibitem{Fo94b} {J.-M.\,Fontaine}, 
{\em Le corps des p\'eriodes $p$-adiques.}
Ast\'erisque~{\bf 223} (1994), 59--111.

\bibitem{Fo94a} {J.-M.\,Fontaine}, 
{\em Repr\'esentations $p$-adiques semi-stable}.
 Ast\'erisque~{\bf 223} (1994), 113--184.

\bibitem{Fo00} {J.-M.\,Fontaine},
{\em Arithm\'etique des repr\'esentations galoisiennes $p$-adiques.}
Ast\'erisque~{\bf 295} (2004), 1--115.

\bibitem{Fo-Cp} J.-M.\,Fontaine, 
{\em Presque $C_p$-repr\'esentations.} 
Kazuya Kato's fiftieth birthday. Doc. Math. 2003, Extra Vol., 285--385.

  \bibitem{Fonew} J.-M.\,Fontaine, 
{\em Almost $C_p$ Galois representations and vector bundles.}     
Tunisian J. Math.\,{\bf 2} (2020), 667--732.

  \bibitem{FQ} L.\,Fourquaux, 
{\em Applications $\Q_p$-lin\'eaires, continues et Galois-\'equivariantes de $C_p$ dans lui-m\^eme}. 
J. Number Theory~{\bf 129} (2009), 1246--1255. 

\bibitem{SG} S.\,Gilles, 
{\em Morphismes de p\'eriodes et cohomologie syntomique}.  	
 Algebra Number Theory~{\bf 17} (2023), 603--666.

\bibitem{GKF} E.\,Grosse-Kl\"onne, 
{\em Finiteness of de Rham cohomology in rigid analysis}. 
Duke Math. J.\,{\bf 113} (2002), 57--91.

\bibitem{GKdR} E.\,Grosse-Kl\"onne, 
{\em De Rham cohomology of rigid spaces.}
Math. Z.~{\bf 247} (2004), 223--240.

\bibitem{KT} K.\,Kedlaya, M.\,Temkin, 
{\em Endomorphisms of power series fields and residue fields of Fargues-Fontaine curves}.
Proc. AMS~{\bf146} (2018), 489--495.

\bibitem{Kol} J.~Kohlhaase, 
{\em The cohomology of locally analytic representations.}  
J. Reine Angew. Math.~{\bf 651} (2011),  187--240.

  \bibitem{lebras}A.-C.\,Le Bras,  
{\em Espaces de Banach-Colmez et faisceaux coh\'erents sur la courbe de Fargues-Fontaine.} 
Duke Math. J.\,{\bf 167} (2018), 3455--3532.

\bibitem{Ni08} { W. Nizio\l},
{\em Semistable Conjecture via K-theory}. 
Duke Math. J.\,{\bf 141} (2008), 151--178.

\bibitem{Ni18} {W. Nizio\l}, 
{\em On uniqueness of $p$-adic period morphisms II}.
Compos. Math.\,{\bf 156} (2020), 1915--1964.

\bibitem{plut} J.\,Pl\^{u}t, 
{\em Espaces de Banach analytiques $p$-adiques et espaces de Banach-Colmez}. 
Th\`ese Orsay (2009).

\bibitem{plut1} J.\,Pl\^{u}t, 
{\em Slope filtration on Banach-Colmez spaces}.  	
\url{arXiv:1610.09822 [math.NT]}, preprint 2016.


\bibitem{Sch0} P.\,Scholze, 
{\em Perfectoid spaces}.
  Publ. IHES~{\bf 116} (2012), 245--313.

 \bibitem{Sch} {P.\,Scholze}, 
{\em $p$-adic Hodge theory for rigid-analytic varieties}.
 Forum Math. Pi~{\bf 1} (2013), e1, 77 pp. 

\bibitem{Schn} J.-P. Schneiders, 
{\em Quasi-abelian categories and sheaves.}  
M\'em. SMF~{\bf 76}, 1999.

\bibitem{sen} {S.\,Sen},
{\em Continuous cohomology and $p$-adic Galois representations}. 
Invent. math.\,{\bf 62} (1980/81), 89--116. 

\bibitem{SP} The Stacks Project Authors, 
{\em Stack Project}. http://stacks.math.columbia.edu, 2020. 

\bibitem{Shao} X.~Shao,  
{\em On the cohomology of $p$-adic analytic spaces}.
 PhD thesis, Sorbonne University, 2025.

\bibitem{Tate} {J.\,Tate},
{\em $p$-divisible groups}.
{\it Proc. Conf. Local Fields (Driebergen, 1966)}, 158--183, Springer 1967.

 \bibitem{Ts} T.\,Tsuji, 
{\em $p$-adic \'etale cohomology and crystalline cohomology in the semi-stable reduction case}.
 Invent. Math.\,{\bf 137} (1999), 233--411.
 \end{thebibliography}
\end{document}